\documentclass[9pt]{amsart}

\usepackage{graphicx}
\usepackage{latexsym}
\usepackage{amssymb, latexsym}
\usepackage{amsmath}
\usepackage{amsxtra}
\usepackage{amsthm}
\usepackage{url}
\usepackage{verbatim}

\usepackage{tikz}
\usetikzlibrary{shapes.misc, positioning, arrows, arrows.meta, calc}
\usepackage[all,cmtip,color]{xy}
\usepackage{float}

\usepackage[textsize=footnotesize,color=orange!40, bordercolor=white]{todonotes}
\usepackage[urlcolor=blue]{hyperref}
\usepackage{xcolor}
\definecolor{dblue}{rgb}{0,0,0.70}
\hypersetup{
	unicode=true,
	colorlinks=true,
	citecolor=dblue,
	linkcolor=dblue,
	anchorcolor=dblue
}

\makeatletter
\@namedef{subjclassname@2020}{\textup{2020} Mathematics Subject Classification}
\makeatother

\DeclareRobustCommand{\SkipTocEntry}[5]{}

\usepackage
[a4paper, margin=1.5in]{geometry}

\newtheorem{theorem}{Theorem}[section]
\newtheorem{corollary}[theorem]{Corollary}

\newtheorem{lemma}[theorem]{Lemma}

\newtheorem{proposition}[theorem]{Proposition}
\newtheorem{remark}[theorem]{Remark}

\theoremstyle{definition}
\newtheorem{definition}[theorem]{Definition}

\newtheorem{question}[theorem]{Question}
\newtheorem{problem}[theorem]{Problem}

\newcommand{\Ult}{\mathop{\rm Ult}}

\newcommand{\Union}{\bigcup}
\newcommand{\union}{\cup}

\newcommand{\smallleq}{\mathrel{\mathchoice{\raise2pt\hbox{$\scriptstyle\leq$}}{\raise1pt\hbox{$\scriptstyle\leq$}}{\raise1pt\hbox{$\scriptscriptstyle\leq$}}{\scriptscriptstyle\leq}}}
\newcommand{\smalllt}{\mathrel{\mathchoice{\raise2pt\hbox{$\scriptstyle<$}}{\raise1pt\hbox{$\scriptstyle<$}}{\raise0pt\hbox{$\scriptscriptstyle<$}}{\scriptscriptstyle<}}}
\newcommand{\Add}{\mathop{\rm Add}}

\newcommand{\GCH}{{\rm GCH}}

\newcommand{\ZFC}{{\rm ZFC}}
\newcommand{\KP}{{\rm KP}}
\newcommand{\VV}{\mathbf{V}}

\newcommand{\p}{\mathbb{P}}

\newcommand{\la}{\langle}
\newcommand{\ra}{\rangle}

\newcommand{\tail}{\text{tail}}

\newcommand{\forces}{\Vdash}

\newcommand{\Los}{\L o\'s}

\newcommand{\GBC}{{\rm GBC}}

\newcommand{\KM}{{\rm KM}}

\newcommand{\crit}{\text{crit}}

\newcommand{\ran}{{\rm ran}}

\newcommand{\CC}{{\rm CC}}

\newcommand{\Tr}{{\rm Tr}}
\newcommand{\Ord}{{\rm Ord}}

\newcommand{\g}{{\mathsf{G}}}
\newcommand{\rg}{{\mathsf{RG}}}
\newcommand{\frg}{{\mathsf{frG}}}
\newcommand{\Ra}{{\mathsf{Ramsey}}}

\title{Between Ramsey and measurable cardinals}
\author{Victoria Gitman}
\address{The City University of New York, CUNY Graduate Center, Mathematics Program, 365 Fifth Avenue, New York, NY 10016}
\email{vgitman@gmail.com}
\urladdr{https://victoriagitman.github.io/}
\author{Philipp Schlicht}
\address{School of Mathematics, University of Bristol, Fry Building, Woodland Road, Bristol, BS8 1UG, UK}
\email{philipp.schlicht@bristol.ac.uk}


\thanks{We are grateful to Philip Welch for his kind permission to use an argument of his in the proof of Lemma \ref{lem:Sigma_nDefinableSubstructure}.}
\thanks{This research was funded in whole or in part by EPSRC grant number EP/V009001/1 of the second-listed author.  For the purpose of open access, the authors have applied a ‘Creative Commons Attribution' (CC BY) public copyright licence to any Author Accepted Manuscript (AAM) version arising from this submission.}

\begin{document}
\maketitle

\begin{abstract}
We study several intertwined hierarchies between $\kappa$-Ramsey cardinals and measurable cardinals to illuminate 
the structure of the large cardinal hierarchy in this region. 
In particular, we study baby versions of measurability
introduced by Bovykin and McKenzie and some variants by locating these notions in the large cardinal hierarchy and providing characterisations via filter games.
As an application, we determine the theory of the universe up to a measurable cardinal.
\end{abstract}

\tableofcontents

\section{Introduction}

Measurable cardinals and most stronger large cardinals are defined by the existence of elementary embeddings $j:V\to M$ from the universe $V$ into an inner model $M$ with that cardinal as the critical point. Stronger large cardinal axioms impose additional assumptions on the target model $M$ that allow it to capture more and more sets from $V$. Weakly compact and many other smaller large cardinals are defined via combinatorial properties, often involving existence of large homogeneous sets for colorings. But almost all of them also have elementary embedding characterizations that, like the larger large cardinals, follow prescribed patterns as consistency strength grows. These smaller large cardinals $\kappa$ are characterized by elementary embeddings of models $\la M,\in\ra\models\ZFC^-$ of size $\kappa$, most often transitive, with critical point $\kappa$, and usually, but not always, with well-founded targets. $\ZFC^-$ is the theory $\ZFC$ with the powerset axiom removed, the collection scheme in place of the replacement scheme, and the version of the axiom of choice which states that every set can be well-ordered.\footnote{Without the power set axiom, collection and replacement schemes are not equivalent and neither are the various forms of the axiom of choice. See \cite{zarach:unions_of_zfminus_models} and \cite{zfcminus:gitmanhamkinsjohnstone}.} In many cases, the existence of these embeddings can be equivalently expressed in terms of the existence of certain filters on the subsets of $\kappa$ contained in $M$.

Suppose $\la M,\in\ra\models\ZFC^-$ is a transitive model of size $\kappa$ with $\kappa\in M$ (we will give precise definitions and drop the transitivity requirement in the next section). We call such structures \emph{weak $\kappa$-models}. We will call a filter $U$ on the subsets of $\kappa$ of a weak $\kappa$-model $M$ an \emph{$M$-ultrafilter} if the structure $\la M,\in,U\ra$, together with a predicate for $U$, satisfies that $U$ is a uniform, normal ultrafilter on $\kappa$. What this says is that $U$ is an ultrafilter on the subsets of $\kappa$ that live in $M$ and if a $\kappa$-length sequence of elements of $U$ is an element of $M$, then its diagonal intersection is in $U$. \Los' theorem holds for ultrapowers by an $M$-ultrafilter, but the ultrapower need not be well-founded. Note that the filter $U$ may not be countably complete for sequences outside of $M$. The $M$-ultrafilter $U$ is, in most interesting cases, external to $M$ and even though $\la M,\in\ra\models\ZFC^-$, separation and replacement can fail completely in the structure $\la M,\in,U\ra$ once we let $M$ know about $U$. In the same way that making the inner model $M$ closer to $V$ in the characterizations of larger large cardinals increases strength, for the smaller large cardinals, we increase strength by making $M$ be more compatible with $U$, which amounts to $M$ being more correct about the properties of $U$ or to having more of the $\ZFC^-$ axioms in the structure $\la M,\in,U\ra$.

Let's look at some examples. A cardinal $\kappa$ is weakly compact if and only if it is inaccessible and every $A\subseteq\kappa$ is an element of a weak $\kappa$-model $M$ for which there is an $M$-ultrafilter with a well-founded ultrapower. It turns out that it is equivalent to assume the a priori stronger assertion that every $A\subseteq \kappa$ is an element of a weak $\kappa$-model for which there is a (externally) countably complete $M$-ultrafilter $U$. In particular, the ultrapower is well-founded. A cardinal $\kappa$ is $1$-iterable if every $A\subseteq\kappa$ is an element of a weak $\kappa$-model $M$ for which there is an $M$-ultrafilter $U$ such that $\la M,\in, U\ra$ satisfies $\Sigma_0$-separation and the ultrapower is well-founded \cite[Definition 2.11]{gitman:welch}. The $1$-iterable cardinals are stronger than ineffable cardinals, and therefore much stronger than weakly compact cardinals. They are however still compatible with $L$. Thus, the additional requirement that the structure $\la M,\in, U\ra$ satisfies $\Sigma_0$-separation pushes up the consistency strength. A cardinal $\kappa$ is Ramsey if and only if every $A\subseteq\kappa$ is an element of a weak $\kappa$-model $M$ for which there is an $M$-ultrafilter $U$ such that $\la M,\in,U\ra$ satisfies $\Sigma_0$-separation and which is (externally) countably complete \cite[Theorem 3]{mitchell:ramsey}. Ramsey cardinals sit much higher in the hierarchy than $1$-iterable cardinals. Thus, in particular, in the presence of the requirement of $\Sigma_0$-separation for the structure $\la M,\in, U\ra$ for the $M$-ultrafilter $U$, having $U$ be countably complete is much stronger than having it just produce a well-founded ultrapower.

In this article, motivated by the work of Bovykin and McKenzie \cite{BovykinMcKenzie:RamseyLikeNFUM}, we consider a hierarchy of large cardinal notions characterized by the existence, for weak $\kappa$-models $M$, of $M$-ultrafilters $U$ such that the structure $\la M,\in,U\ra$ satisfies fragments up to full $\ZFC^-$. Following Bovykin and McKenzie, we call such cardinals \emph{$n$-baby measurable} where the fragment is $\ZFC^-_n$ (the separation and collection schemes are restricted to $\Sigma_n$-formulas).
Baby measurable cardinals and some variants were introduced by Bovykin and McKenzie in \cite[Definition 4.2]{BovykinMcKenzie:RamseyLikeNFUM} and used to obtain the following application to the theory ${\rm NFUM}$.
This theory is a natural strengthening of ${\rm NFU}$ due to Holmes (see \cite[Section 2.2]{BovykinMcKenzie:RamseyLikeNFUM}) that aims to facilitate mathematics in ${\rm NFU}$, the latter being a variant of Quine's \emph{New Foundations} with Urelements introduced by Jensen.

\begin{theorem}[Bovykin, McKenzie {\cite[Section 4]{BovykinMcKenzie:RamseyLikeNFUM}}]
\label{BovykinMcKenzieMain}
The following theories are equiconsistent.\footnote{Note that their notion of $n$-baby measurable cardinal in \cite[Definition 4.2]{BovykinMcKenzie:RamseyLikeNFUM} is slightly different from ours in Definition \ref{def baby meas} \eqref{def baby meas main}, but the theorem holds for our notion as well by Theorem \ref{Bovykin equicon}.}
\begin{enumerate}
\item $\ZFC$ together with the scheme consisting of the assertions for every $n\in \omega$:
\begin{center}
``There exists an $n$-baby measurable cardinal $\kappa$ such that $V_\kappa\prec_{\Sigma_n} V$.''
\end{center}
\item ${\rm NFUM}$
\end{enumerate}
\end{theorem}

We collect facts about ultrafilters and large cardinals in Section \ref{section: preliminaries} and study structures consisting of models with ultrafilters satisfying fragments of $\ZFC^-$ in Section \ref{section: amen coll}.
The new large cardinals notions are introduced in Section \ref{section: babym}.
We show in Section \ref{sec:hierarchy} that increasing levels of collection and separation for the structures $\la M,\in,U\ra$ increase the large cardinal strength. For example, adding even $\Sigma_0$-replacement to the characterization of $1$-iterable cardinals pushes consistency strength well beyond a Ramsey cardinal, and hence beyond $L$.
We then provide a fine analysis of the resulting hierarchies.
Surprisingly, the closure of the models does not play a role when working solely with fragments of collection.
However, closure conditions do induce a strict hierarchy for the setting of $\ZFC^-_n$ with both the collection and separation schemes restricted to $\Sigma_n$-formulas.
This is shown in Section \ref{section: games} with the help of filter games resembling those introduced by Holy and the second-listed author \cite[Section 5]{HolySchlicht:HierarchyRamseyLikeCardinals}.
We thus arrive at similar patterns of large cardinal notions as the one around strongly Ramsey and $\alpha$-Ramsey cardinals.

\tikzset{>={Stealth[round, length=1.5mm,width=1.1mm]}}
     
\newcommand{\zz}{2.8}
\begin{figure}[h]\label{Figure.Theories}
  \begin{tikzpicture}[theory/.style={draw,rounded rectangle,scale=.75,minimum height=6.5mm},scale=.35]
      \draw (0:0) node[theory] (1) {$1$-iterable}
           ++(90:\zz) node[theory] (2) {$\alpha$-Ramsey}
           ++(90:\zz) node[theory] (3) {strongly Ramsey}
           ++(90:\zz) node[theory] (4) {$\kappa$-Ramsey};
     \draw (0:20) node[theory, blue!80!black] (1a) {weakly baby measurable}
           ++(90:\zz) node[theory, blue!80!black] (2a) {$\alpha$-baby measurable}
           ++(90:\zz) node[theory, blue!80!black] (3a) {baby measurable}
           ++(90:\zz) node[theory, blue!80!black] (4a) {$\kappa$-baby measurable};
     \draw (0:10) node[theory, blue!80!black] (1b) {weakly $n$-baby measurable}
           ++(90:\zz) node[theory, blue!80!black] (2b) {$(\alpha,n)$-baby measurable}
           ++(90:\zz) node[theory, blue!80!black] (3b) {$n$-baby measurable}
           ++(90:\zz) node[theory, blue!80!black] (4b) {$(\kappa,n)$-baby measurable};
     \draw[<-]
                        (1) edge (2)
                        (3) edge (4);
     \draw[<-, dotted]
                        (2) edge (3); 
     \draw[<-, blue!80!black]
                        (1a) edge (2a)
                        (3a) edge (4a);
     \draw[<-, dotted]
     			     (2a) edge (3a); 
     \draw[<-, blue!80!black]
                        (1b) edge (2b)
                        (3b) edge (4b);
     \draw[<-, dotted]
                        (2b) edge (3b); 
  \end{tikzpicture}
  \caption{Patterns in analogy with $\alpha$-Ramsey cardinals.
Solid arrows denote direct implications, dotted arrows implications in consistency strength.}
\end{figure}
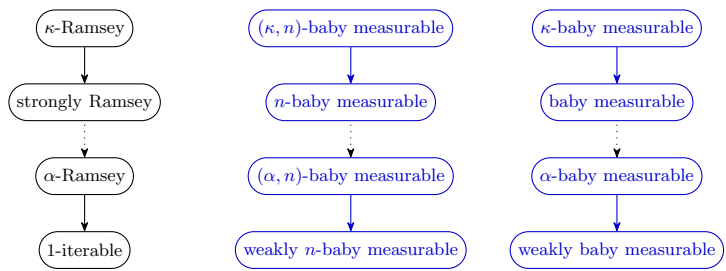

We show that some of these new large cardinal notions are robust under forcing in Section \ref{section: indestructibility}.
Finally, we will see that these notions are naturally connected to models of Kelley-Morse set theory in which $\Ord$ is measurable and help us to understand the structure of these second-order models in Section \ref{section: second-order}.
In particular, we apply the above methods to show that the theory of structures of the form $V_\kappa^M$, where $M$ is a model of $\ZFC^-$ or $\ZFC$ where $\kappa$ is measurable, is axiomatizable by the existence of large cardinals similar to the ones above.
These results shed light on the large cardinal hierarchy below measurable cardinals and provide a blueprint how to approximate a large cardinal notion from below by a natural hierarchy.

\section{Preliminaries}
\label{section: preliminaries}

Given a set $X$, when we say that $X$ is a model of set theory, we will tacitly assume that the membership relation is the actual membership $\in$ restricted to $X$, so that $X$ is really the model $\la X,\in\ra$.
\begin{definition}
Suppose that $\kappa$ is an inaccessible cardinal.
\begin{itemize}
\item A \emph{weak $\kappa$-model} is a transitive set $M\models\ZFC^-$ of size $\kappa$ with $V_\kappa\in M$.
\item A \emph{$\kappa$-model} is a weak $\kappa$-model such that $M^{{<}\kappa}\subseteq M$.
\item A \emph{basic weak $\kappa$-model} is a (not necessarily transitive) set $M\models\ZFC^-$ of size $\kappa$ such that $M\prec_{\Sigma_0}V$ and $V_\kappa\union \{V_\kappa\}\subseteq M$.
\item A \emph{basic $\kappa$-model} is a basic weak $\kappa$-model such that $M^{{<}\kappa}\subseteq M$.
\end{itemize}
In each case, we will say that $M$ is \emph{simple} if $\kappa$ is the largest cardinal in $M$.\footnote{Note that these definitions differ slightly from those appearing in earlier literature. For instance, in the definition of weak $\kappa$-model, it is not usually assumed that $\kappa$ is inaccessible and $V_\kappa\in M$.}
\end{definition}
\noindent Note that if $\kappa$ is a basic weak $\kappa$-model, then $\kappa\in M$. To see this, let $\alpha\in M$ be such that $$M\models``\alpha=V_\kappa\cap \Ord",$$ and then observe that by $\Sigma_0$-elementarity, $\alpha=V_\kappa\cap\Ord$ holds true in $V$, which means that $\alpha=\kappa$.

Our canonical examples of basic models will be elementary substructures of size $\kappa$ of some large $H_\theta$ for a regular $\theta$.\footnote{$H_\theta$ for a cardinal $\theta$ is the collection of all sets whose transitive closure has size less than $\theta$ and given that $\theta$ is regular, we have $H_\theta\models\ZFC^-$.} For simple models, the notions of basic weak $\kappa$-models and weak $\kappa$-models coincide.

\begin{lemma}
\label{simple basic transitive}
If $M$ is a simple basic weak $\kappa$-model, then $M$ is a weak $\kappa$-model.
\end{lemma}
\begin{proof}
We just need to show that $M$ is transitive. Fix $a\in M$. Since $M$ is simple, it thinks that there is a bijection $f:\kappa\to a$. By $\Sigma_0$-elementarity, $f$ is really a bijection between $\kappa$ and $a$. Thus, since $\kappa\subseteq M$, $a\subseteq M$.
\end{proof}

Note that even a simple weak $\kappa$-model might not be a $\kappa$-model. For instance, its height may have countable cofinality.

Given a basic weak $\kappa$-model $M$, we will let $P^M(\kappa)$ denote the collection of all the subsets of $\kappa$ that are elements of $M$. Note that in most cases, $P^M(\kappa)$ will be a class, but not an element of $M$.

\begin{lemma}\label{prop:internalKappaModel}
Suppose that $M$ is a basic weak $\kappa$-model. If $\bar M\in M$ and $M$ thinks that $\bar M$ is a basic $\kappa$-model, then $\bar M$ is a basic $\kappa$-model.
\end{lemma}
\begin{proof}
The model $\bar M$ is a basic weak $\kappa$-model by the $\Sigma_0$-elementarity of $M$, so it remains to check closure. Let $f:\bar M\to V_\kappa$ be a bijection in $M$, which really must be a bijection by $\Sigma_0$-elementarity. Fix a sequence $\vec a=\la a_\xi\mid\xi<\beta\ra$ such that $\beta<\kappa$ and $a_\xi\in \bar M$ for every $\xi<\beta$. Let $b_\xi=f(a_\xi)$ and let $\vec b=\la b_\xi\mid\xi<\beta\ra$. The sequence $\vec b\in V_\kappa$, and hence $\vec b\in M$. Thus, $f[\vec b]=\vec a\in M$ by $\Sigma_0$-elementarity, and hence $\vec a\in\bar M$ since $M$ thinks it is a basic $\kappa$-model.
\end{proof}

\begin{definition}
Suppose that $M$ is a basic weak $\kappa$-model.
\begin{itemize}
\item We will say that $U\subseteq P^M(\kappa)$ is an \emph{$M$-ultrafilter} if the structure
\begin{center}
$\la M,\in,U\ra\models ``U$ is a uniform\footnote{Recall that a filter on a cardinal $\kappa$ is \emph{uniform} if it contains all tail sets $(\alpha,\kappa)$ for $\alpha<\kappa$.} normal ultrafilter on $\kappa.$"
\end{center}
\item An $M$-ultrafilter $U$ is $\emph{good}$ if the ultrapower of $M$ by $U$ is well-founded.
\item An $M$-ultrafilter $U$ is \emph{countably complete} if for every sequence $\la A_n\mid n<\omega\ra$ with $A_n\in U$ (but the sequence itself not necessarily in $M$), $\bigcap_{n<\omega}A_n\neq\emptyset$.
\end{itemize}
\end{definition}
If $M$ is a basic weak $\kappa$-model and $j:M\to N$ is an elementary embedding with $\crit(j)=\kappa$, then $$U=\{A\subseteq\kappa\mid A\in M\text{ and }\kappa\in j(A)\}$$ is the $M$-ultrafilter \emph{derived} from $j$, and if $N$ is well-founded, then $U$ is good. If $j$ happens to be the ultrapower by an $M$-ultrafilter $U$, then the $M$-ultrafilter derived from $j$ is precisely $U$. Standard arguments using \Los' theorem show that if $M$ is a basic weak $\kappa$-model and $U$ is a countably complete $M$-ultrafilter, then $U$ is good.
In particular, if a basic weak $\kappa$-model $M$ is closed under $\omega$-sequences ($M^\omega\subseteq M$), e.g. if $M$ is a $\kappa$-model, then every $M$-ultrafilter must be countably complete, and hence good. Thus, whenever $M$ is a basic weak $\kappa$-model such that \hbox{$M^\omega\subseteq M$}, every $M$-ultrafilter is automatically good. Even when the ultrapower $N$ of $M$ by $U$ is ill-founded, by our assumptions on $U$, $V_\kappa\union \{V_\kappa\}\subseteq N$ and $V_\kappa^M=V_\kappa^N$.\footnote{We are assuming here that we have collapsed the well-founded part of $N$.}

\begin{definition}
Suppose that $M$ is a basic weak $\kappa$-model. An $M$-ultrafilter $U$ is \emph{weakly amenable} if for all $A\in M$ with $|A|^M=\kappa$, $A\cap U\in M$.
\end{definition}

Note that for simple models $M$, a weakly amenable $M$-ultrafilter $U$ is fully amenable in the sense that for every $A\in M$, we have $A\cap U\in M$. A weakly amenable $M$-ultrafilter can be iterated to carry out the iterated ultrapower construction for any ordinal length. Recall that in the case of say a measure $U$ on $\kappa$ (namely if $\kappa$ is a measurable cardinal), we can define an $\Ord$-length system of iterated ultrapowers of $U=U_0$. We let $j_{01}:V=M_0\to M_1$ be the ultrapower by $U_0$ and use $j_{01}$ to obtain the ultrafilter for the next stage of the iteration by defining $U_1=j_{01}(U_0)$. This gives us a general procedure for obtaining the next stage ultrafilter at the successor stages of the iteration and at limits we take a direct limit of the system of embeddings obtained thus far. By a theorem of Gaifman, all the iterated ultrapowers $M_\xi$ are well-founded \cite[Section II, Theorem 5]{gaifman:ultrapowers}. We cannot apply the same procedure to obtain successor stage ultrafilters with an $M$-ultrafilter $U$ because $U$ is (in most interesting cases) not an element of $M$. However, given weak amenability, we can define, for example, $U_1$, given that $j:M\to N$ is the ultrapower embedding, by $$U_1:=\{A=[f]_U\subseteq j(\kappa)\mid \{\alpha<\kappa\mid f(\alpha)\in U\}\in U\}.$$

Weak amenability has an equivalent characterization in terms of the preservation of subsets of $\kappa$ between the model and its ultrapower. An $M$-ultrafilter $U$ for a basic weak $\kappa$-model $M$ is weakly amenable if and only if $M$ and its ultrapower $N$ have the same subsets of $\kappa$. Note that this holds true regardless of whether $N$ is well-founded or not. It is also the case that if $j:M\to N$ is an elementary embedding with $\crit(j)=\kappa$ such that $M$ and $N$ have the same subsets of $\kappa$, then the $M$-ultrafilter derived from $j$ is weakly amenable.

\begin{lemma}\label{prop:wellFoundedPartOfUltrapowerWA}
If $M$ is a simple basic weak $\kappa$-model, $U$ is a weakly amenable $M$-ultrafilter, and $\langle N,\bar{\in}\rangle$ is the ultrapower of $M$ by $U$, then $M=H_{\kappa^+}^N$.
\end{lemma}
\begin{proof}
Let's assume that we have collapsed the well-founded part of $\langle N,\bar{\in}\rangle$ so that for well-founded sets, $\in=\!\!\bar{\in}$. Clearly $M\subseteq H_{\kappa^+}^N$. So suppose that $a\,\bar{\in}\, H_{\kappa^+}^N$ and assume without loss that $a$ is transitive in $\bar{\in}$. Let $f:a\to\kappa$ be a bijection in $N$, and let $$A=\{(f(a),f(b))\mid (a,b)\in \bar{\in}\}\subseteq \kappa\times \kappa.$$ Then $A\in M$ by weak amenability. First, suppose that $A$ is ill-founded. Then $M$ would know this, and hence it would have a descending $\omega$-sequence witnessing the ill-foundedness. But then $N$ would have the sequence as well, which is impossible. Thus, $A$ is well-founded. We can now argue by recursion on rank that the Mostowski collapse of $A$ in $M$ is the Mostowski collapse of $A$ in $N$, and hence $(a,\bar{\in})=(a,\in)$.
\end{proof}

The simplest characterization of a large cardinal in terms of embeddings on weak $\kappa$-models belongs to weakly compact cardinals.
\begin{theorem}[folklore]
Suppose that $\kappa$ is inaccessible. Then the following are equivalent.
\begin{enumerate}
\item $\kappa$ is weakly compact.
\item Every $A\subseteq \kappa$ is an element of a weak $\kappa$-model $M$ for which there is a good $M$-ultrafilter.
\item Every $A\subseteq \kappa$ is an element of a weak $\kappa$-model $M$ for which there is a countably complete $M$-ultrafilter.
\item Every weak $\kappa$-model has a countably complete $M$-ultrafilter.
\end{enumerate}
\end{theorem}

It is then natural to ask what happens if we addtionally require that the $M$-ultrafilter is weakly amenable. The answer is that we get a much stronger large cardinal notion. Let's start by introducing the following large cardinal notion.
\begin{definition}
A cardinal $\kappa$ is $0$-\emph{iterable} if every $A\subseteq \kappa$ is an element of a weak $\kappa$-model $M$ for which there is a weakly amenable (but not necessarily good) $M$-ultrafilter.
\end{definition}

Recall that a cardinal $\kappa$ is \emph{weakly ineffable} if for every sequence $\la A_\alpha\mid\alpha<\kappa\ra$ with $A_\alpha\subseteq\alpha$, there is a threading set $A\subseteq\kappa$ such that for unboundedly many $\alpha<\kappa$, $A\cap\alpha=A_\alpha$. A cardinal $\kappa$ is \emph{ineffable} if we can find such a set $A$ that is stationary in $\kappa$.

\begin{proposition}
A $0$-iterable cardinal is a weakly ineffable limit of ineffable cardinals.
\end{proposition}
\begin{proof}
Fix a sequence $\vec A=\la A_\alpha\mid\alpha<\kappa\ra$ with $A_\alpha\subseteq\alpha$, and find a weak $\kappa$-model $M$ with $\vec A\in M$ for which there is weakly amenable $M$-ultrafilter $U$. Let $j:M\to N$ be the possibly ill-founded ultrapower by $U$. It is easy to check using \Los' theorem that $A=j(\vec A)(\kappa)$ has the required property. This argument also shows that $\kappa$ is ineffable in the ultrapower $N$. Thus, $\kappa$ is a limit of ineffable cardinals.
\end{proof}

In particular, requiring that the $M$-ultrafilter be weakly amenable, but not even requiring that it is good, gives a large cardinal notion much stronger than weakly compact cardinals.

\begin{definition}[{\cite[Definition 2.11]{gitman:welch}}]
A cardinal $\kappa$ is \emph{1-iterable} if every $A\subseteq\kappa$ is an element of a weak $\kappa$-model $M$ for which there is a good weakly amenable $M$-ultrafilter.
\end{definition}
Iterating a measure on $\kappa$ gives rise to $\Ord$-many well-founded iterated ultrapowers. By a theorem of Kunen, the same holds true for iterating a weakly amenable countably complete $M$-ultrafilter for a weak $\kappa$-model $M$ \cite{kunen:ultrapowers}. However, a weakly amenable $M$-ultrafilter that is not countably complete can give rise to $0\leq\alpha<\omega_1$ or $\Ord$-many well-founded ultrapowers (the later follows just as in the case of a measure from having $\omega_1$-many). In fact, we can define the hierarchy of $\alpha$-iterable cardinals (for $1\leq\alpha\leq\omega_1$), where a cardinal $\kappa$ is \emph{$\alpha$-iterable} if every $A\subseteq\kappa$ is an element of a weak $\kappa$-model $M$ for which there is a weakly amenable $M$-ultrafilter with $\alpha$-many well-founded iterated ultrapowers \cite{gitman:welch}. The 0-iterable cardinals fit naturally at the head of this hierarchy, which is consistent with $L$ for $\alpha<\omega_1$ \cite{gitman:welch}, while it is not difficult to see that $\omega_1$-iterable cardinals imply $0^\#$.
\begin{theorem}[Mitchell {\cite[Theorem 3]{mitchell:ramsey}}]
A cardinal $\kappa$ is Ramsey if and only if every $A\subseteq\kappa$ is an element of a weak $\kappa$-model $M$ for which there is a weakly amenable countably complete $M$-ultrafilter on $\kappa$.
\end{theorem}
The $\omega_1$-iterable cardinals are slightly weaker than Ramsey cardinals, which are $\omega_1$-iterable limits of $\omega_1$-iterable cardinals \cite[Lemma 5.2]{welch:ramsey}.

Note that unlike the situation with weakly compact cardinals, asking that the weakly amenable $M$-ultrafilter be countably complete, and not just good, pushes up the large cardinal strength from 1-iterable cardinals, which are consistent with $L$, to Ramsey cardinals. For weakly compact cardinals $\kappa$, every weak $\kappa$-model, and so, in particular, every $\kappa$-model, has an $M$-ultrafilter. If we ask that every $A\subseteq\kappa$ is an element of a $\kappa$-model for which there is a weakly amenable $M$-ultrafilter (which is automatically good), then we get a large cardinal notion stronger than a Ramsey cardinal. Indeed, assuming that we have a weakly amenable $M$-ultrafilter for every weak $\kappa$-model is inconsistent \cite[Theorem 1.7]{gitman:ramsey}.

\begin{definition} Suppose that $\kappa$ is a cardinal.
\begin{itemize}
\item
{\cite[Definition 1.4]{gitman:ramsey}}
$\kappa$ is \emph{strongly Ramsey} if every $A\subseteq\kappa$ is an element of a $\kappa$-model $M$ for which there is a weakly amenable $M$-ultrafilter.
\item
{\cite[Definition 1.5]{gitman:ramsey}}
$\kappa$ is \emph{super Ramsey} if every $A\subseteq\kappa$ is an element of a $\kappa$-model $M\prec H_{\kappa^+}$ for which there is a weakly amenable $M$-ultrafilter.
\end{itemize}

\end{definition}
Strongly Ramsey cardinals are limits of Ramsey cardinals, super Ramsey cardinals are limits of strongly Ramsey cardinals, and measurable cardinals are limits of super Ramsey cardinals \cite[Theorem 1.7]{gitman:ramsey}. It is natural to ask given these various notions what will happen (1) if we stratify by closure on the weak $\kappa$-model $M$, weakening, for instance, the $\kappa$-model assumption to just countable closure on $M$, and (2) if we ask for elementarity in a large $H_\theta$. The second question does not make sense as stated because a weak $\kappa$-model cannot be elementary in $H_\theta$ for any $\theta>\kappa^+$, but this is precisely where the weakening to basic weak $\kappa$-models comes into play. Combining both of these ideas, Holy and the second-listed author introduced the $\alpha$-Ramsey hierarchy below a measurable cardinal \cite{HolySchlicht:HierarchyRamseyLikeCardinals}.

\begin{definition}
Suppose that $\kappa$ is a cardinal.
 \begin{itemize}
 \item
{\cite[Definition 5.1]{HolySchlicht:HierarchyRamseyLikeCardinals}}
$\kappa$ is $\alpha$-Ramsey for a regular \hbox{$\omega\leq\alpha\leq\kappa$} if for every $A\subseteq\kappa$ and arbitrarily large cardinals $\theta$, there is a ${<}\alpha$-closed basic weak $\kappa$-model $M\prec H_\theta$ with $A\in M$ for which there is a good weakly amenable $M$-ultrafilter.

\item
{\cite[Definition 5.7]{HolySchlicht:HierarchyRamseyLikeCardinals}}
$\kappa$ is \emph{faintly $\omega$-Ramsey}\footnote{This property is called the \emph{$\omega$-filter property} in {\cite{HolySchlicht:HierarchyRamseyLikeCardinals}}.} if in the definition of an $\omega$-Ramsey cardinal, the $M$-ultrafilter $U$ is not required to be good.
 \end{itemize}
\end{definition}

Note that in the definition of $\alpha$-Ramsey cardinals, only the $\omega$-Ramsey cardinals need the extra assumption that the $M$-ultrafilter is good; for the others, the closure on the basic $\kappa$-model already implies it. The $\alpha$-Ramsey cardinals have a natural game theoretic characterization \cite[Theorem 5.6]{HolySchlicht:HierarchyRamseyLikeCardinals} arising out of the question of whether given a weak $\kappa$-model $M$, an $M$-ultrafilter $U$, and another weak $\kappa$-model $N$ extending $M$, we can always find an $N$-ultrafilter extending $U$. The first-listed author showed that this extension property is inconsistent \cite[Proposition 2.13]{HolySchlicht:HierarchyRamseyLikeCardinals}, but the above characterization can be understood as a game theoretic variant of this property.

Consider the following game $\mathsf{Ramsey}\mathsf G^\theta_\alpha(\kappa)$, for regular cardinals $\omega\leq\alpha \leq\kappa$ and $\theta>\kappa$, of perfect information played by two players the challenger and the judge. The challenger starts the game and plays a basic $\kappa$-model $M_0\prec H_\theta$. The judge responds by playing an $M_0$-ultrafilter $U_0$. In the next step, the challenger plays a basic $\kappa$-model $M_1\prec H_\theta$ such that $\la M_0,U_0\ra\in M_1$, and the judge responds with an $M_1$-ultrafilter $U_1$ extending $U_0$. More generally, at stage $\gamma$, the challenger plays a basic $\kappa$-model $M_\gamma\prec H_\theta$ such that \hbox{$\{\la M_\xi,U_\xi\ra\mid\xi<\gamma\}\in M_\gamma$}, and the judge responds with an $M_\gamma$-ultrafilter $U_\gamma$ extending $\Union_{\xi<\gamma}U_\gamma$. The judge wins the game if she can continue playing for $\alpha$-many steps (and $U=\bigcup_{\xi<\alpha}U_\xi$ is a good $M=\bigcup_{\xi<\alpha}M_\xi$-ultrafilter). Otherwise, the challenger wins. Since $M=\bigcup_{\xi<\alpha}M_\xi$ is a union of $\kappa$-models, it is ${<}\alpha$-closed, and thus, the additional assumption that $U=\bigcup_{\xi<\alpha}U_\xi$ is good is only necessary for $\alpha=\omega$. Let $\mathsf{faintRamseyG}^\theta_\omega(\kappa)$ be an analogous game to $\mathsf {RamseyG}^\theta_\omega(\kappa)$, but where we don't require the judge to ensure that the union $U$ of her plays is good. It is shown in \cite[Lemma 3.3]{HolySchlicht:HierarchyRamseyLikeCardinals} that if either of the players has a winning strategy in the game $\mathsf{Ramsey}\mathsf G^\theta_\alpha(\kappa)$, then the same player has a winning strategy in the game $\mathsf{Ramsey}\mathsf G_\alpha^\rho(\kappa)$ for any other regular cardinal $\rho>\kappa$. An analogous result holds for the game $\mathsf{faintRamseyG}^\theta_\omega(\kappa)$.

\begin{theorem}\label{th:alphaRamseyEquivGame}
Suppose that $\kappa$ is a cardinal.
\begin{itemize}
\item
{\cite[Theorem 5.6]{HolySchlicht:HierarchyRamseyLikeCardinals}}
For $\omega\leq\alpha<\kappa$, $\kappa$ is $\alpha$-Ramsey if and only if the challenger does not have a winning strategy in the game $\mathsf{RamseyG}^\theta_\alpha(\kappa)$ for some/all regular cardinals $\theta>\kappa$.

\item
$\kappa$ is faintly $\omega$-Ramsey if and only if the challenger does not have a winning strategy in the game $\mathsf{faintRamseyG}^\theta_\omega$ for some/all regular cardinals $\theta>\kappa$.
\end{itemize}
\end{theorem}
The second part of the Theorem, although not explicitly proved in \cite{HolySchlicht:HierarchyRamseyLikeCardinals}, is proved analogously to proof of the first part given there.

Nielsen and Welch showed that faintly $\omega$-Ramsey cardinals are precisely the well-known completely ineffable cardinals .
The first author showed that the $\omega$-Ramsey cardinals lie between 1-iterable and 2-iterable cardinals in consistency strength. The result was mentioned in \cite[Section 8]{HolySchlicht:HierarchyRamseyLikeCardinals} without proof.
\begin{theorem}
A $2$-iterable cardinal $\kappa$ is a limit of $\omega$-Ramsey cardinals.
\end{theorem}
\begin{proof}
Suppose that $\kappa$ is $2$-iterable. It follows from the diagram in \cite[Lemma 4.4]{gitman:ramsey} that there is a weak $\kappa$-model $M\models\ZFC$ and a good $M$-ultrafilter $U$ such that for the ultrapower map $j:M \to N$, we have $M=V_{j(\kappa)}^N$. We would like to argue that $\kappa$ is $\omega$-Ramsey in $V_{j(\kappa)}^N(=M)$. Fix a regular $\theta\in M$ and $A\subseteq\kappa$ in $M$. We need to produce a basic weak $\kappa$-model $m\prec H_\theta^M$ in $M$ with $A\in m$ for which there is a good weakly amenable $m$-ultrafilter on $\kappa$ in $M$. Let $A\in M_0\prec H_\theta^M$ be any basic weak $\kappa$-model in $M$. Since $M_0$ has size $\kappa$, the restriction $j:M_0 \to j(M_0)$ is in $N$. Next, let $M_1\prec H_\theta^M$ be a basic weak $\kappa$-model in $M$ such that $M_0, U \cap M_0 \in M_1$. Again, we have $j:M_1 \to j(M_1)$ is in $N$. In this way, we construct an elementary sequence $\{ M_n\mid n<\omega\}$ such that for each $n<\omega$, $M_n,U\cap M_n\in M_{n+1}$, and the restriction $j:M_n\to j(M_n)$ is in $N$.

Now working in $N$, we define a tree $T$ whose elements are finite sequences $$\la h_0:m_0\to k_0,\ldots,h_{n-1}:m_{n-1}\to k_{n-1}\ra,$$ ordered by end-extension, such that for all $i,j<n$:
\begin{enumerate}
\item $m_i\prec H_\theta^M$,
\item $m_i,u_i\in m_{i+1}$, where $u_i$ is the $m_i$-ultrafilter derived from $h_i$.
\item $m_i\prec m_j$, $k_i\prec k_j$, and $h_i\subseteq h_j$.
\end{enumerate}
Note crucially that the tree $T$ is an element of $N$ by the assumption that $j(\kappa)$ is a set in $N$. The sequence $\{ j_n:M_n\to j(M_n)\mid n<\omega\}$, constructed above, witnesses that the tree $T$ has a branch, and is thus ill-founded in $V$. By the absoluteness of well-foundedness, $T$ is ill-founded in $N$ as well. Let $h=\bigcup_{i<\omega} h_i:m\to n$, where $m=\bigcup_{i<\omega}m_i$ and $n=\bigcup_{i<\omega}n_i$. Finally, observe that $u=\bigcup_{i<\omega}u_i$ is the $m$-ultrafilter derived from $h$, and must be weakly amenable and good by construction.
\end{proof}
It is shown in~\cite[Proposition 5.2]{HolySchlicht:HierarchyRamseyLikeCardinals} that a $\kappa$-Ramsey cardinal $\kappa$ is a limit of super Ramsey cardinals. Let us say that a cardinal $\kappa$ is ${<}\alpha$-\emph{Ramsey} if it is $\beta$-Ramsey for all regular $\omega\leq \beta<\alpha$.

\begin{proposition}
A strongly Ramsey cardinal is a limit of cardinals $\alpha$ that are ${<}\alpha$-Ramsey.
\end{proposition}
\begin{proof}
Suppose that $\kappa$ is strongly Ramsey. Let $M$ be a simple $\kappa$-model for which there is a weakly amenable $M$-ultrafilter $U$. Let $j:M\to N$ be the ultrapower by $U$. We will argue that $\kappa$ is ${<}\kappa$-Ramsey in $V_{j(\kappa)}^N$. Fix $\alpha<\kappa$. By Theorem~\ref{th:alphaRamseyEquivGame}, it suffices to show that the challenger does not have a winning strategy in the game $\mathsf{RamseyG}^{\kappa^+}_\alpha(\kappa)$.  Suppose towards a contradiction that, in $V^N_{j(\kappa)}$, the challenger has a winning strategy $\sigma$ in $\mathsf{RamseyG}^{\kappa^+}_\alpha(\kappa)$. Note that by the weak amenability of $U$, the moves of the challenger are in $M$. Consider the following run of the game $\mathsf{RamseyG}^{\kappa^+}_\alpha(\kappa)$. The challenger plays some $M_0$ according to $\sigma$. The judge responds with $U\cap M_0$, which is an element of $M$ by weak amenability. The challenger then plays $M_1$ according to $\sigma$, and the judge again responds with $U\cap M_1$. Suppose that the challenger and the judge continue to play in this manner up to some limit step $\beta$. Since $M$ is a $\kappa$-model, the run of the game up to $\beta$ is an element of $M$, and hence, the challenger can respond to it with some $M_\alpha$ according to $\sigma$. Thus, the judge and the challenger can continue playing in this manner for $\alpha$-many steps. The entire run of the game is also an element of $M$ by closure. But clearly the judge wins this play, contradicting our assumption that $\sigma$ was a winning strategy for the challenger. Now we can conclude by elementarity that, in $V_\kappa$, $\kappa$ is a limit of cardinals $\alpha$ that are ${<}\alpha$-Ramsey. Moreover, $V_\kappa$ is correct about this since to verify that $\alpha$ is $\beta$-Ramsey (for some $\beta<\kappa$) we only need to consider games with elementary substructures of $H_{\alpha^+}\subseteq V_\kappa$.
\end{proof}

For future sections, let us consider a variant game $\mathsf{Ramsey\bar G}^{\theta}_\alpha(\kappa)$, where we ask the judge to play, instead of an $M_\gamma$-ultrafilter $U_\gamma$ extending her moves from the previous stages, a structure $\la N_\gamma,\in,U_\gamma\ra$ such that $N_\gamma$ is a $\kappa$-model with $P^{M_\gamma}(\kappa)\subseteq N_\gamma$ and $U_\gamma$ is an $N_\gamma$-ultrafilter, keeping the requirement that the $U_\gamma$'s must extend. We additionally require the challenger to ensure that the sequence of the judge's previous plays $\{\la N_\xi,\in,U_\xi\ra\mid \xi<\gamma\}$ is an element of $M_\gamma$. In this variant game, we are giving the judge the extra ability to ensure that certain subsets of $\kappa$ make it into the final model. Essentially the same argument as for the game $\mathsf{RamseyG}^\theta_\alpha(\kappa)$ shows that the existence of a winning strategy for either player is independent of $\theta$ in the game $\mathsf{Ramsey\bar G}^{\theta}_\alpha(\kappa)$, and moreover $\kappa$ is $\alpha$-Ramsey if and only if the challenger doesn't have a winning strategy in the game $\mathsf{Ramsey\bar G}^{\theta}_\alpha(\kappa)$. It follows that the challenger has a winning strategy in the game $\mathsf{RamseyG}^\theta_\alpha(\kappa)$ if and only if he has a winning strategy in $\mathsf{Ramsey\bar G}^{\theta}_\alpha(\kappa)$. But, in fact, this is true for the judge as well.

\begin{lemma}
Either player has a winning strategy in the game $\mathsf{RamseyG}^{\theta}_\alpha(\kappa)$ if and only if they have a winning strategy in the game $\mathsf{Ramsey\bar G}^{\theta}_\alpha(\kappa)$.
\end{lemma}
\begin{proof}
We argued above that this is true for the challenger. So suppose that the judge has a winning strategy $\sigma$ in the game $\mathsf{RamseyG}^{\theta}_\alpha(\kappa)$. The winning strategy $\bar\sigma$ for the judge in the game $\mathsf{Ramsey\bar G}^{\theta}_\alpha(\kappa)$ is going to be to play $\la N_\xi,\in,U_\xi\ra$ at stage $\xi$ of the game, where $U_\xi$ is the response of the judge according to $\sigma$ to the moves of the challenger so far (because these could very well be the moves of the challenger in the game $\mathsf{Ramsey G}^{\theta}_\alpha(\kappa)$) and $N_\xi=M_\xi\cap H_{\kappa^+}$. Next, suppose that the judge has a winning strategy $\bar\sigma$ in the game $\mathsf{Ramsey\bar G}^{\theta}_\alpha(\kappa)$. The winning strategy $\sigma$ for the judge in the game $\mathsf{RamseyG}^{\theta}_\alpha(\kappa)$ is going to be to play $U_\xi\cap M_\xi$ at stage $\xi$ of the game, where $\la N_\xi,\in,U_\xi\ra$ is the response of the of judge according to $\bar\sigma$ to the moves of the challenger so far (because again these could be the moves of the challenger in the game $\mathsf{Ramsey\bar G}^{\theta}_\alpha(\kappa)$). Note that if $\{U_\xi\mid\xi<\gamma\}\in M_\gamma$, then the sequence $\{\la N_\xi,\in,U_\xi\ra\mid \xi<\gamma\}$ is also in $M_\gamma$ because each $N_\xi$ is definable from $U_\xi$ since every element of $N_\xi$ is coded by a subset of $\kappa$ and $U_\xi$ has either the subset or its complement.
\end{proof}
Above the $\alpha$-Ramsey hierarchy, but still below a measurable cardinal sits a large cardinal notion introduced by Holy and L\"ucke \cite{HolyLucke:smallModels}.

\begin{definition}[{\cite[Definition 1.1]{HolyLucke:smallModels}}]
\label{def: loc meas}
A cardinal $\kappa$ is \emph{locally measurable} if every $A\subseteq\kappa$ is an element of a weak $\kappa$-model $M$ which thinks that it has normal ultrafilter on $\kappa$.
\end{definition}
Standard arguments show that a measurable cardinal is a limit of locally measurable cardinals. It follows immediately from \cite[Theorem 15.3]{HolyLucke:smallModels} that a locally measurable cardinal is a limit of cardinals $\kappa$ that are $\kappa$-Ramsey. To see this, note that the set of $\Delta^\forall_\kappa$-Ramsey cardinals is unbounded in $\kappa$ by this theorem. Moreover, by the list of large cardinal properties after \cite[Definition 9.4]{HolyLucke:smallModels}, a cardinal $\kappa$ is $\kappa$-Ramsey if and only if it is $\mathbf{T}^\forall_\kappa$-Ramsey, where the latter follows from $\Delta^\forall_\kappa$-Ramsey by \cite[Definition 9.4]{HolyLucke:smallModels}. As happens often with these smaller large cardinals, replacing weak $\kappa$-models by $\kappa$-models in the definition of locally measurable increases consistency strength (the argument is similar to that given in Proposition~\ref{bm stronger than weaklybm}). In the situation of Definition \ref{def: loc meas}, $P^M(\kappa)$ must be an element of any such weak $\kappa$-model $M$, and so it obviously cannot be simple.
While the normal ultrafilter $U$ is not required to be good, Proposition \ref{prop:locallyMeasurableInternalUltrapower} shows that this would not add extra strength.

\begin{proposition}\label{prop:locallyMeasurableInternalUltrapower}
If $\kappa$ is locally measurable, then every $A\subseteq\kappa$ is an element of a weak $\kappa$-model $M$ such that $M$ contains what it thinks is a normal ultrafilter $U$ on $\kappa$ with a well-founded ultrapower $N\subseteq M$.
\end{proposition}
\begin{proof}
Fix $A\subseteq\kappa$. Choose any weak $\kappa$-model $M$ such that $A\in M$ and $M$ has what it thinks is a normal ultrafilter $U$ on $\kappa$. Let $\delta=(2^\kappa)^M$. We can assume without loss of generality that $\delta$ is the largest cardinal of $M$ by replacing $M$ with $H_{\delta^+}^M$ if necessary.
Let $\la N,\bar{\in}\ra$ be the ultrapower of $M$ by $U$. The model $M$ believes that the relation $\bar{\in}$ is well-founded in the sense that it doesn't have a sequence of functions $\{f_n:\kappa\to M\mid n<\omega\}$ such that for all $n<\omega$, $\{\xi<\kappa\mid f_{n+1}(\xi)\in f_n(\xi) \}\in U$. Since models of $\ZFC^-$ can perform Mostowski collapses, it suffices to argue that the relation $\bar{\in}$ is set-like from the perspective of $M$, namely for every function $f:\kappa\to M$, there is a set $X_f$ in $M$ such that for every $g:\kappa\to M$ with $[g]_U \bar{\in} [f]_U$, there is some $g':\kappa\to M$ with $[g']_U=[g]_U$ and $g'\in X_f$. By elementarity, $[c_\delta]_U$, where $c_\delta$ is the constant function with value $\delta$, is the largest cardinal of $N$, and thus, every element of $N$ is bijective with $[c_\delta]_U$. So it actually suffices to argue that $M$ has a set $X_{c_\delta}$ as above. For this, observe that $[g]_U\,\bar{\in}\,[c_\delta]_U$ if and only if $[g]_U=[g']_U$ for some $g':\kappa\to \delta$, and the collection of functions $g':\kappa\to \delta$ is a set in $M$ because $\delta=2^\kappa$ is a set in $M$.
\end{proof}
Thus, as long as $2^\kappa$ is the largest cardinal of a weak $\kappa$-model $M$ with a normal ultrafilter $U$ on $\kappa$, the ultrapower of $M$ by $U$ is contained in $M$ just as in the case of an ultrapower by an actual measure. In fact, it follows from the proof that this will be case whenever $M$ has a largest cardinal $\gamma$ and $\gamma^\kappa$ is a set in $M$. Note that it is definitely possible for an ultrapower of a weak $\kappa$-model $M$ by a normal ultrafilter $U$ in $M$ to have more ordinals than $M$. Suppose that $\kappa$ is measurable, $U$ is a normal ultrafilter on kappa, $j:V\to M$ is the ultrapower embedding, and $\lambda>2^\kappa$ is a cardinal with $\text{cof}(\lambda)=\kappa$. By elementarity, $\text{cof}(j(\lambda))=j(\kappa)$ in $M$, which means, since $M^\kappa\subseteq M$, that $j(\lambda)>\lambda^+$, and hence also then $(j(\lambda)^+)^M> \lambda^+$. Thus, $H_{(j(\lambda)^+)^M}^M$, which is the ultrapower of $H_{\lambda^+}$ by $U$, has more ordinals than $H_{\lambda^+}$. Let $N$ be the Mostowski collapse of an elementary substructure of $H_{\lambda^+}$ of size $\kappa$. Let $\bar U$ be the preimage of $U$ under the collapse and let $\bar \lambda$ be the preimage of $\lambda$. Then, by elementarity, $N$ satisfies that the image of $\bar\lambda$ under the ultrapower map cannot be a set in $M$.

\section{From amenability to collection}
\label{section: amen coll}

Suppose that $M$ is a simple weak $\kappa$-model and $U$ is an $M$-ultrafilter. By assumption, $M$ is a model of $\ZFC^-$, but separation and replacement can fail very badly in the structure $\la M,\in, U\ra$ once the model learns about the ultrafilter. Weak amenability of $U$ is equivalent to having a minimal amount of separation for $\la M,\in,U\ra$. Note that for the results in this section, we mostly do not need to assume that the ultrapower of $M$ by $U$ is well-founded, that is, $U$ need not be good.

\begin{lemma}
Suppose that $M$ is a simple weak $\kappa$-model and $U$ is an $M$-ultrafilter. Then $U$ is weakly amenable if and only if the structure $\la M,\in,U\ra$ satisfies $\Delta_0$-separation.
\end{lemma}
\begin{proof}
Suppose that $\Delta_0$-separation holds in the language with $U$. Fix $A\in M$ and observe that $$U\cap A=\{x\in A\mid x\in U\},$$ which indeed requires separation only for atomic formulas. Now suppose that $U$ is weakly amenable to $M$. Fix $A\in M$ and a $\Delta_0$-formula $\varphi(x,b)$ in the language with $U$. We need to argue that $$\{x\in A\mid M\models\varphi(x,b)\}\in M.$$ Now observe that since $\varphi(x,b)$ is a $\Delta_0$-formula, all quantifiers are bounded and therefore, $U$ restricted to $\mathrm{TCl}(b)$, the transitive closure of $b$, suffices to interpret $\varphi(x,b)$ correctly. But this piece of $U$ is an element of $M$ by weak amenability.
\end{proof}

It follows from Theorem \ref{th:0babymeasurable} below that a structure $\la M,\in,U\ra$ with a weakly amenable $M$-ultrafilter $U$ need not need to satisfy even $\Delta_0$-replacement.

Next, let's observe that our typical structures $\la M,\in,U\ra$ have a $\Delta_1$-definable total order that is a strong well-order from the point of view of the structure. Let us say that a total order $\lhd$ on a weak $\kappa$-model $M$ is a \emph{strong well-order of $M$} if for every formula $\varphi(x,a)$ over $\la M,\in\ra$, the structure $\la M,\in,\lhd\ra$ satisfies that there is a $\lhd$-least $b$ such that $\varphi(b,a)$. Usually, a well-order is defined to have the property that every set has a least element. This property is equivalent to our stronger requirement for a set-like order, but the orders we encounter won't necessarily be set-like. These structures $\la M,\in,U\ra$ will also have a $\Delta_1$-definable truth predicate for $\la M,\in\ra$.

\begin{lemma}\label{lem:definableWellOrder}
Suppose that $M$ is a simple weak $\kappa$-model and $U$ is a weakly amenable $M$-ultrafilter. Then the structure $\la M,\in,U\ra$ has a $\Delta_1$-definable strong well-order $\lhd_U$ of $M$.
\end{lemma}
\begin{proof}
Let $\la N,\bar{\in}\ra$ be the possibly ill-founded ultrapower by $U$. Since $U$ is weakly amenable, we have $M=H_{\kappa^+}^N$ and $(\kappa^+)^N=\Ord^M$ by Lemma~\ref{prop:wellFoundedPartOfUltrapowerWA}.
Note that every set $a\in M$ can be coded in $M$ by a subset $A$ of $\kappa$. The code $A$ codes a subset of $\kappa\times\kappa$, which in turn, codes the membership relation on $\mathrm{TCl}(a)$.
While a set $a$ can have many different codes, $N$ has a set $\mathcal C$ consisting of a unique code for every set in $M$ by elementarity, since this holds true in $M$ of every $H_\nu$.
In fact, $N$ has a membership relation $\mathcal E$ for elements of this set as well, making $\la \mathcal C,\mathcal E\ra$ isomorphic to $\la M,\in\ra$. Let $[C]_U$ be the equivalence class representing the set $\mathcal C$ of $N$.

Let's consider the complexity of a series of statements in the structure $\la M,\in,U\ra$.
First observe that any $A\subseteq\kappa$ from $M$ is represented in the ultrapower $N$ by the equivalence class of the function $f_A$ such that $f_A(\xi)=A\cap \xi$. To verify this, it suffices to check that $[f_A]_U\cap \xi=A\cap\xi$ for every $\xi<\kappa$. Since $\xi$ and $A\cap \xi$ are in $V_\kappa$, they are represented by the constant functions $[c_\xi]_U$ and $[c_{A\cap\xi}]_U$. So we need to check that $$[f_A]_U\cap [c_\xi]_U=[c_{A\cap \xi}]_U$$ holds in the ultrapower $N$.
By \Los' theorem, this holds if and only if $$\{\alpha<\kappa\mid f_A(\alpha)\cap \xi=A\cap \xi\}\in U,$$ but this is certainly true since it holds on a tail. Let ``$A$ is a code" be the assertion that $A$ is an element of $\mathcal C$ in the ultrapower $N$. The complexity of the statement ``$A$ is a code" in the structure $\la M,\in,U\ra$ is $\Delta_1$. The $\Sigma_1$-version says that there exists a function $f$ such that $f(\xi)=A\cap \xi$ and there exists a set $\{\alpha<\kappa\mid f(\xi)\in C(\xi)\}$ and this set is in $U$, and the $\Pi_1$-version says that for every function $f$ such that $f(\xi)=A\cap \xi$ and for every set equal to $\{\alpha<\kappa\mid f(\xi)\in C(\xi)\}$, the set is in $U$.

Next, let ``$A$ is the code for $a$" be the conjunction of assertions that ``$A$ is a code" and that $A$ codes $a$ as explained above. The assertion that $A$ codes $a$ is actually $\Delta_1$ in the structure $\la M,\in\ra$ without the ultrafilter.  The $\Sigma_1$-assertion is that there are sets $\bar a$ and $\pi$ such that $\bar a=\mathrm{TCl}(a)$ and $\pi$ is an isomorphism between $\bar a$ and $A$. Note that $\bar a=\mathrm{TCl}(a)$ is $\Delta_1$. The $\Pi_1$-assertion is that for every set $\bar a$, $B$, and $\pi$, if $\bar a=\text{TCl}(a)$ and $\pi$ is an isomorphism between $\bar a$ and $B$, then $A=B$.

Finally, let $[W]_U$ represent the equivalence class of a well-order $\mathcal W$ of $\mathcal C$ in the ultrapower $N$. Given $a$ and $b$ in $M$, we define that $a\lhd_U b$ whenever there are codes $A$ and $B$ of $a$ and $b$ respectively such that $A$ is below $B$ according to $\mathcal W$. The assertion that $A$ is below $B$ according to $\mathcal W$ translates to saying that $[f_A]_U$ is below $[f_B]_U$ according to $[W]_U$. With our preliminary analysis this is clearly a $\Delta_1$-assertion in the language with $U$, using \Los' theorem.
\end{proof}

If the $\GCH$ holds below $\kappa$, or just $\{\alpha<\kappa\mid M\models 2^\alpha=\alpha^+\}\in U$, then the ultrapower $N$ has a well-order of $M$ of order-type $\kappa^+$, which must then be an actual well-order as $\kappa^{+N}$ is well-founded.
If the $M$-ultrafilter $U$ is good, and hence $N$ is well-founded, then its well-order of $M$ is an actual well-order. Otherwise, it is possible that $N$'s well-order of $M$ is ill-founded externally.

Recall that a \emph{choice function} for a binary relation $R$ on a set $X$ is a function $F\colon X\rightarrow X$ with $(x,F(x))\in R$ for all $(x,y)\in R$. The next lemma shows that our typical structures $\la M,\in, U\ra$ have a $\Delta_1$-definable choice function for every relation $R$ on $M$ that exists in the ultrapower of $M$ by $U$.

\begin{lemma}
\label{Delta1ClassChoice}
Suppose that $M$ is a simple weak $\kappa$-model and $U$ is an $M$-ultrafilter such that $\langle M,\in,U\rangle$ is
weakly amenable. Then any binary relation $R$ on $M$ that exists in the ultrapower of $M$ by $U$ admits a choice function that is $\Delta_1$-definable in $\langle M,\in,U\rangle$.
\end{lemma}
\begin{proof}
Let $N$ be the, possibly ill-founded, ultrapower of $M$ by $U$. Let $R\in N$ be a binary relation on $M$, and observe that $N$ must have choice functions for $R$. Let $\bar F\in M$ be a function on the class of codes $\mathcal C$ corresponding to a choice function on $R$. Then we can define a choice function $F$ on $R$ by $(x,y)\in F$ whenever $A_x$ is a code for $x$, $A_y$ is a code for $y$ and $\bar F(A_x)=A_y$. Letting $[f]_U$ represent the equivalence class of $\bar F$, we use similar arguments as above to show that $F$ is $\Delta_1$ definable over $\la M,\in, U\ra$.
\end{proof}

A similar argument shows that the ultrafilter provides a definable truth predicate.

\begin{lemma}
\label{truth predicate from ultrafilter}
Suppose that $M$ is a simple weak $\kappa$-model and $U$ is a weakly amenable $M$-ultrafilter.
Then $\la M,\in,U\ra$ has a $\Delta_1$-definable truth predicate $\Tr_M(\varphi,x)$ for $\la M,\in\ra$.
\end{lemma}
\begin{proof}
Let $N$ be the possibly ill-founded ultrapower of $M$ by $U$. In the notation of the proof of Proposition~\ref{lem:definableWellOrder}, the model $N$ has a truth predicate $\mathcal T$ for the structure $\la \mathcal C,\mathcal E\ra$. Thus, we have $M\models\varphi(a)$ whenever $A$ is the code for $a$ and $\la \varphi,A\ra\in \mathcal T$. Let $[T]_U$ represent $\mathcal T$ in the ultrapower $N$. Then $\la\varphi,A\ra\in \mathcal T$ if and only if
$\{\alpha<\kappa\mid \la \varphi,f_A(\alpha)\ra\in T(\alpha)\}\in U$.
\end{proof}

\begin{lemma}
\label{truth pred for simple weak}
Suppose that $M\models\ZFC^-$, $U$ is an amenable predicate on $M$.
Then for every $n<\omega$, $\la M,\in,U\ra$ has a $\Sigma_n$-definable truth predicate $\Tr^U_n(\varphi,x)$ for $\Sigma_n$-formulas in the language with $U$.
\end{lemma}
\begin{proof}
Given a $\Delta_0$-formula $\varphi(x,a)$ in the language with $U$, observe that we only need\break $U_a:=U\cap \mathrm{TCl}(a)$ to evaluate it. By the amenability of $U$, we have that $U_a\in M$ for every $a\in M$. Thus, a truth predicate for a $\Delta_0$-formula $\varphi(x,a)$ can be defined as usual in a $\Delta_1$ fashion with $U_a$ as a parameter. This means we should require that the set $U_a$ is a part of the witnessing sequence for truth and the extra step in the definition needs to check that $U_a=U\cap \mathrm{TCl}(a)$, but that is a $\Delta_0$-assertion. Once we have truth for $\Delta_0$-formulas, the rest follows by induction on complexity of formulas because we have only finitely many quantifiers.
\end{proof}

We will mainly be interested in properties of models $\la M,\in, U\ra$ of the following fragments of $\ZFC^-$:

\begin{itemize}
\item
$\ZFC^-_n$ denotes the theory where the collection and separation schemes are restricted to $\Sigma_n$-formulas.

\item
$\KP_n$ denotes the theory where the collection scheme is restricted to $\Sigma_n$-formulas and the separation scheme is restricted to $\Delta_0$-formulas.
\end{itemize}

Note that $\ZFC^-_0$, $\KP_0$ and $\KP_1$ are equivalent. The theory $\KP_n$ further implies fragments of separation and recursion.
Parts of the next folklore result are stated without proof in \cite[Theorem 1.5]{kranakis1982reflection}.

\begin{lemma}
\label{KP implies sep}
The theory $\KP_n$ implies the schemes of $\Sigma_n$-replacement, $\Delta_n$-separation and $\Sigma_n$-recursion along the ordinals.
\end{lemma}
\begin{proof}
The claims for $n=0$ and $n=1$ are identical, since $\KP_0=\KP_1$.
Their proofs are easy variants of the following argument.

For $n>1$, suppose that $M$ is a model of $\KP_n$.
We work in $M$.
We have $\Sigma_{n-2}$-separation by the induction hypothesis. We first show that $\Sigma_{n-1}$-separation holds.
Suppose that $A$ is a set and $\varphi(x,y)$ is a $\Pi_{n-2}$-formula. We need to show that the set $$ B:= \{ x\in A\mid \exists y\, \varphi(x,y)  \} $$ is in $M$. Consider the $\Pi_{n-1}$-formula $\bar\varphi(x,y):=\varphi(x,y)\vee \forall z\,\neg\varphi(x,z)$ and observe that $$M\models\forall x\in A\,\exists y\,\bar\varphi(x,y).$$ Thus, by $\Sigma_n$-collection, there is a set $C$ such that $$M\models\forall x\in A\,\exists y\in C\,\bar\varphi(x,y).$$ Thus, also, $M\models\forall x\in A\exists y\in C\,\varphi(x,y)$. It follows that $$B=\{x\in A\mid \exists y\in C\,\varphi(x,y)\}.$$ The formula $\exists y\in C\,\varphi(x,y)$ is a equivalent to a $\Pi_{n-2}$-formula by $\Sigma_{n-2}$-collection, and so we can use $\Sigma_{n-2}$-separation to conclude that $B$ exists.

For $\Sigma_n$-replacement, suppose that $A$ is a set and $\varphi$ is a $\Pi_{n-1}$-formula such that $\exists z\ \varphi(x,y,z)$ defines a function $F\colon A\rightarrow M$.
By $\Sigma_n$-collection, there exists a set $B$ such that for all $x\in A$, there is some $y\in B$ with $\exists z\ \varphi(x,y,z)$.
Again by $\Sigma_n$-collection, there exists a set $C\in M$ such that for all $x\in A$, there is some $z\in C$ with $\exists y\ \varphi(x,y,z)$.
Since $y$ is uniquely determined by $x$, $$\ran(F)=\{ y\in B \mid \exists z\in C\ \varphi(x,y,z) \}.$$
By $\Sigma_{n-1}$-collection, the formula $\exists z\in C\ \varphi(x,y,z)$ is equivalent to a $\Pi_{n-1}$-formula, and so $\ran(F)$ is a set by $\Sigma_{n-1}$-separation.

For $\Delta_n$-separation, suppose that $A$ is a set and $\varphi(x)$, $\psi(x)$ are $\Pi_n$-formulas such that $$\varphi(x) \leftrightarrow \neg \psi(x)$$ holds for all $x\in A$.
We need to show that $B:=\{x\in A \mid \varphi(x)\}$ is a set.
Assume there exists some $y\in A$ with $\varphi(y)$.
Since the function $F\colon A\rightarrow A$ defined by letting $F(x)=x$ if $\varphi(x)$ holds and $F(x)=y$ if $\psi(x)$ holds is $\Sigma_n$-definable, $B=\ran(F)$ is a set by $\Sigma_n$-replacement.

$\Sigma_n$-recursion along ordinals now follows just like $\Sigma_1$-recursion follows from $\KP$, using $\Sigma_n$-replacement for the existence of the recursion at limit stages.
In fact, the proof works for recursion along $\Delta_n$-definable strongly well-founded relations.
\end{proof}
Thus, in particular, $\KP_{n+1}$ implies $\ZFC_n^-$, which in turn implies $\KP_n$.

The fragments $\ZFC_n^-$ and $\KP_n$ have their own advantages for different situations. The theory $\ZFC^-_n$ implies a fragment of \Los' theorem for $\Sigma_n$-formulas. To state this, we first fix some notation.
Suppose that $M$ is a weak $\kappa$-model and $U$ is a weakly amenable $M$-ultrafilter. Let $N$ be the ultrapower of $M$ by $U$ and let $$W=\{[f]_U\subseteq [c_\kappa]_U\mid \{\alpha<\kappa\mid f(\alpha)\in U\}\in U\}$$ be the $N$-ultrafilter on $[c_\kappa]_U$ derived from $U$. We know that \Los' theorem holds for the ultrapower in the language $\in$. We will say that \Los' theorem \emph{holds} for an assertion $\varphi(x_1,\ldots,x_n)$ in the language with a predicate for the ultrafilter if for every sequence $\vec f=\la f_1,\ldots,f_n\ra$ with $f_i:\kappa\to M$, the set

 $$A_{\vec f,\varphi}=\{ \alpha<\kappa\mid \la M,\in,U\ra\models\varphi(f_1(\alpha),\ldots,f_n(\alpha))\}\in M$$ and $\la N,\in,W\ra\models\varphi([f_1]_U,\ldots,[f_n]_U)$ if and only if $A_{\vec f,\varphi}\in U$.

\begin{lemma}
\label{Los theorem}
Suppose $M$ is a weak $\kappa$-model and $U$ is an $M$-ultrafilter. If $\la M,\in,U\ra\models\ZFC^-_n$, then \Los' theorem holds for $\Sigma_n$ and $\Pi_n$-assertions in the language with a predicate for the ultrafilter.
\end{lemma}
\begin{proof}
Let's first argue that the extended \Los' theorem is true for all $\Delta_0$-assertions. It is true for atomic formulas by the definition of $W$ and the case of conjunctions is clear. So let's suppose that the assertion is true for some $\Delta_0$-formula $\varphi(x)$ and argue that it is true for $\neg\varphi(x)$. By our assumption, for every function $f:\kappa\to M$ from $M$, the set $$A_{\la f\ra}=\{\alpha<\kappa\mid \la M,\in,U\ra\models \varphi(f(\alpha))\}\in M.$$ Thus, its complement $$A_{\la f\ra,\neg\varphi}=\{\alpha<\kappa\mid \la M,\in,U\ra\models\neg\varphi(f(\alpha))\}$$ is in $M$ as well. Now we have that $A_{\la f\ra,\neg\varphi}\in U$ if and only if $A_{\la f\ra,\varphi}\notin U$ if and only if $\la N,\in,W\ra$ does not satisfy $\varphi([f]_U)$ if and only if $\la N,\in,W\ra\models\neg\varphi([f]_U)$.
So suppose now that the assertion holds for a $\Delta_0$-formula $\varphi(x,y)$. First, observe that for functions $f:\kappa\to M$ and $g:\kappa\to M$ from $M$, the set $$A_{\la f,g\ra,\varphi}=\{\alpha<\kappa\mid \la M,\in,U\ra\models\exists x\in g(\alpha)\,\varphi(x,f(\alpha))\}$$ is in $M$ by $\Delta_0$-separation because $G=\bigcup_{\alpha<\kappa}g(\alpha)$ is a set in $M$. Suppose next that $$\la N,\in,W\ra\models\exists x\in [g]_U\,\varphi(x,[f]_U).$$
Then there is $[h]_U$ such that $$\la N,\in,W\ra\models\varphi([h]_U,[f]_U)\wedge [h]_U\in[g]_U$$ and so by our assumption, the set $$A_{\la f,g,h\ra,\varphi}=\{\alpha<\kappa\mid \varphi(h(\alpha),f(\alpha))\wedge h(\alpha)\in g(\alpha)\}\in U,$$ and hence $A_{\la f,g\ra,\varphi}\in U$. In the other direction, suppose that $A_{\la f,g\ra,\varphi}\in U$. Let $w_G\in M$ be any well-order of $G$. Using $w_G$, we can choose for every $\alpha\in A_{\la f,g\ra,\varphi}$ the $w_G$-least $b\in g(\alpha)$ such that $\la M,\in,U\ra \models\varphi(b,f(\alpha))$ and define $h(\alpha)=b$. The definition of $h$ uses $\Delta_0$-separation in the structure $\la M,\in, U\ra$.

Next, observe that the negation case holds for any formula provided that we have the inductive assumption for that formula, and so it remains to argue the case of the unbounded existential quantifier. So suppose that the inductive assumption holds for a formula $\varphi(x,y)$ of complexity at most $\Pi_{n-1}$. The set $$A_{\la f\ra,\varphi}=\{\alpha<\kappa\mid \la M,\in,U\ra\models\exists x\,\varphi(x,f(\alpha))\}$$ is in $M$ since $\la M,\in,U\ra$ satisfies $\Sigma_n$-separation. If $\la N,\in,W\ra\models\exists x\,\varphi(x,[f]_U)$, then as before there is some $[h]_U$ such that $\la N,\in,W\ra\models\varphi([h]_U,[f]_U)$, so by the inductive assumption $A_{\la f,h\ra,\varphi}\in U$, which in turn implies that $A_{\la f\ra,\exists x\varphi}\in U$. In the other direction, supposing that $A_{\la f\ra,\varphi}\in U$, we use $\Sigma_n$-collection to obtain a set $C$ such that for every $\alpha\in A_{\la f\ra,\varphi}$, there is $x\in C$ for which $\la M,\in, U\ra\models\varphi(x,f(\alpha))$. Let $w_C\in M$ be a well-order of $C$, and use $w_C$ to chose for every $\alpha\in A_{\la f\ra,\varphi}$, the $w_C$-least $b$ such that $\la M,\in,U\ra \models\varphi(b,f(\alpha))$ and define $h(\alpha)=b$. The definition of $h$ uses $\Sigma_n$-separation in the structure $\la M,\in, U\ra$.
\end{proof}

The next lemma shows that the fragment $\KP_n$ is more natural than $\ZFC^-_n$ with respect to forming $\Sigma_n$-elementary substructures. Namely, we have that any element of a model of $\KP_n$ that is transitive and $\Sigma_n$-elementary in the model is itself a model of $\KP_n$. This can fail to hold for models of $\ZFC_n$. For instance, it is easy to see that the $\Sigma_1$-Skolem hull is $\Sigma_1$-definable in any model $L_\alpha\models\ZFC^-_1$, and therefore must be a set in $L_\alpha$ by collection. Thus, every model $L_\alpha\models\ZFC^-_1$ has a proper $\Sigma_1$-elementary substructure, but then the least $\Sigma_1$-substructure of $L_\alpha$ cannot be a model of $\ZFC^-_1$. In the complexity calculations that follow, we will repeatedly use the observation that $\Sigma_n$-collection implies a normal form theorem for $\Sigma_n$-assertions, namely that every assertion with a $\Sigma_n$-alternation of unbounded quantifiers is equivalent to a $\Sigma_n$-assertion.

\begin{lemma}
\label{lem:nElemImpliesKPn}
Suppose that $M\models\ZFC^-$ and $U$ is a predicate such that $\la M,\in,U\ra\models\KP_n$ for some $n\geq 1$. If $\bar M\in M$ is transitive in $M$ and $\la \bar M,\in,U\cap\bar M\ra\prec_{\Sigma_n}\la M,\in,U\ra$, then $\la\bar M,\in,U\cap\bar M\ra\models\KP_n.$
\end{lemma}
\begin{proof}
$\Delta_0$-separation in $\la \bar M,\in,U\cap\bar M\ra$ can be easily verified as follows. Suppose that $A\in \bar M$ and $\varphi(x)$ is a $\Delta_0$-formula. Then, by $\Delta_0$-separation, $$\la M,\in,U\ra\models\exists z\,(\forall x\in z\, x\in A\wedge \forall x\in A(\varphi(x)\leftrightarrow x\in z)).$$ Since this is $\Sigma_1$-assertion, we have that it holds in $\la \bar M,\in, U\cap\bar M\ra $ by $\Sigma_1$-elementarity.

Next, we verify $\Sigma_n$-collection. Suppose that $$\la \bar M,\in,U\cap\bar M\ra\models\forall x\in a\,\exists y\,\varphi(x,y,a),$$ for a $\Sigma_n$-formula $\varphi(x,y,a)$. Indeed, we can assume without loss of generality that $\varphi(x,y,a)$ is $\Pi_{n-1}$.
Since $\bar M$ is transitive in $M$, $\Sigma_n$-elementarity yields
$$\la M,\in,U\ra \models \forall x\in a\,\exists y\,\varphi(x,y,a).$$
By $\Sigma_n$-collection, $\la M,\in,U\ra\models\psi(a)$, where $$\psi(a):=\exists z \forall x\in a\,\exists y\in z\,\varphi(x,y,a).$$
For $n=1$, the formula $\psi(a)$ is $\Sigma_1$, and so $\la \bar M,\in, U\cap \bar M\ra\models\psi(a)$ by $\Sigma_1$-elementarity, verifying $\Sigma_1$-collection. For $n>1$, we can assume inductively that we have already verified $\Sigma_{n-1}$-collection in $\la \bar M,\in,U\cap\bar M\ra$. Under $\Sigma_{n-1}$-collection, the formula $\forall x\in a\,\exists y\in z\,\varphi(x,y,a)$ is equivalent to a $\Sigma_n$-formula $\bar\psi(z)$. Since $\la M,\in,U\ra$ satisfies the $\Sigma_n$-formula $\exists z\,\bar\psi(z)$, by $\Sigma_n$-elementarity, $\la \bar M,\in, U\cap\bar M\ra\models\exists z\,\bar\psi(z)$. Thus, by $\Sigma_{n-1}$-collection, $\la \bar M,\in,U\cap \bar M\ra\models \psi(a)$, verifying $\Sigma_n$-collection.
\end{proof}

Since $\Sigma_0$-elementarity is equivalent to just being a submodel for transitive structures, it is not difficult to find a counterexample to Lemma~\ref{lem:nElemImpliesKPn} for $n=0$.

The \emph{$\Sigma_n$-reflection} scheme states that for every true $\Sigma_n$-assertion $\varphi(x,a)$ with parameter $a$, there is a transitive set $M$ containing $a$ such that for every $x\in M$, $M\models\varphi(x,a)$ if and only if $\varphi(x,a)$.  If $M$ is a simple weak $\kappa$-model and $U$ is a good $M$-ultrafilter such that $\la M,\in, U\ra\models\KP_{n+1}$, then the structure $\la M,\in, U\ra$ satisfies a strong form of $\Sigma_n$-reflection.

\begin{lemma}\label{lem:ZFCnImpliesElemSubstructures}
Suppose that $M$ is a simple weak $\kappa$-model and $U$ is a good $M$-ultrafilter.
If $\la M,\in,U\ra\models\KP_{n+1}$, for some $n\geq 1$, then for every $A\in M$ there is a $\kappa$-model $\bar M\in M$ such that $A\in M$, $\bar{M}\prec M$,
$$\la \bar M,\in,U\cap \bar M\ra\prec_{\Sigma_n}\la M,\in,U\ra,$$ and $\la \bar M,\in, U\cap \bar M\ra\models\ZFC^-_n$.
\end{lemma}
\begin{proof}
Recall that, by Lemma~\ref{lem:definableWellOrder}, the structure $\la M,\in, U\ra$ has a $\Delta_1$-definable well-order $\lhd_U$. Since $U$ is good, it follows that $\lhd_U$ is externally a well-order.
First, let's argue that for every set $X\in M$, there exists a unique set $X^*\in M$ of $\lhd_U$-least witnesses for $\Sigma_n$-assertions $\theta(x,a)$ in the language with the ultrafilter and with $a\in X$.
Let $C$ be a collecting set for the assertion
$$\forall a\in X\,\forall\theta\in\omega\,\exists y\,\varphi(a,\theta,y),$$ where $\varphi(a,\theta,y)$ is the assertion $$\left(y=\emptyset \wedge \forall z\,\neg \Tr_{n}^U(\theta,\la z,a\ra)\right)\vee\left(\Tr_{n}^U(\theta,\la y,a\ra)\wedge \forall z(z\lhd_U y\rightarrow \neg\Tr_{n}^U(\theta,\la z,a\ra))\right).$$ The complexity of the assertion $\varphi(a,\theta,y)$ is clearly at most $\Sigma_{n+1}$ (it is a disjunction of a $\Pi_n$ and a $\Sigma_n$-formula), and so we can apply $\Sigma_{n+1}$-collection.
Now using $\Delta_{n+1}$-separation for $C$ by Lemma \ref{KP implies sep}, we can obtain the desired unique set $X^*$.

Next, let's argue that every set $X\in M$ can be put into a $\kappa$-model $N\in M$ such that $N\prec M$. For every cardinal $\alpha<\kappa$, we have that every set in $H_{\alpha^+}$ is contained in an elementary submodel of size $\alpha$ that is closed under sequences of length ${<}\alpha$. Thus, the ultrapower of $M$ by $U$ satisfies that every set in $H_{\kappa^+}=M$ is contained in such an elementary submodel, and all these submodels are in $M$.

Observe that the assertion that $N$ is a $\kappa$-model is $\Delta_1$ for the following reasons. The assertion that $N$ is a weak $\kappa$-model is $\Delta_1$ because you just have to say that it is transitive, has $\kappa, V_\kappa$ as elements, and satisfies $\ZFC^-$ (this last gives the unbounded quantifier). The assertion that the model is closed under ${<}\kappa$-sequences is clearly $\Pi_1$, but in fact it is also $\Sigma_1$. The $\Sigma_1$-assertion is ``there is a bijection $f:V_\kappa\to N$ such that for every $g:\xi\to V_\kappa$ in $V_\kappa$ with $\xi<\kappa$, there exists $x\in N$ such that $x=f\circ g$." Next, observe that the assertion that $N$ is the $\lhd_U$-least $\kappa$-model extending a set $X$ is $\Pi_1$. Finally, observe that the assertion that $\la N,\in\ra\prec \la M,\in\ra$ is $\Delta_1$ using the $\Delta_1$-definable truth predicate $\Tr_M(\varphi,x)$ for $\la M,\in\ra$.

Now consider the assertion $\psi(s,\lambda)$, which states that $s$ is a sequence of length $\lambda$ such that:
\begin{enumerate}
\item $s_0=\{A\}$
\item $s_\delta=\bigcup_{\xi<\delta}s_\xi$ for limit ordinals $\delta$,
\item if $\xi+1$ is an even successor ordinal, then $s_{\xi+1}=s^*_{\xi}$, the unique closure under existential witnesses for $\Sigma_n$-assertions,
\item if $\xi+1$ is an odd successor ordinal, then $s_{\xi+1}$ is the $\lhd$-least $\kappa$-model such that\break $s_{\xi}\subseteq s_{\xi+1}\prec M$.
\end{enumerate}
The recursion defining the $s_\xi$ is $\Delta_{n+1}$, so we can use the $\Sigma_{n+1}$-recursion scheme to conclude that there exists a function $f:\kappa\to M$ such that $f(\xi)=s_\xi$.  Let $\bar M=\bigcup_{\xi<\kappa}f(\xi)$. By construction, $\la \bar M,\in, U\cap \bar M\ra\prec_n \la M,\in,U\ra$ with $\bar M\prec M$.
Since $\bar M$ is a $\kappa$-length union of $\kappa$-models in $M$, it is itself a $\kappa$-model in $M$. The model $\bar M$ really is a $\kappa$-model by Lemma~\ref{prop:internalKappaModel}. Note that the model $\la M,\in, U\cap\bar M\ra$ satisfies $\Sigma_n$-reflection by construction because it is a union of $\Sigma_n$-elementary substructures. It will follow by the next Lemma~\ref{lem:ElemSubstructuresImplyZFCn} that $\Sigma_n$-reflection implies $\ZFC^-_n$.
\end{proof}
Again, the result can fail for $n=0$ because being $\Sigma_0$-elementary and transitive is equivalent to being a submodel.
\begin{lemma}\label{lem:ElemSubstructuresImplyZFCn}
Suppose that $M\models\ZFC^-$ and $U$ is a predicate.
If $\la M,\in,U\ra$ satisfies $\Sigma_n$-reflection for some $n\geq 1$, then it is a model of $\ZFC_n^-$.
\end{lemma}
\begin{proof}
Concerning $\Sigma_n$-collection, suppose that $$\la M,\in,U\ra\models\forall x\in a\,\exists y\,\varphi(x,y,a)$$ for a $\Sigma_n$-formula $\varphi(x,y,a)$. Let $a\in \bar M\in M$ such that $$\la \bar M,\in,U\cap \bar M\ra\prec_{\Sigma_n}\la M,\in,U\ra.$$ Fix $b\in a$.
By assumption, $\la M,\in,U\ra\models\exists y\,\varphi(b,y,a)$.
By elementarity, we have $\la \bar M,\in,U\cap \bar M\ra\models\exists y\,\varphi(b,y,a)$.
Thus, $$\la \bar M,\in,U\cap \bar M\ra\models\forall x\in a\,\exists y\,\varphi(x,y,a),$$ and so $\bar M$ is the required collecting set for $\varphi$.

For $\Sigma_n$-separation, fix $a\in M$ and a $\Sigma_n$-formula $\psi(x,b)$ in the language with $U$. Let $a,b\in \bar N$ such that $$\la \bar N,\in,U\cap\bar N\ra\prec_{\Sigma_n}\la M,\in,U\ra.$$ The structure $\la \bar N,\in,U\cap\bar N\ra$ reflects $\la M,\in,U\ra$ for the formula $\psi(x,b)$. Therefore $$\{x\in a\mid \la M,\in,U\ra\models\psi(x,b)\}=\{x\in a\mid \la \bar N,\in, U\cap \bar N\ra\models\psi(x,b)\},$$ and the right-hand side set exists by separation in $M$.
\end{proof}

The hypothesis in Lemma \ref{lem:ZFCnImpliesElemSubstructures} cannot be reduced to $\ZFC^-_n$, since the claim fails for an $\in$-minimal model $\langle M,\in,U\rangle$ of $\ZFC^-_n$. For the same reason, the result can fail if $n=0$. The argument above does not work if the ultrapower by $U$ is ill-founded. In this case, the strong well-order $\lhd_U$ might be ill-founded for formulas in the language with $U$. Thus, it is not clear whether the above lemma holds in the case where $U$ is not good. If $U$ is not good, we can however take, for any set $A\in M$, a substructure of $\la M,\in, U\ra$ with a $\Delta_1$-definable true well-order, namely, $L_\alpha[A,U]$, where $\alpha$ is the height of the model. Indeed, replacing a model $\la M,\in, U\ra$ by $\la L_\alpha[A,U],\in,U\cap L_\alpha[A,U]\ra$ will prove useful in other ways as well. First though we have to verify that this move preserves the theory.

Suppose that $M$ is a weak $\kappa$-model and $U$ is an $M$-ultrafilter such that $$\la M,\in,U\ra\models\KP_n(\ZFC^-_n),$$ for $n\geq 1$, and $\alpha=\Ord^M$. Let $A\in M$ be a subset of $\kappa$. Consider the model $\la L_\alpha[A,\bar U],\in,\bar U\ra$, where $U\cap L_\alpha[A,U]$.
\begin{lemma}\label{lem:LU}
$\la L_\alpha[A,\bar U],\in, \bar U\ra\models\KP_{n}(\ZFC^-_n)$.
\end{lemma}
\begin{proof}
Observe that the $L[A,U]$ construction (up to $\alpha$) can be carried out in the structure $\la M,\in, U\ra$ by $\Sigma_1$-recursion.

First, suppose that $\la M,\in, U\ra\models\KP_n$. $\Sigma_0$-separation clearly holds in $\la L_\alpha[A,\bar U],\in, \bar U\ra$, so it suffices to verify $\Sigma_n$-collection. Suppose that for some $\Pi_{n-1}$-formula $\varphi(x,a)$, $$\la L_\alpha[A,\bar U],\in, \bar U\ra\models\forall x\in a\,\exists y\,\varphi(x,y,a).$$ For every $x\in a$, let $\alpha_x$ be the least ordinal such that $x\in L_{\alpha_x}[A,\bar U]$ and let $\beta_x$ be least ordinal such that $L_{\beta_x}[A,\bar U]$ has some $y$ witnessing $\varphi(x,y,a)$. Let $a\in L_\lambda[A,\bar U]$. Let $\varphi^*(x,y,a)$ be the formula $\varphi(x,y,a)$ relativized to $L_\alpha[A,\bar U]$, that is, for every unbounded quantifier we add the assertion that the variable is in $L[A,U]$. Since the assertion $x\in L[A,\bar U]$ is $\Sigma_1$ over $\la M,\in, U\ra$, the formula $\varphi^*(x,y,a)$ is $\Pi_{n-1}$. Next, observe that in $\la M,\in, U\ra$, the function $f:\lambda\to\alpha$ defined by $f(\xi)=0$ if $\xi\neq \alpha_x$ for any $x\in a$ and otherwise $f(\alpha_x)=\beta_x$ is $\Sigma_n$-definable. Thus, by $\Sigma_n$-collection in $\la M,\in, U\ra$, there is some $\beta<\alpha$ such that the range of $f$ is contained in $\beta$. It follows that $L_\beta[A,U]$ is a collecting set.

Next, suppose that $\la M,\in, U\ra\models\ZFC^-_n$. We already showed that $\la L_\alpha[A,\bar U],\in, \bar U\ra$ satisfies $\Sigma_n$-collection. Thus, it suffices to verify $\Sigma_n$-separation. Fix a $\Sigma_n$-formula $\varphi(x,b):=\exists y\,\psi(y,x,b)$, where $\psi(x,y,b)$ is $\Pi_{n-1}$, and a set $C\in L_\alpha[A,\bar U]$. The set $$\bar C=\{c\in C\mid \la L_\alpha[A,\bar U],\in, \bar U\ra\models\varphi(c,b)\}$$ exists in $M$ by $\Sigma_n$-separation in $\la M,\in, U\ra$. For every $c\in \bar C$, there is $y\in L_\alpha[A,U]$ such that $\la L_\alpha[A,U],\in, U\cap L_\alpha[A,U]\ra\models\psi(y,c,b)$. Thus, by $\Sigma_n$-collection in $\la M,\in, U\ra$, there is an ordinal $\beta$ such that $L_\beta[A,U]$ already has all the witnesses $y$ for $c\in\bar C$. It follows that we can replace the formula $\varphi(x,b)$ by the formula $\exists y \in L_\beta[A,U]\,\psi(y,x,b)$, in verifying separation for the set $C$ and the formula $\varphi(x,b)$ in $\la L_\alpha[A,U],\in, U\cap L_\alpha[A,U]\ra$. Since $\exists y \in L_\beta[A,U]\,\psi(y,x,b)$ is equivalent to a $\Pi_{n-1}$-formula by $\Sigma_n$-collection in $\la L_\alpha[A,U],\in, U\cap L_\alpha[A,U]\ra$, we have separation for it by $\KP_n$, which implies $\Sigma_{n-1}$-separation by Lemma~\ref{KP implies sep}.
\end{proof}
We can ensure that $L_\alpha[A,\bar U]$ is a weak $\kappa$-model by taking $\bar A$, instead of $A$, that codes $A$ and $V_\kappa$. However, if $M$ is simple, it is not necessarily the case that $L_\alpha[A,\bar U]$ is simple. In this case, letting $\beta=(\kappa^+)^{L_\alpha[A,\bar U]}$, we will be able to replace $L_\alpha[A,\bar U]$ by the simple weak $\kappa$-model $L_\beta[A,\bar U]$.
\begin{proposition}
$\la L_{\beta}[A,\bar U],\in,\bar U\cap L_{\beta}[A,\bar U]\ra\models\KP_n(\ZFC^-_n)$.
\end{proposition}
\begin{proof}
First, let's argue that $L_{\beta}[A,\bar U]=H_{\kappa^+}^{L_\alpha[A,\bar U]}$. Fix a set $B\in L_\alpha[A,\bar U]$ whose transitive closure $\mathrm{TCl}(B)$ has size at most $\kappa$. Fix some limit ordinal $\gamma$ such that $\mathrm{TCl}(B)\in L_\gamma[A,\bar U]$. In $L_\alpha[A,\bar U]$, let $$\la X,\in, \bar U\cap X\ra\prec \la L_\gamma[A,\bar U],\in,\bar U\cap L_\gamma[A,\bar U]\ra,$$ with $|X|=\kappa$, $\kappa+1\subseteq X$, and $A,\mathrm{TCl}(B)\in X$. A $\Sigma_1$-recursion suffices to construct $X$ because truth in $\la L_\gamma[A,\bar U],\in, \bar U\cap L_\gamma[A,\bar U]\ra$ is $\Delta_1$-definable. Note that $U^*=\bar U\cap X$ is a set because $X$ has size $\kappa$ in $L_\alpha[A,\bar U]$. Let $\pi:X\to N$ be the Mostowski collapse. Note that $\pi(\mathrm{TCl}(B))=\mathrm{TCl}(B)$. Since $\pi$ fixes subsets of $\kappa$, we have that $\pi$ is an isomorphism between $\la X,\in, U^*\ra$ and $\la N,\in, U^*\ra$. It follows that $N=L_{\bar\beta}[A,\bar U]$ for some $\bar\beta$ of size $\kappa$ in $L_\alpha[A,\bar U]$. Thus, $B\in L_{\beta}[A,\bar U]$.

If $\beta=\alpha$, then we are done. So we can assume that $\beta<\alpha$, and in particular, that $L_\beta[A,\bar U]$ is a set in $L_\alpha[A,\bar U]$. Since $L_\beta[A,\bar U]=H_{\kappa^+}^{L_\alpha[A,\bar U]}$, $\la L_{\beta}[A,\bar U],\in,\bar U\cap L_{\beta}[A,\bar U]\ra$ has the same amount of comprehension as $\la L_\alpha[A,\bar U],\in,\bar U\ra$. Next, we will argue that $\la L_\beta[A,\bar U],\in,\bar U\cap L_\beta[A,U]\ra$ satisfies full collection. Suppose that collection fails in \hbox{$\la L_\beta[A,\bar U],\in, \bar U\cap L_\beta[A,\bar U]\ra$}. We can assume without loss of generality that it fails for a formula $\forall \gamma\in\kappa\,\exists y\,\psi(\gamma,y)$. We can use the failure to obtain an injection from $\kappa$ into $\beta$ that is definable over  $\la L_\beta[A,\bar U],\in,\bar U\cap L_\beta[A,\bar U]\ra$. We define the injection by mapping $\xi$ to the least $\eta_\xi$ such that $\exists y\in L_{\eta_\xi}[A,\bar U]\,\psi(\xi,y)$ holds. Now, we use $\Sigma_1$-collection in $\la L_\alpha[A,\bar U],\in,\bar U\ra$ to construct a set consisting of bijections $f_\xi:\kappa\to\eta_\xi$. But this set contradicts that $\beta=(\kappa^+)^{L_\alpha[A,\bar U]}$.
\end{proof}
The following lemmas will prove useful in later arguments.
\begin{lemma}
\label{Sigma n uniformisation}
If $\la L_\alpha[A],A\ra\models\KP_{n-1}$ for some set $A$ and $n\geq 1$, then it has $\Sigma_n$-definable Skolem functions for $\Sigma_n$-formulas.
\end{lemma}
\begin{proof}
Suppose that $\varphi(x,y)$ is the formula $\exists z\ \psi(x,y,z)$, where $\psi(x,y,z)$ is a $\Pi_{n-1}$-formula.
Let $\gamma$ be least such that $L_\gamma[A]$ contains witnesses $y$ and $z$ for $\psi(x,y,z)$. Then let $y$ be the $<_{L[A]}$-least in $L_\gamma[A]$ such that there is a witness $z\in L_\gamma[A]$ for which $\psi(x,y,z)$ holds. Since $\la L_\alpha[A],A\ra\models\KP$, we can define the $L[A]$-hierarchy. Since $\la L_\alpha[A],A\ra\models\KP_{n-1}$, the statement that $\psi(x,y,z)$ fails for all $<_{L[A]}$-smaller $z$ is a $\Sigma_{n-1}$-statement.
\end{proof}
As a consequence of having $\Sigma_n$-definable Skolem functions for $\Sigma_n$-formulas, we have that the $\Sigma_n$-Skolem hull (over any collection of parameters) taken using these functions is $\Sigma_n$-elementary in $\la L_\alpha[A],A\ra$.

The following proof is based on an argument of Philip Welch which shows that the least admissible ordinal $\alpha>\omega_1$ has uncountable cofinality.

\begin{lemma}\label{lem:Sigma_nDefinableSubstructure}
Suppose that $M$ is a simple weak $\kappa$-model and $U$ is an $M$-ultrafilter such that $M=L_\alpha[A,U]$ for some $A\subseteq\kappa$ and $\la M,\in, U\ra\models\KP_n$. Then $\la M,\in, U\ra$ has a transitive $\Sigma_n$-elementary substructure $\la \bar M,\in, U\cap\bar M\ra$ such that $\bar M$ is a $\kappa$-model.
\end{lemma}
\begin{proof}
We can assume without loss of generality that $A$ codes $V_\kappa$.  Let $\la \bar M,\in, A,U\cap \bar M\ra$ be the $\Sigma_n$-Skolem hull of $\kappa+1$, using the $\Sigma_n$-definable Skolem functions (from Lemma~\ref{Sigma n uniformisation}) in $\la L_\alpha[A,U],\in, A,U\ra$. By Lemma~\ref{Sigma n uniformisation}, we have $$\la \bar M,\in, A,U\cap \bar M\ra\prec_{\Sigma_n}\la L_\alpha[A,U],\in,A,U\ra.$$ Since every set in $L_\alpha[A,U]$ has size $\kappa$, it follows that $\bar M$ is transitive. Thus, by Lemma~\ref{lem:nElemImpliesKPn}, $$\la \bar M,\in, A,U\cap \bar M\ra\models\KP_n.$$ By $\Sigma_n$-elementarity, we have $\bar M=L_{\bar\alpha}[A,U]$ for some $\bar\alpha\leq\alpha$ and the structure $$\la L_{\bar\alpha}[A,U],\in,U\cap L_{\bar\alpha}[A,U]\ra$$ has the additional property that every element is $\Sigma_n$-definable using parameters from $\kappa+1$.

We claim that $\bar M^{<\kappa}\subseteq \bar M$.
To see this, suppose that $\vec{x}=\{ x_i \mid i<\xi\}$ is a sequence of elements of $\bar M$ for some $\xi<\kappa$. Let each $x_i$ be definable in $\la L_{\bar\alpha}[A,U],\in,A,U\cap L_{\bar\alpha}[A,U]\ra$ by the $\Sigma_n$-formula $\varphi_i$ using parameters $\nu_i<\kappa$ and $\kappa$. Since $V_\kappa\in\bar M$, the sequences $\{ \nu_i\mid i<\xi\}$ and \hbox{$\{\varphi_i\mid i<\xi\}$} are both in $\bar M$. Thus, using the $\Sigma_n$-definable $\Sigma_n$-truth predicate $\Tr_n^{U}(x)$ (which exists by Lemma~\ref{truth pred for simple weak}), the sequence $\{ x_i\mid i<\xi\}$ is $\Sigma_n$-definable over \hbox{$\la L_{\bar\alpha}[A,U],\in,A,U\cap L_{\bar\alpha}[A,U]\ra$}. By $\Sigma_n$-collection, there must be some $\beta<\bar\alpha$ such that all $x_i$ are in $L_\beta[A,U]$. At this point, we would be done if we had $\Sigma_n$-separation, but we have only $\Delta_n$-separation available. Thus, we need to do some more work to reduce the complexity of the formula defining the sequence $\{x_i\mid i<\xi\}$.

For each $i<\xi$, let $\varphi_i(x,\nu_i,\kappa):=\exists y\,\psi_i(x,\nu_i,\kappa,y)$, where $\psi$ is a $\Pi_{n-1}$-formula. Next, we use $\Sigma_n$-collection on the formula $$\forall i<\xi\,\exists y\,\Tr_n^{U}(\la \exists y\,\psi_i(x,z,w,y),x_i,\nu_i,\kappa)$$ to obtain a set $Y$ containing, for each $i<\xi$, a witness $y_i$ such that $\psi_i(x_i,\nu_i,\kappa,y_i)$ holds. Using the set $Y$, we can now reduce the complexity of the definition of the sequence to a $\Pi_{n-1}$-formula.

\end{proof}
Thus, in particular, we will be able to replace a model $\la M,\in, U\ra\models\KP_n$ with a $\kappa$-model with the same properties without an increase in consistency strength.


The next result will allow us to separate large cardinal notions defined using $\ZFC^-_n$ and $\KP_n$ in Section \ref{sec:hierarchy}. 

\begin{lemma}\label{lem:KP_nImpliesKappaModel}
Suppose that $M$ is a simple weak $\kappa$-model and $U$ is an $M$-ultrafilter such that $\la M,\in, U\ra\models\ZFC^-_n$. Then for every set $A\in M$, there is a weak $\kappa$-model $N\in M$ with $A\in N$ such that $\la N,\in, U\cap N\ra\models\KP_n$.
\end{lemma}
\begin{proof}
By shrinking the model, if necessary, we can assume without loss of generality that $M=L_\alpha[A,U]$. By Lemma~\ref{lem:Sigma_nDefinableSubstructure}, $\la M,\in, U\ra$ has a transitive $\Sigma_n$-elementary substructure \hbox{$\la \bar M,\in, U\cap\bar M\ra$} such that $\bar M$ is a $\kappa$-model. If $\bar M\in M$, then we are done. Otherwise, $\bar M=M$. Observe that for any set $B\in M$, the set $$T_B=\{\la\varphi,b\ra\mid b\in B\text{ and }\Tr_n^U(\varphi,b)\}$$ is in $M$ by $\Sigma_n$-separation. Also, using $\Sigma_n$-collection and the set $T_B$, $M$ has a set $B^*$ such that for every $\Sigma_n$-formula $\exists y\,\psi(x,y)$ and $b\in B$ such that $\la M,\in, U\ra\models\exists y\,\psi(y,b)$, there is $y_b\in B^*$ such that $\la M,\in,U\ra\models\psi(b,y_b)$. Let $M_0\in M$ be any weak $\kappa$-model with $A\in M_0$. Let $M_1=\mathrm{TCl}(M_0^*)$, and more generally, given $M_n$, let $M_{n+1}=\mathrm{TCl}(M_n^*)$. Let $N=\bigcup_{n<\omega} M_n$, and observe that $N\in M$ by closure. By construction $N$ is transitive and $$\la N,\in, U\cap N\ra\prec_{\Sigma_n}\la M,\in, U\ra,$$ and so $\la N,\in, U\cap N\ra\models\KP_n$ by Lemma~\ref{lem:nElemImpliesKPn}.
\end{proof}


\section{Baby versions}
\label{section: babym}

We now define $n$-baby measurable cardinals by slightly simplifying the notions of Bovykin and McKenzie \cite{BovykinMcKenzie:RamseyLikeNFUM} and we further introduce some related large cardinal notions.

\begin{definition}
\label{def baby meas}
A cardinal $\kappa$ is:
\begin{enumerate}
\item
\label{def baby meas veryweak}
\emph{faintly $n$-baby measurable} if every $A\subseteq\kappa$ is an element of a weak $\kappa$-model $M$ for which there is an $M$-ultrafilter such that $\la M,\in,U\ra\models\ZFC^-_n$.



\item
\label{def baby meas weak}
 \emph{weakly $n$-baby measurable} if \eqref{def baby meas veryweak} holds and in addition, $U$ is good.
%
%

\item
\label{def baby meas main}
\emph{$n$-baby measurable} if \eqref{def baby meas weak} holds and in addition, $M$ is a $\kappa$-model.
\item 
\label{def [n] baby meas}
\emph{{$[n]$-baby measurable}} if \eqref{def baby meas main} holds but with $\KP_n$ instead of $\ZFC^-_n$.

\end{enumerate}
Notions (1)-(3) have variants where $\ZFC^-_n$ is replaced by $\ZFC^-$. We then omit $n$ from the notation. We will see from Lemma~\ref{lem:noFaintWeakKP} below that the concepts of faintly and weakly $[n]$-baby measurable cardinals are equivalent to \eqref{def [n] baby meas}. 
\end{definition}

We can further assume that all weak $\kappa$-models involved in the definitions above are simple. While the $M$-ultrafilter in the definition of ($n$-)faintly baby measurable cardinals need not be good, the notion will still turn out to be quite strong.

The faintly ($n$-)baby measurable cardinals are analogues of completely ineffable cardinals, but with stronger $M$-ultrafilters. The weakly ($n$-)baby measurable cardinals are analogues of $1$-iterable cardinals, but with stronger $M$-ultrafilters.
 The ($n$-)baby measurable cardinals are analogues of strongly Ramsey cardinals, but with stronger $M$-ultrafilters.

We will vary the closure condition to study the gap between weakly $n$-baby measurable and $n$-baby measurable cardinals.
For this purpose, it is useful to discard the requirement that $M$ be transitive and require that $M$ is elementary in some $H_\theta$ as in \cite{HolySchlicht:HierarchyRamseyLikeCardinals}.

\begin{definition}
Suppose that $\kappa$ is a cardinal and $\omega\leq\alpha\leq\kappa$ is a regular cardinal.\footnote{We use the notation $\alpha$ for cardinals following \cite{HolySchlicht:HierarchyRamseyLikeCardinals} with the motivation that these notions can be characterissed by games of length $\alpha$ and they could thus, in principle, be generalized to ordinals.}
$\kappa$ is called \emph{$(\alpha,n)$-baby measurable} if for every $A\subseteq\kappa$ and arbitrarily large $\theta$ there is a ${<}\alpha$-closed basic weak $\kappa$-model $M\prec H_\theta$ with $A\in M$ and a (good) $M$-ultrafilter $U$ such that $\la H_{\kappa^+}^M,\in,U\ra\models\ZFC^-_n$.

$\kappa$ is called \emph{faintly $(\omega,n)$-baby measurable} if the $M$-ultrafilter is not required to be good.

$\kappa$ is called \emph{$\alpha$-baby measurable} if we replace $\ZFC^-_n$ by $\ZFC^-$ in the above definition.
\end{definition}
Note that we only need to explicitly state that the $M$-ultrafilter $U$ is good in the case that $\alpha=\omega$. Further variants such as \emph{$({<}\alpha,n)$-baby measurable} are used below with the obvious meaning.

The $(\alpha,n)$-baby measurable ($\alpha$-baby measurable) cardinals are  analogues of the $\alpha$-Ramsey cardinals, but with stronger $M$-ultrafilters. 

The reflective $(\alpha,n)$-baby measurable ($\alpha$-baby measurable) cardinals strengthen this notion by requiring strong $\Sigma_n$-reflection (elementary reflection) in the structure $\la H_{\kappa^+},\in,U\ra$.
\begin{definition}
\label{def: reflective} Suppose that $\kappa$ is a cardinal and $\omega\leq\alpha\leq\kappa$ is a  regular cardinal.
$\kappa$ is called \emph{reflective $(\alpha,n)$-baby measurable} if for every $A\subseteq\kappa$ and arbitrarily large $\theta$ there is a ${<}\alpha$-closed basic weak $\kappa$-model model $M\prec H_\theta$ with $A\in M$ and a (good) $M$-ultrafilter $U$ such that for every $B\subseteq\kappa$ in $M$, there is a $\kappa$-model $\bar M\in M$ with $B\in\bar M$ such that $\la \bar M,\in,U\cap\bar M\ra\prec_{\Sigma_n} \la H_{\kappa^+}^M,\in,U\ra$ and $\la \bar M,\in,U\cap\bar M\ra\models\ZFC^-_n$.

$\kappa$ is called \emph{faintly reflective $(\omega,n)$-baby measurable} if the $M$-ultrafilter is not required to be good.

$\kappa$ is called \emph{reflective $\alpha$-baby measurable} if we replace $\Sigma_n$-elementarity by full elementarity in the above definition.
\end{definition}

Note that we get  $\la H_{\kappa^+}^M,\in,U\ra\models\ZFC^-_n$ for free by Lemma~\ref{lem:ElemSubstructuresImplyZFCn}.
These notions will have analogous game theoretic definitions similar to that of $\alpha$-Ramsey cardinals. We define these games in  Section~\ref{section: games}.

\section{The hierarchies}\label{sec:hierarchy}

In this section, we will show where the various (non-game related) notions defined above fit into the large cardinal hierarchy.

First, we show that there are no faint or weak versions of the $[n]$-baby measurable cardinals because this is one of those rare instances where the closure on the model (and hence the well-foundedness of the ultrapower) comes for free.
\begin{lemma}\label{lem:noFaintWeakKP}
Suppose that $M$ is a weak $\kappa$-model and $U$ is an $M$-ultrafilter such that\break $\la M,\in, U\ra\models\KP_n$ for some $n\geq 0$. Then for every $A\subseteq\kappa$ in $M$, there is a $\kappa$-model $\bar M\subseteq M$ with $A\in\bar M$ such that $\la \bar M,\in U\cap\bar M\ra\models\KP_n$.
\end{lemma}
\begin{proof}
By a sequence of lemmas at the end of Section~\ref{section: preliminaries}, we can assume without loss of generality that $M=L_\alpha[A,U]$ with $\alpha=(\kappa^+)^{L_\alpha[A,U]}$. Thus, by Lemma~\ref{lem:Sigma_nDefinableSubstructure}, there is a $\kappa$-model $\bar M\subseteq M$ with $A\in\bar M$ such that $\la \bar M,\in, U\cap \bar M\ra\models\KP_n$.
\end{proof}
Since the theories $\ZFC^-_0$ and $\KP_0$, and $\KP_1$ are all the same, the following large cardinals are all equivalent:
\begin{enumerate}
\item $[0]$-baby measurable cardinals
\item $[1]$-baby measurable cardinals
\item faintly $0$-baby measurable cardinals
\item weakly $0$-baby measurable cardinals
\item $0$-baby measurable cardinals
\end{enumerate}

Next, we show that the faintly $0$-baby measurable cardinals, which are like 0-iterable cardinals with just the additional assumption of $\Delta_0$-collection, are stronger than Ramsey cardinals, and hence cannot exist in $L$.

\begin{theorem}
\label{th:0babymeasurable}
If $\kappa$ is faintly $0$-baby measurable, then $\kappa$ is a strongly Ramsey limit of strongly Ramsey cardinals.
\end{theorem}
\begin{proof}
Suppose that $M$ is a simple weak $\kappa$-model and $U$ is an $M$-ultrafilter with $\la M,\in,U\ra\models\KP$.
Let $N$ be the, not necessarily well-founded, ultrapower of $M$ with $U$.

We first show that $\kappa$ is strongly Ramsey in $M$. Fix any $A\subseteq\kappa$  in $M$.
We construct a continuous increasing sequence $\langle M_\alpha \mid \alpha\leq \kappa \rangle$ with $A\in M_0$ by a $\Sigma_1$-recursion in $\langle M,\in,U\rangle$. Note that the $\Sigma_1$-recursion scheme holds in $\langle M,\in,U\rangle$ by Lemma~\ref{KP implies sep}.
Let $M_0\in M$ be arbitrary with $A\in M_0$.
For even $\alpha<\kappa$, let $M_{\alpha+1}:= M_\alpha \cup \{\{U\cap M_\alpha\}\}$.
For odd $\alpha<\kappa$, let $M_{\alpha+1}\prec M$ be a $\kappa$-model from the perspective of $M$, with $M_\alpha\in M_{\alpha+1}$, given by a $\Delta_1$-definable choice function for the relation $R$ that consists of all pairs $(x,y)\in M$,
where $x\in y$ and $(y,\in)$ is an elementary substructure of $\langle M,\in\rangle$ that is closed under sequences of length less than $\kappa$.
Such a choice function exists by Lemma \ref{Delta1ClassChoice} since the relation $R$ exists in $N$.
For all limits $\alpha\leq\kappa$, let $M_\alpha:=\bigcup_{\bar{\alpha}<\alpha} M_{\bar{\alpha}}$.
In $M$, we have that $M_\kappa$ is a $\kappa$-model with $A\in M_\kappa$ and $U\cap M_\kappa$ is a weakly amenable $M_\kappa$-ultrafilter.
Since $A$ was arbitrary, $\kappa$ is strongly Ramsey in $M$. By Proposition~\ref{prop:internalKappaModel}, $M_\kappa$ is really a $\kappa$-model, and so, since $M$ was arbitrary, $\kappa$ is strongly Ramsey. Moreover, $\kappa$ is strongly Ramsey in $N$, and hence it is a limit of strongly Ramsey cardinals by elementarity.
\end{proof}

Thus, in particular, the existence of faintly $0$-baby measurable cardinals already implies that $0^\#$ exists. In fact, faintly $0$-baby measurable cardinals have higher consistency strength.

\begin{theorem}
\label{th:0bmalpharam}
If there is a faintly $0$-baby measurable cardinal, then there is a model of $\ZFC$ with a $\kappa$-Ramsey cardinal $\kappa$.
\end{theorem}
\begin{proof}
Suppose that $\la M,\in,U\ra\models\KP$ where $M$ is a simple weak $\kappa$-model and $U$ is an $M$-ultrafilter, and let $N$ be the (not necessarily well-founded) ultrapower of $M$ by $U$.
It suffices to show that $\kappa$ is $\kappa$-Ramsey in $V_{j(\kappa)}^N$.
Otherwise the challenger has a winning strategy $\sigma$ in the game $\mathsf{RamseyG}_{\kappa}^{\kappa^+}(\kappa)$ in $V_{j(\kappa)}^N$ (equivalently in $N$). The element $\sigma$ of $N$ is represented by some equivalence class $[S]_U$.
Since the ultrapower is $\kappa$-powerset preserving, $\sigma$ is a map from sequences of elements of $M$ to elements of $M$ and all proper initial segments of runs are in $M$.
By coding, we can assume that elements of $\sigma$ are coding subsets of $\kappa$ (elements of $\mathcal C$ as defined in the proof of Lemma~\ref{lem:definableWellOrder}), which code pairs $(s,y)$, where $s$ is a sequence of plays by the judge and $y$ is the response to the last move. In $M$, we will use $U$ to play against $\sigma$ and build a run of the game $\mathsf{RamseyG}_\kappa^{\kappa^+}(\kappa)$ won by the judge, contradicting that $\sigma$ was a winning strategy.
In detail, the judge responds to the model $M_\alpha$ played by the challenger in round $\alpha$ by playing $U\cap M_\alpha$.
The challenger plays according to $\sigma$. The assertion that the challenger's response to $x$ is $y$ according to $\sigma$ translates to asking whether the set $\{\alpha<\kappa\mid f_A(\xi)\in S(\xi)\}$ is an element of $U$, where $A$ is the coding subset for pair $(s,y)$ with $s$ being the sequence of the judge's moves so far ending in $x$, and  $f_A(\xi)=A\cap \xi$ for every $\xi<\kappa$. The above recursion is thus $\Delta_1$, so has a solution $M$.
\end{proof}

Note that the least $0$-baby measurable cardinal $\kappa$ cannot be super Ramsey because being $0$-baby measurable is a property of $H_{\kappa^+}$ and would therefore be reflected down if $\kappa$ was super Ramsey.

Next, we show that surprisingly the hierarchies of the variants of $n$-baby measurable cardinals are intertwined both in the case of $\ZFC^-_n$ and that of $\KP_n$.

\begin{theorem}
For any $n\geq 1$, a faintly $n$-baby measurable $\kappa$ is a limit of $[n]$-baby measurable cardinals.
\end{theorem}
\begin{proof}
Fix $A\subseteq\kappa$ and choose a weak $\kappa$-model $M$, with $A\in M$, for which there is an $M$-ultrafilter $U$ such that $\la M,\in,U\ra\models\ZFC^-_n$. By the sequence of lemmas at the end of Section~\ref{section: preliminaries}, we can assume without loss of generality that $M=L_\alpha[A,U]$ and $\alpha=(\kappa^+)^{L_\alpha[A,U]}$. By Lemma~\ref{lem:KP_nImpliesKappaModel}, there is a weak $\kappa$-model $N\in M$ with $A\in N$ such that $\la N,\in, U\cap N\ra\models\KP_n$. Since $U$ is an $M$-ultrafilter and $N$ is a set in $M$, $U\cap N$ is countably complete, and hence good. It follows $\kappa$ is $[n]$-baby measurable.

An analogous argument also shows that $\kappa$ is $[n]$-baby measurable in the, not necessarily well-founded, ultrapower of $M$ by $U$. Thus, $\kappa$ is a limit of $[n]$-baby measurable cardinals.
\end{proof}

\begin{theorem}\
\label{babym hierarchy}
For any $n\geq 1$, a faintly $[n+1]$-baby measurable $\kappa$ is an $n$-baby measurable limit of $n$-baby measurable cardinals.
\end{theorem}
\begin{proof}
Fix $A\subseteq\kappa$ and choose a weak $\kappa$-model $M$, with $A\in M$, for which there is an $M$-ultrafilter $U$ such that $\la M,\in,U\ra\models\KP_{n+1}$. We can assume without loss of generality that $M=L_\alpha[A,U]$ and thus has a $\Delta_1$-definable true well-order. Thus, the proof of Lemma~\ref{lem:ZFCnImpliesElemSubstructures} shows that there is a $\kappa$-model $\bar M\in M$ such that $\la \bar M,\in, U\cap\bar M\ra\models\ZFC_n^-$. Since the subset $A$ was arbitrary, we have verified that $\kappa$ is $n$-baby measurable. Since $\kappa$ is $n$-baby measurable in the ultrapower $N$, we also have that $\kappa$ is a limit of $n$-baby measurable cardinals.
\end{proof}

\begin{theorem}
\label{weakly vs faintly nmb}
For any $n\geq 1$, a weakly $n$-baby measurable cardinal $\kappa$ is a limit of faintly $n$-baby measurable cardinals.
\end{theorem}
\begin{proof}
Suppose that $n\geq 1$, $M$ is a weak $\kappa$-model, and $U$ is a good $M$-ultrafilter such that $\langle M,\in,U\rangle\models\ZFC^-_n$.
We claim that $\kappa$ is faintly $n$-baby measurable in $N$, and hence a limit of faintly $n$-baby measurable cardinals.

Recall the structure $\la M,\in, U\ra$ has a $\Sigma_n$-definable truth predicate $\Tr_n^U(x)$ for $\Sigma_n$-formulas (by Lemma~\ref{truth pred for simple weak}). For the future, we fix a canonical $\Sigma_n$-formula $\exists y\,\theta(x,y)$ defining $\Tr_n^U(x)$. Observe that for any set $B\in M$, the set $$T_B=\{\la\varphi,b\ra\mid b\in B\text{ and }\Tr_n^U(\la \varphi,b\ra)\}$$ is in $M$ by $\Sigma_n$-separation. Also,  using $\Sigma_n$-collection and the set $T_B$, $M$ has a set $B^*$ such that for every $\Sigma_n$-formula $\exists y\,\psi(x,y)$ and $b\in B$, there is $y_b\in B^*$ such that $\la M,\in,U\ra\models\psi(b,y_b)$.

We say that a formula is in \emph{$\Sigma_n$-} or \emph{$\Pi_n$-normal form} if it has a block of at most $n$ alternating quantifiers preceding a $\Sigma_0$-formula.
We can assume that the $\Sigma_n$-formula $\exists y\,\theta(x,y)$ defining $\Tr_n^U(x)$ is in normal form.
We denote the standard way of converting a negation $\neg \varphi$ of a $\Sigma_n$-formula $\varphi$ in normal form to a logically equivalent $\Pi_n$-formula in normal form by $\neg^*\varphi$.
We denote the standard way of converting a conjunction $\varphi\wedge\psi$ of two $\Sigma_n$- or $\Pi_n$-formulas $\varphi$ and $\psi$ into a logically equivalent $\Sigma_n$- or $\Pi_n$-formula in normal form by $\varphi\wedge^*\psi$.

Given a structure $\bar M$, let us call a set $T\subseteq \bar M$ a \emph{$\Sigma_n$-pseudo truth predicate} if it consists of pairs $\la \varphi,\bar a\ra$, where $\varphi$ is a $\Sigma_n$- or $\Pi_n$-formula in normal form and $\bar a\in \bar M$ is a finite tuple, and $T$ satisfies the following conditions.
We will simplify the notation by writing $\varphi(\bar a)\in T$ instead of $\la \varphi,\bar a\ra\in T$.
\begin{enumerate}
\item For every $\Sigma_0$-formula $\varphi(\bar x)$ and finite tuple $\bar a\in \bar M$, $\varphi(\bar a)\in T$ if and only if $\bar M\models\varphi(\bar a)$.
\item For every $\Sigma_n$-formula $\varphi(\bar x)$ and finite tuple $\bar a\in \bar M$,
either $\varphi(\bar a)\in T$ or $\neg^*\varphi\in T$.
\item
If $\forall x\,\psi(x,\bar a)\in T$ for some finite tuple $\bar a\in \bar M$, then $\psi(b,\bar a)\in T$ for every $b\in \bar M$.
\end{enumerate}
Note that what separates $T$ from an actual truth predicate is the omission of the existential quantifier condition that witnesses for existential statements are provided.

Take any $A\subseteq\kappa$ in $M$. We construct a tree $\mathcal T$ of height $\omega$ in $N$ whose branches union up to produce structures  $\la \bar M,\in, \bar U\ra\models\ZFC_n^-$, where $\bar M$ a weak $\kappa$-model with $A\in\bar M$ and $\bar U$ is an $M$-ultrafilter on $\kappa$. We will then proceed to show that $\mathcal T$ has a branch in $V$ and hence also in $N$ by the absoluteness of well-foundedness.

A node on level $m$ of $\mathcal T$ is going to be a sequence $\{\la M_i,\in, U_i, T_i\ra\mid i<m\}$ satisfying the following conditions:
\begin{enumerate}
\item
$A\in M_0$.
\item Each $M_i$ is a weak $\kappa$-model, $U_i$ is an $M_i$-ultrafilter, and $T_i$ is a $\Sigma_n$-pseudo truth predicate for the structure $\la M_i,\in, U_i\ra$.
\item For $i<j<m$, $M_i\subseteq M_j$, $U_i\subseteq U_j$, and $T_i\subseteq T_j$.
\item
\label{weakly vs faintly nmb - amenable}
For each $i<m$, $U_i\cap M_i\in M_{i+1}$.
\item For each $j<m-1$, $\Sigma_n$-formula $\exists y\,\psi(x,y)\in T_j$, and $a\in M_j$, there is some $b\in M_{j+1}$ such that $\psi(a,b)\in T_{j+1}$.
\item
For each $j<m-1$, $M_{j+1}$ contains elements $X_j\subseteq M_j$ and $Y_j\subseteq M_{j+1}$ with the following properties:
\begin{enumerate}
\item
For each $x\in X_j$,
$\exists y\, [\theta(x,y)\wedge^* (y\in Y_j)]$ is in $T_{j+1}$.
\item For each $x\in M_j\setminus X_j$, $\forall y\,\neg^* \theta(x,y)$ is in $T_{j+1}$.
\end{enumerate}
\end{enumerate}
Suppose that $\{\la M_i,\in,U_i\ra\mid i<\omega\}$ is a branch of $\mathcal T$. Let $\bar M=\bigcup_{i<\omega}M_i$, $\bar U=\bigcup_{i<\omega}U_i$ and $\bar T=\bigcup_{i<\omega}T_i$. Let's argue that $\bar T$ is an actual $\Sigma_n$-truth predicate for $\la \bar M,\in,\bar U\ra$.
We argue by induction on complexity of formulas. Suppose that $\varphi(x)$ is a $\Sigma_0$-formula. We have that $\la \bar M,\in,\bar U\ra\models\varphi(a)$ if and only if $\la M_i,\in,U_i\ra\models\varphi(a)$, where $a\in M_i$, by absoluteness if and only if $\varphi(a)\in T_i\subseteq \bar T$. Next, consider a formula $\exists y\,\psi(x,y)$ such that the hypothesis holds for $\psi(x,y)$ by the inductive assumption. Suppose first that $\la \bar M,\in,\bar U\ra\models\exists y\,\psi(a,y)$. Then $\la \bar M,\in,\bar U\ra\models\psi(b,a)$ for some $b$. By the inductive assumption $\psi(a,b)$ is in $\bar T$ and hence in $T_j$ for some $j$. Thus, $\forall y\,\neg^*\psi(a,y)$ cannot be in $T_j$. It follows that $\exists y\,\psi(a,y)\in T_j\subseteq \bar T$. Finally, suppose that $\exists y\,\psi(a,y)$ is in $\bar T$ and hence in some $T_j$. Then there is $b\in M_{j+1}$ such that $\psi(a,b)$ is in $T_j\subseteq T$. Thus, by our inductive assumption, $\la \bar M,\in,\bar U\ra\models \psi(a,b)$.

We now show that $\la \bar M,\in,\bar U\ra\models\ZFC^-_n$.
Since $\bar{U}$ is amenable to $\bar{M}$ by \eqref{weakly vs faintly nmb - amenable}, $\exists y\,\theta(\langle \varphi,a\rangle,y)$ is equivalent to $\exists x\, \varphi(x, a)$ for any $\Pi_{n-1}$-formula $\varphi(x,y)$ and any $a\in \bar{M}$.
We start with $\Sigma_n$-separation. Fix a $\Sigma_n$-formula $\exists z\,\psi(x,w,z)$ and sets $a,b\in \bar M$. We need to show that the set $X=\{x\in a\mid \exists z\,\psi(x,b,z)\}$ is in $\bar M$. Let $a,b\in M_j$. The set $X_j\in M_{j+1}$ consists of all $x\in M_j$ for which $\exists y\,\theta(x,y)\in T_{j+1}$. Since $\bar T$ is a truth predicate for $\bar M$, the set $X_j$ consists of all $x\in M_j$ for which $\exists y\,\theta(x,y)$ holds in $\bar M$. We can use separation in $\la M_{j+1},\in\ra$ to obtain the set $\{x\in a\mid \la \exists z\,\psi(x,b,z),a\ra\in X_j\}$, which is precisely the required set $X$. Next, let's argue that $\Sigma_n$-collection holds in $\la \bar M,\in,\bar U\ra$. Fix a $\Sigma_n$-formula $\exists y\,\rho(x,y)$ and a set $a\in\bar M$ such that $\la \bar M,\in,\bar U\ra\models\forall x\in a\,\exists y\,\rho(x,y)$. We need to find a set $b\in \bar M$ such that for every $x\in a$, there is $y_x\in b$ such that $\la \bar M,\in,\bar U\ra\models\rho(x,y_x)$. Let $a\in M_j$. The set $Y_j\in M_{j+1}$ consists of all witnesses $y$ for $x\in M_j$ for the formula $\theta(x,y)$. Since $\exists y\,\theta(x,y)$ was the canonical $\Sigma_n$-definition of $\Sigma_n$-truth, we have that $\theta(b,c)$ holds if and only if $b=\la d,\exists z\,\psi(x,z)\ra$ and $\psi(d,c)$ holds. Thus, in particular, the set $Y_j$ already contains all witnesses for $\Sigma_n$-formulas with parameters in $M_j$ true in $\bar{M}$.

It remains to argue that the tree $\mathcal T$ has a branch in $V$. Let $M_0$ be any weak $\kappa$-model in $M$ with $A\in M_0$, let $U_0=M_0\cap U$, and let $T_0=T_{M_0}$, where $T_{M_0}$ denotes the restriction of the $\Sigma_n$-truth predicate for $\la M,\in,U\ra$ to formulas with parameters in $M_0$ constructed earlier using $\Sigma_n$-separation. Let $M_0^*\in M$ contain witnesses for every $\Sigma_n$-formula with parameters in $M_0$. The set $X_0$ exists by $\Sigma_n$-separation applied to the formula $\exists y\,\theta(x,y)$ and $M_0$. The set $Y_0$ is then the collecting set of witnesses for the set $X_0$ and the formula $\exists y\,\theta(x,y)$. So  let $M_1$ be any weak $\kappa$-model in $M$ such that $M_0^*\subseteq M_1$ and $X_0,Y_0\in M_1$. Moreover, let $U_1=U\cap M_1$ and $T_1=T_{M_1}$. We define the structures $\la M_n,\in,U_n,T_n\ra$ for $n\geq 1$ analogously.
\end{proof}

Next, we show that weakly $n$-baby measurable cardinals are weaker than $n$-baby measurable cardinals.

\begin{theorem}
\label{bm stronger than weaklybm}
An $n$-baby measurable cardinal is a limit of weakly $n$-baby measurable cardinals.
\end{theorem}
\begin{proof}
Suppose that $M$ is a $\kappa$-model such that $\la M,\in,U\ra\models\ZFC_n^-$.
We show that $\kappa$ is weakly $n$-baby measurable in the ultrapower $N$ of $M$ by $U$, and hence a limit of weakly $n$-baby measurable cardinals. As we already argued in the proof of Theorem~\ref{weakly vs faintly nmb}, $\ZFC^-_n$ implies that for every set $B$, there is a set $B^*$ containing for every $\Sigma_n$-formula $\varphi(x):=\exists y\psi(x,y)$ and $b\in B$, a witness $y_b$ such that $\psi(b,y_b)$ holds in $\la M,\in, U\ra$.

Fix $A\subseteq \kappa$ in $M$. Let $M_0\in M$ be any weak $\kappa$-model with $A_0\in M_0$. Let $M_0^*$, as above, contain witnesses for all $\Sigma_n$-formulas with parameters from $M_0$ that are true over $\la M,\in, U\ra$. Let $M_1\in M$ be any weak $\kappa$-model with $M_0^*\subseteq M_1$. Given $M_n$, we define $M_{n+1}$ analogously. Let $\bar M_0=\bigcup_{n<\omega} M_n$. Then clearly $\la \bar M_0,\in, U\cap \bar M\ra\prec_{\Sigma_n}\la M,\in, U\ra$. By closure, the sequence $\{M_n\mid n<\omega\}$ is in $M$, and hence so is $\bar M_0$. Next, we repeat the process, starting with $\bar M_0$ instead of $A$, to build a model $\bar M_1\in M$ such that $$\la \bar M_0,\in, U\cap \bar M_0\ra\prec_{\Sigma_n} \la \bar M_1,\in,U\cap \bar M_1\ra\prec_{\Sigma_n}\la M,\in, U\ra.$$ Continuing in this manner, we define a $\Sigma_n$-elementary chain of models $\la \bar M_n,\in, U\cap \bar M_n\ra$, and observe that the sequence must be in $M$ by closure. Let $\bar M=\bigcup_{n<\omega}\bar M_n$. Then, by construction, $\la \bar M,\in, U\cap \bar M\ra$ satisfies $\Sigma_n$-reflection, and so satisfies $\ZFC_n^-$ by Lemma~\ref{lem:ElemSubstructuresImplyZFCn}.
\end{proof}

Note that the above proof only required the model $M$ to be closed under countable sequences.



\section{Games}
\label{section: games}

The distinction between weakly $n$-baby measurable and $n$-baby measurable cardinals (see Proposition \ref{bm stronger than weaklybm}) suggests that one can obtain different large cardinal notions by varying the closure of the models.
We will show that in fact closure properties induce a hierarchy between these two notions using games analogous to the ones for $\alpha$-Ramsey cardinals from \cite{HolySchlicht:HierarchyRamseyLikeCardinals}.
We now describe games with perfect information associated to the $(\alpha,n)$-baby measurable cardinals and their variants.
All games have two players, the challenger and the judge.
Suppose that $\kappa$, $\alpha$ and $\theta$ are regular cardinals with $\omega\leq\alpha\leq\kappa<\theta$.

\begin{definition}
The game $\g^{\theta,n}_\alpha(\kappa)$ proceeds for $\alpha$-many steps. The challenger starts the game and plays a basic $\kappa$-model $M_0\prec H_\theta$. The judge responds by playing a structure $\la N_0,\in, U_0\ra$, where \hbox{$P^{M_0}(\kappa)\subseteq N_0$} is a $\kappa$-model and $U_0$ is an $N_0$-ultrafilter. At stage $\gamma$, the challenger plays a basic $\kappa$-model $M_\gamma\prec H_\theta$ such that
$$\{\la N_\xi,\in,U_\xi\ra\mid\xi<\gamma\}\in M_\gamma,$$
and the judge responds with a structure $\la N_\gamma,\in,U_\gamma\ra$ such that $P^{M_\gamma}(\kappa)\subseteq N_\gamma$ is a $\kappa$-model and $U_\gamma$ is an $N_\gamma$-ultrafilter extending $\Union_{\xi<\gamma}U_\gamma$. After $\alpha$-many steps, let $M=\Union_{\xi<\alpha}M_\xi$ and $U=\Union_{\xi<\alpha}U_\xi$. The judge wins the game if she was able to play for $\alpha$-many steps such that at the end $U$ is a (good) $M$-ultrafilter and $\la H_{\kappa^+}^M,\in,U\ra\models\ZFC_n^-$. Otherwise, the challenger wins.

Let $\mathsf{faint}\g^{\theta,n}_\omega(\kappa)$ be the analogous game where we do not require the $M$-ultrafilter $U$ to be good.
\end{definition}

Note that this game is played as the game $\Ra\bar{\g}^{\theta}_\alpha(\kappa)$, but, while in that game in order to win, the judge needed to ensure that $U$ is a (good) $M$-ultrafilter, here the judge also needs to ensure that $\la H_{\kappa^+}^M,\in,U\ra\models\ZFC^-_n$. Note also that the $M$-ultrafilter $U$ is automatically good by closure for uncountable $\alpha$.

Let's argue that $H^M_{\kappa^+}=\Union_{\xi<\alpha}N_\xi$ is the union of the judge's moves. If $A\in H_{\kappa^+}^M$, then $A$ can be coded by a subset of $\kappa$ in some $M_\xi$, and hence $A\in N_\xi$. If $A\in N_\xi$, then $A\in H_{\kappa^+}$, and $M$ knows this because $M\prec H_\theta$. 

We will characterize $(n,\alpha)$-baby measurable cardinals by the statement that the challenger does not have a winning strategy in the game $\g^{\theta,n}_\alpha(\kappa)$ and faintly $(\omega,n)$-baby measurable cardinals by the statement that the challenger does not have a winning strategy in the game $\mathsf{faint}\g^{\theta,n}_\omega(\kappa)$ .

The next game strengthens the winning requirement that $\la H_{\kappa^+}^M,\in,U\ra\models\ZFC_n^-$ to the requirement that $\Sigma_n$-reflection holds.

\begin{definition}
\label{def: refl game}
The game $\rg^{\theta,n}_\alpha(\kappa)$ is defined just like the game $\g^{\theta,n}_\alpha(\kappa)$, except that the judge has to ensure that (1) $\la N_\gamma,\in, U_\gamma\ra\models\ZFC^-_n$ and (2) $\la N_\xi,\in,U_\xi\ra \prec_{\Sigma_n} \la N_\gamma,\in,U_\gamma\ra$ for all $\xi<\gamma$ in move $\gamma$. The game $\mathsf{faint}\rg^{\theta,n}_\omega(\kappa)$ is defined analogously.
\end{definition}

As before, observe that $H_{\kappa^+}^M=\Union_{\xi<\alpha}N_\xi$.
Moreover, requiring elementarity between the moves ensures that $\la N_\xi,\in,U_\xi\ra \prec_{\Sigma_n} \la H_{\kappa^+}^M,\in,U\ra$ for all $\xi<\gamma$. Note that this already implies $(H_{\kappa^+}^M,\in, U)$ is a model of $\ZFC^-_n$ by Lemma~\ref{lem:ElemSubstructuresImplyZFCn}.

We will characterise reflective $(n,\alpha)$-baby measurable cardinals by the statement that the challenger does not have a winning strategy in $\rg^{\theta,n}_\alpha(\kappa)$ and faintly reflective $(\omega,n)$-baby measurable cardinals by the statement that the challenger does not have a winning strategy in $\mathsf{faint}\rg^{\theta,n}_\alpha(\kappa)$.

Suppose that $\omega\leq\alpha\leq\kappa$ is a regular cardinal.
Let's argue that in the definition of the (reflective) $(\alpha,n)$-baby measurable cardinals, we can strengthen the assumption that every subset of $\kappa$ can be put into the required basic $\kappa$-model to show that, in fact, we can put every set into such a model.

\begin{proposition}
\label{strong version alphanbm}
If $\kappa$ is $(\alpha,n)$-baby measurable, then for every set $A$ and all sufficiently large cardinals $\theta$, there exists a basic weak $\kappa$-model $M\prec H_\theta$ with $A\in M$ for which there is a (good) $M$-ultrafilter $U$ such that $\la H_{\kappa^+}^M,\in,U\ra\models\ZFC^-_n$.
The same strengthening applies to faintly $(\omega,n)$-baby measurable cardinals, reflective $(\alpha,n)$-baby measurable cardinals, and faintly reflective $(\omega,n)$-baby measurable cardinals.
\end{proposition}
\begin{proof}
Suppose that $\kappa$ is $(\alpha,n)$-baby measurable and fix a set $A$.
Suppose towards a contradiction that there is a cardinal $\theta$ with $A\in H_\theta$ for which there exists no basic weak $\kappa$-model as required containing $A$ and elementary in $H_\theta$.
Choose any large enough $H_\nu$ (1) which sees that there is such a counterexample $(\theta,A)$ and (2) for which there is some basic $\kappa$-model $M\prec H_\nu$ satisfying the requirements of $(\alpha,n)$-baby measurability for $\nu$, namely that there is a (good) $M$-ultrafilter $U$ such that $\la H_{\kappa^+}^M,\in,U\ra\models\ZFC^-_n$. By elementarity, $M$ has some counterexample $(\bar\theta,\bar A)$. Let $\bar M=M\cap H_{\bar\theta}$. Then we have $\bar M\prec H_{\bar\theta}$, $\bar A\in\bar M$, and $H_{\kappa^+}^M=H_{\kappa^+}^{\bar M}$. Thus, $U$ is a (good) $\bar M$-ultrafilter such that $\la H_{\kappa^+}^{\bar M},\in,U\ra\models\ZFC^-_n$, but this contradicts that $(\bar\theta,\bar A)$ was a counterexample.
The proof obviously generalizes to the other large cardinal notions.
\end{proof}

Indeed, the above proof shows that in the definition of the (reflective) $(\alpha,n)$-baby measurable cardinals it suffices to assume that for sufficiently large $\theta$, there is a single basic $\kappa$-model $M$ satisfying the requirements.

\begin{lemma}\label{prop:weakWinningStrategyIndependentOfTheta}
The existence of a winning strategy for either player in the game $\g^{\theta,n}_\alpha(\kappa)$ is independent of $\theta>\kappa$. An analogous result holds for the games $\mathsf{faint}\g^{\theta,n}_\omega(\kappa)$, $\rg^{\theta,n}_\alpha(\kappa)$, and $\mathsf{faint}\g^{\theta,n}_\omega(\kappa)$.
\end{lemma}
\begin{proof}
We will prove the result for the game $\g^{\theta,n}_\alpha(\kappa)$.
The proof for the other games is nearly identical.
Fix cardinals $\theta,\rho>\kappa$. Let's argue that if either player has a winning strategy in the game $\g_\alpha^{\theta,n}(\kappa)$, then they have a winning strategy in the game $\g^{\rho,n}_\alpha(\kappa)$.

Suppose that the challenger has a winning strategy $\sigma$ in the game $\g_\alpha^{\rho,n}(\kappa)$. Let $\tau$ be the following strategy for the challenger in the game $\g^{\theta,n}_\alpha(\kappa)$. Let $M_0\prec H_\rho$ be the first move of the challenger in $\sigma$. The first move of the challenger in $\tau$ is going to be a basic $\kappa$-model $\bar M_0\prec H_\theta$ such that $P^{M_0}(\kappa)\subseteq \bar M_0$. At stage $\gamma$ in the game, $\tau$ needs to respond to the moves $\{\la N_\xi,\in,U_\xi\ra\mid\xi<\gamma\}$ of the judge. Note that as long as we continue to choose $\bar M_\xi$ for $\xi<\gamma$ such that $P^{M_\xi}(\kappa)\subseteq \bar M_\xi$, where $M_\xi$ is the move given by $\sigma$, then any move $\la N_\xi,\in,U\ra$ of the judge in the game $\g_\alpha^{\theta,n}(\kappa)$ will also be a valid move in the game $\g^{\rho,n}_\alpha(\kappa)$. So if $M_\gamma$ is the response of $\sigma$ to the moves $\{\la N_\xi,\in, U_\xi\ra\mid\xi<\gamma\}$ of the judge in the game $\g^{\rho,n}_\alpha(\kappa)$, then $\tau$ will tell the challenger to respond with $\bar M_\gamma\prec H_\theta$ such that $P^{M_\gamma}(\kappa)\subseteq \bar M_\gamma$.
Suppose the judge can win a run of the game $\g^{\theta,n}_{\alpha}(\kappa)$. The run of the game gives models $\la \bar M_\xi, \in,U_\xi\ra$ for $\xi<\alpha$. Let $\bar M=\Union_{\xi<\kappa}\bar M_\xi$ and $U=\Union_{\xi<\kappa}U_\xi$. Since the judge wins, we have $\la H_{\kappa^+}^{\bar M},\in,U\ra\models\ZFC^-_n$. By construction of $\tau$, the models $\bar M_\xi$ were chosen based on the moves $M_\xi$ dictated by $\sigma$. But now it follows that the judge would win against the moves $M_\xi$ by playing $\la N_\xi,\in,U_\xi\ra$ because $H_{\kappa^+}^M=H_{\kappa^+}^{\bar M}=\Union_{\xi<\kappa}N_\xi$.

Next, suppose that the judge has a winning strategy $\sigma$ in the game $\g_\alpha^{\rho,n}(\kappa)$. Let $\tau$ be the following strategy for the judge in the game $\g^{\theta,n}_{\alpha}(\kappa)$. Let $M_0\prec H_\theta$ be the first move of the challenger. Let $\bar M_0\prec H_\rho$ be such that $P^{M_0}(\kappa)\subseteq \bar M_0$. Now let $\tau$ respond with the structure $\la N_0,\in,U_0\ra$ that is the response of $\sigma$ to $\bar M_0$. Given a play $\{M_\xi\mid\xi<\gamma\}$ by the challenger, $\tau$ will tell the judge to respond with the response of $\sigma$ to the play $\la \bar M_\xi\mid\xi<\gamma\}$ such that $\bar M_\xi\prec H_\rho$ and $P^{M_\xi}(\kappa)\subseteq \bar M_\xi$. Now observe that if the challenger wins a play against $\tau$ with the moves $\{ M_\xi\mid\xi<\alpha\}$, then the challenger would also win with the moves $\{\bar M_\xi\mid\xi<\alpha\}$ against $\sigma$ because $H_{\kappa^+}^M=H_{\kappa^+}^{\bar M}=\Union_{\xi<\kappa}N_\xi$ for $M=\Union_{\xi<\alpha}M_\xi$ and $\bar M=\Union_{\xi<\alpha}\bar M_\xi$.
\end{proof}

\begin{theorem}\label{th: char win chal}
A cardinal $\kappa$ is $(\alpha,n)$-baby measurable if and only if the challenger doesn't have a winning strategy in the game $\g^{\theta,n}_\alpha(\kappa)$ for some/all cardinals $\theta>\kappa$. A cardinal $\kappa$ is reflective $(\alpha,n)$-baby measurable if and only if the challenger doesn't have a winning strategy in the game $\rg^{\theta,n}_\alpha(\kappa)$ for some/all cardinals $\theta>\kappa$. An analogous result holds for the faint notions and games.
\end{theorem}
\begin{proof}
Again, we will only prove the result about the $(\alpha,n)$-baby measurable cardinals because the other proof is nearly identical.

Suppose that the challenger doesn't have a winning strategy in the game $\g^{\theta,n}_\alpha(\kappa)$ for some fixed regular cardinal $\theta>\kappa$. Fix $A\subseteq\kappa$. In particular, starting with a basic $\kappa$-model $M_0\prec H_\theta$ with $A\in M_0$ is not a winning strategy for the challenger, and so the judge wins some run of the game, where the challenger starts with such an $M_0\prec H_\theta$. Let $M$ the union of the challenger's moves in this run of the game and $U$ be the union of the ultrafilters played by the judge. Since the game was played for $\alpha$-many steps, the model in each step was closed under ${<}\kappa$-sequences and $\alpha$ is a regular cardinal, the union model $M$ is closed under ${<}\alpha$-sequences. Finally, since the judge wins, we have $\la H_{\kappa^+}^M,\in,U\ra\models\ZFC^-_n$.

In the other direction, suppose that $\kappa$ is $(\alpha,n)$-baby measurable. Suppose towards a contradiction that the challenger has a winning strategy $\sigma$ in the game $\g^{\kappa^+,n}_\alpha(\kappa)$.
It is not hard to see that $\sigma\in H_\theta$ for $\theta=(2^{\kappa})^+$. So fix some basic $\kappa$-model $M\prec H_\theta$ closed under ${<}\alpha$-sequences with $\sigma\in M$ for which there is an $M$-ultrafilter $U$ such that $\la H_{\kappa^+}^M,\in,U\ra\models\ZFC^-_n$. We will use $M$ and $U$ to play against $\sigma$ and win, thereby showing that it couldn't have been a winning strategy. First suppose that $\alpha$ is uncountable.

Let $M_0$ be the first move of the challenger according to $\sigma$, and observe that, by elementarity, we have $M_0\in M$. Let $N_0\in M$ be a simple $\kappa$-model with $M_0\in N_0$ such that every $\Sigma_n$-statement that holds in $H_{\kappa^+}^M$ for an element of $M_0$ has a witness in $N_0$.
Such a model exists since $\ZFC^-_n$ holds in $\la H_{\kappa^+}^M,\in,U\ra$.
Let the judge play $\la N_0,\in,U\cap N_0\ra$.
Since $\la N_0,\in,U\cap N_0\ra\in M$, the response of the challenger according to $\sigma$ must be in $M$ as well.
We continue letting the judge play $\la N_\xi,\in, N_\xi\cap U\ra\in M$ for each $\xi<\alpha$
in this fashion.
Since $M$ is closed under ${<}\alpha$-sequences, $M$ will always have the sequence of the judge's moves at each step $\gamma<\alpha$ of the game.
Thus, the judge can continue to play for $\alpha$-many steps. Let $N=\Union_{\xi<\alpha}N_\xi$.
We have $$\la N_\omega,\in,U\cap N_\omega\ra\prec_{\Sigma_n}\cdots\prec_{\Sigma_n} \la N_\lambda,\in, U\cap N_\lambda\ra\prec_{\Sigma_n}\cdots\prec_{\Sigma_n}\la N,\in, U\cap N\ra\prec_{\Sigma_n} \la H_{\kappa^+}^M,\in,U \ra$$ for limit ordinals $\lambda$. Thus, $\la N,\in, U\cap N\ra$ satisfies $\Sigma_n$-reflection, and therefore $\la N,\in,U\cap N\ra\models\ZFC^-_n$ by Lemma~\ref{lem:ElemSubstructuresImplyZFCn}. As we already observed, $N$ is precisely the $H_{\kappa^+}$ of $M=\Union_{\xi<\alpha}M_\xi$, the union of the moves of the challenger. Thus, we have shown that the judge can win against $\sigma$, contradicting that $\sigma$ was a winning strategy.

Finally, suppose that $\alpha=\omega$, which means that the model $M$ doesn't have any closure. In this case, the judge responds to the moves of the challenger by choosing the models with the partial truth predicates as in the proof of Theorem~\ref{weakly vs faintly nmb} to ensure that the limit model satisfies $\ZFC^-_n$.
\end{proof}

It follows from the proof that $\theta= (2^\kappa)^+$ suffices in the definition of $(\alpha,n)$-baby measurable cardinals.
Thus, $\kappa$ is $(\alpha,n)$-baby measurable if and only if every $A\in H_{(2^\kappa)^+}$ is an element of a ${<}\alpha$-closed imperfect weak $\kappa$-model $M\prec H_{(2^\kappa)^+}$ for which there is an $M$-ultrafilter such that $\la H_{\kappa^+}^M,\in,U\ra\models\ZFC^-_n$. An analogous result holds for reflective $(\alpha,n)$-baby measurable cardinals and the faint notions.

Next, we show that the (reflective) $(\alpha,n)$-baby measurable cardinals form a hierarchy.

\begin{proposition}
\label{th: hierarchy alphanbm}
Suppose that $\alpha<\beta\leq\kappa$ are uncountable regular cardinals. Then every $(\beta,n)$-baby measurable cardinal is a limit of cardinals $\nu>\alpha$ that are $(\alpha,n)$-baby measurable. An analogous result holds for reflective $(\beta,n)$-baby measurable cardinals.
\end{proposition}
\begin{proof}
As above, we will only prove the result about the $(\beta,n)$-baby measurable cardinals because the other proof is nearly identical.

Suppose that $\kappa$ is $(\beta,n)$-baby measurable and fix a ${<}\beta$-closed basic weak $\kappa$-model $M$ for which there is an $M$-ultrafilter $U$ with $\la H_{\kappa^+}^M,\in,U\ra\models\ZFC^-_n$.
Note that we do not use $M\prec H_\theta$ for some $\theta$.
Let $N$ be the ultrapower of $M$ by $U$.  We will argue that $\kappa$ is $(\alpha,n)$-baby measurable in $V_{j(\kappa)}^N$.
Otherwise in $V_{j(\kappa)}^N$, the challenger has a winning strategy in the game $\g^{\kappa^+,n}_\alpha(\kappa)$ by Lemma~\ref{prop:weakWinningStrategyIndependentOfTheta}.
We will use $U$ to play against $\sigma$ and argue that the resulting run of the game is in $N$. Using the same argument as in the proof of Theorem~\ref{th: char win chal}, it suffices to observe that $M$ has all the required sequences by ${<}\beta$-closure. Finally, Lemma~\ref{prop:weakWinningStrategyIndependentOfTheta} shows that $V_\kappa$ must be correct about a cardinal being $(\alpha,n)$-game baby measurable.
\end{proof}


\begin{theorem}
\label{nbm versus alphanbm}
For $n\geq 1$, an $n$-baby measurable cardinal is a limit of cardinals $\nu$ that are
$({<}\nu,n)$-baby measurable. Moreover, an $(\kappa,n)$-baby measurable cardinal $\kappa$ is a limit of $n$-baby measurable cardinals.
\end{theorem}
\begin{proof}
The first part follow from the proof of Proposition~\ref{th: hierarchy alphanbm}.

For the second part, observe that $\kappa$ is at $n$-least baby measurable. Now, fix a $\kappa$-model $M\prec H_{\kappa^+}$ for which there is an $M$-ultrafilter such that $U$ such that $\la M,\in, U\ra\models\ZFC^-_n$, and let $N$ be the ultrapower of $M$ by $U$. We will argue that $\kappa$ is $n$-baby measurable in $N$. Since $n$-baby measurability is verifiable in $H_{\kappa^+}$, $M$ satisfies that $\kappa$ is $n$-baby measurable by elementarity, and thus, so does $N$.
\end{proof}

Since being weakly $n$-baby measurable is a property of $H_{\kappa^+}$, it is easy to see that an $(\omega,n)$-baby measurable cardinal is a weakly $n$-baby measurable limit of weakly $n$-baby measurable cardinals.
\begin{theorem}
A weakly $n$-baby measurable is a limit of faintly $(\omega,n)$-baby measurable cardinals.
\end{theorem}
\begin{proof}
As usual, we suppose that $\sigma$ is a winning strategy of the challenger in the ultrapower of $M$ by $U$ and use $U$ to play against the strategy. We use the construction from the proof of Theorem~\ref{weakly vs faintly nmb}, but modify the tree so that the even stages in the sequences are as above and the odd ones are moves of the challenger.
\end{proof}
\begin{theorem}
\label{[n+1]bm limit of reflective} 
For $n\geq 1$ and $\alpha<\kappa$, a $[n+1]$-baby measurable cardinal $\kappa$ is a limit of reflective $(\alpha,n)$-baby measurable cardinals.
\end{theorem}
\begin{proof}
Fix $A\subseteq\kappa$ and choose a weak $\kappa$-model $M$, with $A\in M$, for which there is an $M$-ultrafilter $U$ such that $\la M,\in,U\ra\models\KP_{n+1}$. By moving to $L[A,U]$ as constructed in $M$, we can assume without loss of generality that $\la M,\in, U\ra$ has a $\Delta_1$-definable bijection $F:\Ord^M\to M$. Let $N$ denote the, not necessarily well-founded, ultrapower of $M$ by $U$. Fix $\alpha<\kappa$. We will argue that $\kappa$ is reflective $(\alpha,n)$-baby measurable in $V_{j(\kappa)}^N$. By Theorem~\ref{th: char win chal}, we only need to show that the challenger doesn't have a winning strategy in the game $\rg^{\kappa^+,n}_\kappa$. So suppose that $\sigma\in V_{j(\kappa)}^N$ is a winning strategy for the challenger in $\rg^{\kappa^+,n}_\kappa$. Note that via coding we can view $\sigma$ as a map from $\mathcal C$ to $\mathcal C$, the collection of codes for elements of $M$, defined in the proof of Lemma~\ref{lem:definableWellOrder}.

As in the proof of Theorem~\ref{th: char win chal}, we will use $U$ to play against $\sigma$. Since $\la M,\in,U\ra\models\KP_{n+1}$, by Lemma~\ref{lem:ZFCnImpliesElemSubstructures}, every set $B\in M$ is contained in a $\kappa$-model $M_B\in M$ such that $$(M_B,\in,U\cap M_B)\prec_{\Sigma_n}\la M,\in,U\ra$$ and $(M_B,\in,U\cap M_B)\models\ZFC^-_n$. Observe that, for a set $B$, the assertion that $\bar M$ is the $F$-least $\kappa$-model of this form is $\Sigma_{n+1}$.

Now using a $\Sigma_{n+1}$-recursion of length $\alpha$ in the structure $\la M,\in, U\ra$, we can construct a sequence of models $\{ N_\xi\mid\xi<\alpha\}\in M$ as follows. Let $M_0$ be the first move of the challenger according to $\sigma$. Let $N_0$ be the $F$-least $\kappa$-model of the form $M_{M_0}$. Given that we have constructed the sequence $\{ \la N_\xi,\in,U\cap N_\xi\ra\mid\xi<\beta\}$ for some $\beta<\alpha$ as the judge's moves, let $M_\beta$ be the response of $\sigma$ to this sequence for the challenger and let $N_\beta$ be the $F$-least $\kappa$-model of the form $M_{M_\beta}$. To ask whether $M_\beta$ is a response of the challenger to the sequence $\{ \la N_\xi,\in, U\cap N_\xi\ra\mid\xi<\beta\}$, we just need to find a code in $\mathcal C$ as a subset of $\kappa$ for the sequence, and then decode the subset of $\kappa$ which is the response of $\sigma$. We use functions representing $\mathcal C$ and $\sigma$ in the ultrapower together with $U$ to determine membership in these sets. Clearly, the sequence $\{ \la N_\xi,\in,U\cap N_\xi\ra\mid \xi<\alpha\}\in M\subseteq N$ is winning for the judge, which contradicts that $\sigma$ was a winning strategy for the challenger in $V_{j(\kappa)}^N$.
\end{proof}

Next, we show that the hierarchy of reflective $(\alpha,n)$-baby measurable cardinals for $n\geq 1$ sits on top of the $(\alpha,n)$-baby measurable cardinals. Note that for $n=0$, the $(\alpha,n)$-baby measurable and the reflective $(\alpha,n)$-baby measurable are the same large cardinal notion because $\Sigma_0$-elementarity is equivalent to being a submodel for transitive structures.

\begin{theorem}
\label{th: omega1 ref bm}
For $n\geq 1$, a faintly reflective $(\omega,n)$-baby measurable cardinal is a limit of cardinals $\alpha$ that are $(\alpha,n)$-baby measurable.
\end{theorem}
\begin{proof}
Suppose that $\kappa$ is $(\omega,n)$-baby measurable. Fix a simple weak $\kappa$-model $M$ for which there is an $M$-ultrafilter $U$ such that for every $A\in M$ there is a $\kappa$-model $\bar M\in M$ such that $\la \bar M,\in,U\cap\bar M\ra\prec_{\Sigma_n}\la M,\in,U\ra$ and $\la \bar M,\in, U\cap \bar M\ra\models\ZFC^-_n$. Suppose that $\kappa$ is not $(\kappa,n)$-baby measurable in the ultrapower $N$ of $M$ by $U$. This means that in $N$, the challenger has a winning strategy $\sigma$ in the game $\g^{\kappa^+,n}_\kappa(\kappa)$.
As in Lemma~\ref{lem:definableWellOrder}, the structure $\la M,{\in,} U\ra$ can check membership in $\sigma$ using some function $s$ representing it in the ultrapower and the checking procedure is $\Sigma_1$. Therefore the models $\la \bar M,\in, U\cap\bar M\ra$, where $s\in \bar M$, are going to be correct about membership in $\sigma$ as well.

So fix some such $\la \bar M,\in, \bar M\cap U\ra$, which we will use to play against $\sigma$. The first move $M_0$ of the challenger must in $\bar M$ by $\Sigma_1$-elementarity. Let $\la N_0,\in U\cap N_0\ra$ be the response of the judge, where $N_0$ is a $\kappa$-model in $\bar M$ that has witnesses for all $\Sigma_n$-assertions true in $\la \bar M,\in, U\cap \bar M\ra$ with parameters in $M_0$ (this exists since $\la \bar M,\in, U\cap \bar M\ra\models\ZFC_n^-$). Since $\la N_0,\in,U\cap N_0\ra\in \bar M$, the next move of the challenger according to $\sigma$ will also be in $\bar M$ by $\Sigma_1$-elementarity, and so we can choose $N_1\in\bar M$ analogously. At limits $\lambda$, the sequence $$\{\la N_\xi,\in, U\cap N_\xi\ra\mid\xi<\lambda\}$$ will be in $\bar M$ by closure, allowing us to analogously choose the next move $N_\lambda$. The $\kappa$-length sequence $\{ \la N_\xi,\in,U\cap N_\xi\ra\mid\xi<\kappa\}$ may not be an element of $\bar M$, but since $\la \bar M,\in,U\cap \bar M\ra$ is an element of $M$, $M$ sees the entire construction and therefore has the $\kappa$-length sequence, which is clearly winning for the judge. Thus, we have reached the desired contradiction showing that $\kappa$ is $(\kappa,n)$-baby measurable in $N$.
\end{proof}
Finally, it is not difficult to see that for $n\geq 1$, a reflective $(\omega,n)$-baby measurable cardinal is a limit of faintly reflective $(\omega,n)$-baby measurable cardinals by constructing a tree of partial plays of the judge against the strategy and arguing that the tree has a branch.

\section{From baby to locally measurable}
\label{section baby loc meas}

The pattern of large cardinal notions around baby measurables is similar to the one around $n$-baby measurables, often with analogous and in some cases simpler proofs.
Weakly baby measurable cardinals are above faintly baby measurable cardinals for a similar reason as in Proposition \ref{weakly vs faintly nmb}.

\begin{proposition}
\label{weakbm lim faintbm}
A weakly baby measurable cardinal $\kappa$ is a limit of faintly baby measurable cardinals.
\end{proposition}
\begin{proof}
Let $M$ be a simple weak $\kappa$-model for which there is a good $M$-ultrafilter $U$ such that $\la M,\in,U\ra\models\ZFC^-$.  Let $N$ be the ultrapower of $M$ by $U$.
Working in $N$, we argue that $\kappa$ is faintly baby measurable.
Fix a subset $A$ of $\kappa$ and consider the tree $T$ whose elements on level $n$ are sequences $\la\la M_0,\in,U_0\ra,\ldots,\la M_{n-1},\in,U_{n-1}\ra\ra$ such that
\begin{itemize}
\item
$A\in M_0$,

\item
$\la M_i,\in,U_i\ra\models\ZFC^-_i$ for all $i< n$,

\item
$ \la M_i,\in,U_i\ra\prec_{\Sigma_i} \la M_j,\in,U_j\ra$ for all $i< j< n$.

\end{itemize}
The tree is ill-founded in $V$ as witnessed by a chain of elementary substructures $$\la M_i,\in,U\cap M_i\ra\prec_{\Sigma_i}\la M,\in, U\ra$$ built using Lemma~\ref{lem:ZFCnImpliesElemSubstructures}.
Thus, $N$ has a branch through $T$ as well.
The union of models on the branch gives a structure $\la\bar M,\in,W\ra\models\ZFC^-$ such that $\bar M$ is a weak $\kappa$-model and $W$ is a $\bar M$-ultrafilter. Thus, $\kappa$ is faintly baby measurable in $N$, and hence $\kappa$ is a limit of these cardinals by elementarity.
\end{proof}

Above these notions, it is easy to see that an $\omega_1$-baby measurable cardinal $\kappa$ is a limit of weakly baby measurable cardinals.
The interval between weakly baby measurable and baby measurable cardinals can be studied by replacing $\ZFC^-_n$ with $\ZFC^-$ and $\Sigma_n$-elementarity with full elementarity in the definitions of the games $\g^{\theta,n}_\alpha(\kappa)$ and $\rg^{\theta,n}_\alpha(\kappa)$ above, resulting in games $\g^\theta_\alpha(\kappa)$ and $\rg^\theta_\alpha(\kappa)$ respectively.
Just like in Lemma \ref{prop:weakWinningStrategyIndependentOfTheta}, one can now show that the existence of a winning strategy in these games for either player is independent of $\theta\geq\kappa^+$.
As in Theorem \ref{th: char win chal}, one can then show that a cardinal $\kappa$ is $\alpha$-baby measurable if and only if the challenger does not have a winning strategy in $\g^\theta_\alpha(\kappa)$, and the analogous characterisation holds for reflective $\alpha$-baby measurable cardinals and $\rg^\theta_\alpha(\kappa)$.
The proofs are virtually the same except that in the construction of a winning run for the judge, in step $\lambda+n$ for limits $\lambda$ one adds witnesses for $\Sigma_n$-truths only.
The $\alpha$-baby measurables and the reflective $\alpha$-baby measurables form strict hierarchies as in Proposition \ref{th: hierarchy alphanbm}.
Moreover, the former hierarchy sits strictly below the latter, since a faintly reflective $\omega$-baby measurable cardinal is a limit of cardinals $\alpha$ that are $\alpha$-baby measurable as in Proposition \ref{th: omega1 ref bm}.
At the top of these hierarchies, a baby measurable cardinal is a limit of cardinals $\nu$ that are ${<}\nu$-baby measurable as in Proposition \ref{nbm versus alphanbm}, similarly to how a strongly Ramsey cardinal $\kappa$ is a limit of cardinals $\nu$ that are ${<}\nu$-Ramsey, and a $\kappa$-baby measurable cardinal $\kappa$ is a limit of baby measurable cardinals, similarly to how a $\kappa$-Ramsey cardinal $\kappa$ is a limit of strongly Ramsey cardinals.

Locally measurable cardinals (see Definition \ref{def: loc meas}) are above all these notions.

\begin{proposition}\label{prop:locallymeasurableBabyMeasurable}
A locally measurable cardinal $\kappa$ is baby measurable and a limit of cardinals $\nu$ that are reflective $\nu$-baby measurable.
\end{proposition}
\begin{proof}
Let $M$ be a weak $\kappa$-model which thinks it has a normal ultrafilter $U$ on $\kappa$. Let $N$ be the ultrapower of $M$ by $U$. If $N$ had a strategy $\sigma$ for the challenger in the game $\rg^{\kappa^+}_\kappa$, then $M$ would see the strategy via the function representing it in the ultrapower and would be able to use $U$ to play against it.

It remains to show that $\kappa$ is baby measurable. Fix $A\subseteq\kappa$ and a weak $\kappa$-model $M$, with $A\in M$, having what it thinks is a normal ultrafilter $U$ on $\kappa$. Clearly, $M$ can build what it thinks is a $\kappa$-model $\bar M$ such that $\la \bar M,\in, U\cap \bar M\ra\prec \la H_{\kappa^+}^M,\in, U\ra$, with $A\in\bar M$, because $H_{\kappa^+}^M$ is a set in $M$. But then $\bar M$ is actually a $\kappa$-model.
\end{proof}

\section{Indestructibility}
\label{section: indestructibility}

In this section, we provide a few basic indestructibility results for faintly baby measurable, weakly baby measurable, and baby measurable cardinals. More specifically, we show that these large cardinals $\kappa$ are indestructible by small forcing and can be made indestructible by the forcing $\Add(\kappa,1)$ adding a Cohen subset to $\kappa$.

The indestructibility arguments will use properties of class forcing over models of second-order set theory. A model of second-order set theory is a triple \hbox{$\mathbf M=\la M,\in,\mathcal C\ra$}, where $M$ consists of the sets of the model and $\mathcal C$ consists of the classes. The second-order theory ${\rm GBc}^-$ consists of the axioms $\ZFC^-$ for sets, the extensionality axiom for classes, the class replacement axiom asserting that every class function restricted to a set is a set, and the first-order comprehension scheme asserting that every first-order formula defines a class. The theory $\GBC^-$ is the theory ${\rm GBc}^-$ together with the assertion that there is a class well-order of sets of order-type $\Ord$. Given a class partial order $\p\in\mathcal C$, we say that a filter $G\subseteq\p$ is $\mathbf M$-\emph{generic} for $\p$ if it meets every dense subclass of $\p$ from $\mathcal C$. Given an $\mathbf M$-generic filter $G$, the forcing extension of $\mathbf M$ by $G$ is the structure $\la M[G],\in,\mathcal C[G]\ra$, where $M[G]$ consists of the interpretation of $\p$-names by $G$ and $\mathcal C[G]$ consists of the interpretation of class $\p$-names by $G$, where a \emph{class $\p$-name} is any class whose elements are pairs of the form $\la \sigma,p\ra$ with $p\in\p$ and $\sigma$ a $\p$-name. If $M$ is ill-founded, we can still form the generic extension by taking instead of interpretations of names, the structure whose elements are equivalence classes of (class) $\p$-names moded out by the filter $G$. Although class forcing may not always preserve replacement to the forcing extension, pretame partial orders preserve the theories ${\rm GBc}^-$ and ${\rm GBC}^-$. See \cite{PeterHolyRegulaKrapfPhilippSchlicht:classforcing2} for the definition and results on pretameness, and see \cite{AntosGitman:ModernClassForcing} for details on models of second-order set theory and class forcing.

\begin{proposition}
Faintly baby measurable, weakly baby measurable, and baby measurable cardinals are indestructible by small forcing.
\end{proposition}
\begin{proof}
Suppose that $\kappa$ is weakly baby measurable, $\p\in V_\kappa$ is a forcing notion, and $g\subseteq \p$ is $V$-generic. Fix $A\subseteq\kappa$ in $V[g]$ and let $\dot A$ be a nice $\p$-name such that $\dot A_g=A$. Since $\p\in V_\kappa$, $\dot A$ is a subset of $V_\kappa$ as well and hence we can put it into a simple weak $\kappa$-model $M$ for which there is a good $M$-ultrafilter $U$ such that $\la M,\in,U\ra\models\ZFC^-$. Let $\mathcal C$ be the classes of $M$ generated by $U$, so that the second-order structure $\la M,\in,\mathcal C\ra\models{\rm GBc}^-$. Since $\p$ is a set forcing in $M$, it is, in particular, trivially pretame. Since pretame forcing preserves ${\rm GBc}^-$, we have that the second-order structure $\la M[g],\in,\mathcal C[g]\ra\models{\rm GBc}^-$ as well. Let $W$ the $M[g]$-ultrafilter generated by $U$ in $M[g]$. Clearly $W\in \mathcal C[g]$, and so it follows that $\la M[g],\in,W\ra\models\ZFC^-$. Also, clearly $W$ is good because we can lift the ultrapower embedding $j:M\to N$ to $j:M[g]\to N[g]$ and the $M[g]$-ultrafilter generated from the lift is precisely $W$. Since $A\in M[g]$, the structure $\la M[g],\in, W\ra$ witnesses weak baby measurability for $A$.

The case of faintly baby measurable cardinals is even easier because it suffices to note that $W$ is definable in $\la M[g],\in,\mathcal C[g]\ra$.

For the case of baby measurable cardinals it suffices to observe that a forcing extension $M[g]$ of a $\kappa$-model $M$ by $\p$ is again a $\kappa$-model in $V[g]$.
\end{proof}

\begin{theorem}
Faintly baby measurable cardinals, weakly baby measurable cardinals, and  baby measurable cardinals $\kappa$ can be made indestructible by the forcing $\Add(\kappa,1)$.
\end{theorem}
\begin{proof}
Suppose that $\kappa$ is weakly baby measurable. Let $\p_\kappa$ be the $\kappa$-length Easton support iteration forcing with $\Add(\alpha,1)$ at regular cardinal stages, and let $G*g$ be $V$-generic for $\p_\kappa*\Add(\kappa,1)$. With some slight renamings, we can assume that the poset $\p_\kappa*\Add(\kappa,1)$ is a subset of $V_\kappa$. Every subset $A\subseteq\kappa$ in $V[G*g]$ has a nice $\p_{\kappa}*\Add(\kappa,1)$-name $\dot A$ in $V$, which can therefore be put into a weak $\kappa$-model $M$ for which there is an $M$-ultrafilter $U$ such that $\mathbf{M}:=\la M,\in,U\ra\models\ZFC^-$. By moving to $L[\dot A,U]$ of $M$, we can assume without loss of generality that $M$ has a definable class bijection $F:\Ord^M\to M$. Let $\mathcal C$ be the classes of $M$ generated by $U$. Since the forcing $\p_\kappa*\Add(\kappa,1)$ is a set and hence pretame, we have that the second-order structure $$\la M[G*g],\in,\mathcal C[G*g]\ra\models\GBC^-.$$ Let $\Ult$ be the (not collapsed) ultrapower of $M$ by $U$ that is definable in \hbox{$\la M,\in,U\ra$} and let $\Psi:M\to \Ult$ be the ultrapower map, which is also definable there. Note that $\la M,\in,U\ra$ can pick out a unique element of each equivalence class using the global well-order function $F$. The model $\mathbf M[G*g]=\la M[G*g],\in,\mathcal C[G*g]\ra$ has the classes $M$ and $\Ult$. Using $G*g$, we can define the model $\Ult[G*g]$ inside $\la M[G*g],\in,\mathcal C[G*g]\ra$. We can think of elements of $\Ult[G*g]$ as equivalence classes $[\tau]$ of $\p_\kappa*\Add(\kappa,1)$-names from $\Ult$, where we have that $\tau$ is equivalent to $\sigma$ whenever there is $p\in G*g$ such that $p\forces\tau=\sigma$. The entire construction is definable in $\mathbf M[G][g]$. We will lift the ultrapower embedding $\Psi$ of $M$ by $U$ to an ultrapower embedding of $M[G*g]$ inside the structure $\la M[G*g],\in,\mathcal C[G*g]\ra$.

First, we lift $\Psi$ to $M[G]$. Using the standard lifting criterion for lifting elementary embeddings to a forcing extension,  we need to build an $\Ult$-generic filter $H$ for the poset $\Psi(\p_\kappa)$ with $\Psi[G]\subseteq H$. The poset $\Psi(\p_\kappa)$ factors as $\p_\kappa*\Add(\kappa,1)*\p_\tail$ (we will associate the initial segment of the ultrapower that is isomorphic to $M$ with $M$ itself to simplify notation). We use $G*g$ for the initial segment of the forcing, thereby trivially satisfying the requirement that $\Psi[G]=G$ will be contained in the filter we end up building. So it remains to find an $\Ult[G*g]$-generic filter $G_\tail$ for the tail forcing $\p_\tail$, which is ${<}\kappa^+$-closed there. The model $\la M[G*g],\in,\mathcal C[G*g]\ra$ has a class bijection $F':\Ord^M\to M[G*g]$ constructed from $F$. Using $F'$, in $\la M[G*g],\in,\mathcal C[G*g]\ra$, we can enumerate all the dense subsets of $\Psi(\p_\kappa)$ in $\Ult[G*g]$ in a class sequence of length $\Ord^M$. The length of every initial segment of this enumeration is some ordinal $\alpha\in M$. Note next that every sequence of elements of
$\Ult$ of order type $\alpha$ must be an element of $\Ult$ because it is an ultrapower. The closure also transfers to the pairs $M[G]$ and $\Ult[G]$, as well as $M[G*g]$ and $\Ult[G*g]$ (for details of closure arguments, see \cite[Section 3]{gitman:ramseyindes}). Thus, as we diagonalize against the sequence of dense sets inside $\la M[G*g],\in,\mathcal C[G*g]\ra$, every initial segment of our choices of elements from the dense sets is going to be a sequence in $\Ult[G*g]$ and therefore since $\Psi(\p_\kappa)$ is ${<}\kappa^+$-closed in $\Ult[G*g]$, we can find an element below the diagonalization sequence. Thus, we can continue using replacement in our structure to define the generic filter. Note that the generic filter $G_\tail$ we construct in this manner is a class in $\mathbf M[G*g]$.

Once we have the generic filter $G_\tail$, we can repeat the process to define an $\Ult[G*g][G_\tail]$-generic filter for the image $\Psi(\Add(\kappa,1))=\Add(\Psi(\kappa),1)$ of $\Add(\kappa,1)$ under the lifted ultrapower map $\Psi$. Thus, we can lift the ultrapower embedding $\Psi$ to $M[G*g]$ and use the ultrapower map $\Psi$ to define the $M[G*g]$-ultrafilter $W$ extending $U$. Since $W$ was defined inside the structure $\la M[G*g],\in,\mathcal C[G*g]\ra$, we have that $\la M[G*g],\in,W\ra\models\ZFC^-$. The lifted ultrapower map verifies that $W$ is good.

The case of faintly baby measurable cardinals proceeds identically to above because we never used the well-foundedness of the ultrapower in our construction.

For the case of baby measurable cardinals it suffices to observe that a forcing extension $M[G][g]$ of a $\kappa$-model $M$ by $\p_\kappa*\Add(\kappa,1)$ is again a $\kappa$-model in $V[G][g]$ (for details, see \cite[Section 3]{gitman:ramseyindes}).
\end{proof}

\begin{corollary}
A faintly baby measurable, weakly baby measurable or baby measurable cardinal $\kappa$ can be made indestructible by the forcing $\Add(\kappa,\theta)$ for any cardinal $\theta$.
\end{corollary}
\begin{proof}
Suppose that $G\subseteq \Add(\kappa,\theta)$ is $V$-generic. Observe that any $A\subseteq\kappa\in V[G]$ has an $\Add(\kappa,\theta)$-name $\dot A$ in $V$ that uses at most $\kappa$-many coordinates in $\theta$. Thus, we can view $\dot A$ as an $\Add(\kappa,1)$-name. Hence, it suffices to show that these cardinals can be made indestructible by $\Add(\kappa,1)$.
\end{proof}

It follows that we can make the $\GCH$ fail at these cardinals.
Also, a simple version of the above lifting argument can be used to show that the $\GCH$ forcing that adds Cohen subsets at all successor cardinals preserves these cardinals.

We believe that the indestructibility results should work for the level by level $n$-versions of the baby measurable cardinals as well (at least for some reasonably large $n$). This would rely on set forcing preserving an appropriate version of the class theory $\GBC^-_n$ extending $\ZFC^-_n$ (or $\KP^-_n$) and checking that the complexity of the constructions (ultrapower and generic filters) never beyond $\Sigma_n$.

\section{A detour into second-order set theory}
\label{section: second-order}

Models of $\ZFC^-$ with a largest cardinal $\kappa$ that is inaccessible are bi-interpretable with models of the second-order set theory Kelley-Morse strengthened by a choice principle for classes. Let $\ZFC^-_I$ be the theory asserting that $\ZFC^-$ holds, that there is a largest cardinal $\kappa$, and that $\kappa$ is inaccessible ($P(\alpha)$ exists and $2^\alpha<\kappa$ for all $\alpha<\kappa$). The assumption that $\kappa$ is inaccessible implies, in particular, that $V_\kappa$ exists, and that $V_\kappa\models\ZFC$. Recall that Kelley-Morse ($\KM$) is a second-order set theory whose axioms consist of $\GBC$ together with the full comprehension scheme asserting for every second-order formula that it defines a class. We can further strengthen $\KM$ by adding various very useful choice principles for classes. Let the \emph{choice scheme} (${\rm CC}$) be the scheme which asserts, for every second-order formula $\varphi(x,X,A)$, that if for every set $x$, there is a class $X$ witnessing $\varphi$, then there is a single class $Y$ collecting witnesses for every set $x$ on its slices $Y_x=\{y\mid \la x,y\ra\in Y\}$.

Marek showed that the theory Kelley-Morse together with the choice scheme ($\KM+\CC$) is bi-interpretable with $\ZFC^-_I$ \cite[Section 2]{Marek:KM}. Given a model $\VV=\la V,\in,\mathcal C\ra\models\KM+\CC$, we obtain the corresponding model $M_{\VV}\models\ZFC^-_I$ by taking all the well-founded extensional relations in $\mathcal C$, modulo isomorphism, with the natural membership relation that results from viewing these relations as transitive sets. We get that $V_\kappa^{M_{\VV}}\cong V$ and $P(V_\kappa)^{M_{\VV}}\cong \mathcal C$. In the other direction, given any model $M\models\ZFC^-_I$, we obtain the corresponding model $$\VV=\la V_\kappa^M,\in,P(V_\kappa)^M\ra\models\KM+\CC.$$ Moreover, what gives us bi-interpretability is that the $\ZFC^-_I$-model $M_{\VV}$ corresponding to $\VV$ is precisely $M$.

Let $\ZFC^-_U$ be the theory in the language with an additional unary predicate $U$ consisting of $\ZFC^-_I$ in the extended language together with the assertion that $U$ is a uniform normal ultrafilter on the largest cardinal $\kappa$.

On the second-order side, let $\KM_U$ be the theory in the language of second-order set theory with an additional unary predicate $U$ on classes consisting of $\KM$ in the extended language together with the assertion that $U$ is a uniform normal ultrafilter on $\Ord$. Let $$\VV=\la V,\in,\mathcal C,U\ra\models\KM_U.$$ Consider the ultrapower structure $\la \mathrm{Ult},{\rm E}\ra$ consisting of the equivalence classes of class functions $F:\Ord\to V$ from $\mathcal C$ modulo $U$. It is not difficult to see that \Los' Theorem holds for the structures $\la V,\in\ra$ and $\la\Ult,{\rm E}\ra$: $\la \Ult,{\rm E}\ra\models\varphi([F])$ if and only if $\{\alpha\mid \la V,\in\ra\models \varphi(F(\alpha)\}\in U$. It follows that the universe $V$ is isomorphic to a rank-initial segment $\Ult_\kappa$ of $\Ult$ consisting of equivalence classes of constant functions $C_a:\Ord\to V$ such that $C_a(\alpha)=a$ for all $\alpha$, and that $\Ult_{\kappa+1}$ consists precisely of the classes $\mathcal C$ via this isomorphism. Since by  \Los' Theorem, $\la \Ult,{\rm E}\ra\models\ZFC$, it has a well-ordering of $\Ult_{\kappa+1}$. Now, essentially by the argument of Lemma~\ref{lem:definableWellOrder}, we can conclude that $\VV$ has a definable well-ordering of its classes. The existence of a definable well-ordering of classes is much stronger than the choice scheme, which clearly follows from it.

Similar arguments as above show that $\KM_U$ and $\ZFC^-_U$ are bi-interpretable. Therefore we should view $\ZFC^-_U$ as essentially being a strong second-order set theory asserting that $\Ord$ is measurable.

From the bi-interpretability of these theories, we also know the first-order consequences in models $\VV=\la V,\in,\mathcal C,U\ra$ of $\KM_U$.
Using the proof of Theorem \ref{th:0bmalpharam} and $\Sigma_n$ class reflection mentioned below, it follows that there is a proper class of cardinals $\kappa$ that are $\kappa$-Ramsey.
Our results further shed light on the global structure of $\VV$.
We say that $\Ord$ is \emph{ineffably Ramsey} if every class function $F:[\Ord]^{{<}\omega}\to 2$ has a stationary homogeneous class. By \emph{$\Sigma_n$ class reflection}, we mean the assertions for each $n\in\omega$ that for every class $A\in\mathcal C$, there is a collection $\bar{\mathcal{C}} \subseteq \mathcal{C}$ coded by a single class such that $A\in\bar{\mathcal C}$ and $\la V,\in,\bar{\mathcal C},U\cap \mathcal{\bar C}\ra\prec_{\Sigma_n} \la V,\in,\mathcal C,U\ra$.
Note that the results cited in the following claims can be applied in this situation since the proofs of Lemma \ref{prop:wellFoundedPartOfUltrapowerWA} and Lemmas \ref{lem:definableWellOrder}, \ref{truth predicate from ultrafilter} and \ref{lem:ZFCnImpliesElemSubstructures} only use internal properties of the structures $\langle M,\in,U\rangle$ and therefore, these results hold for all models $\langle M,\in,U\rangle$ of the respective theory such that $U$ is an $M$-ultrafilter from the viewpoint of $\langle M,\in,U\rangle$.
\begin{itemize}
\item
$\VV$ has a definable global well-order of $\mathcal C$ by Lemma \ref{lem:definableWellOrder}.

\item
$\VV$ has a truth predicate for $\la V,\in,\mathcal C\ra$ by Lemma \ref{truth predicate from ultrafilter}.

\item
$\Ord$ is ineffably Ramsey in $\VV$. (This can be seen by carrying out the argument from \cite[Lemma~3.6]{gitman:ramsey} internally to $\VV$.)

\item
$\VV$ satisfies $\Sigma_n$ class reflection for each $n\in\omega$ by Lemma \ref{lem:ZFCnImpliesElemSubstructures}.
\end{itemize}

We next provide a version of Bovykin's and McKenzie's Theorem \ref{BovykinMcKenzieMain} that uses our variants of their $n$-baby measurable cardinals and the above results about them.

\begin{theorem}
\label{Bovykin equicon}
The following theories are equiconsistent:
\begin{enumerate}
\item
\label{Bovykin equicon 1}
$\ZFC^-_U$.

\item
\label{Bovykin equicon 2}
$\ZFC$ together with the scheme consisting of the assertions for all $n\in \omega$:
\smallskip
\begin{quote}
``There exists an $n$-baby measurable cardinal $\kappa$ with $V_\kappa\prec_{\Sigma_n} V.$''
\end{quote}
\smallskip
\item
\label{Bovykin equicon 3}
$\ZFC$ together with the scheme consisting of the assertions for all $n\in\omega$:
\smallskip
\begin{quote}
``There exists an $n$-baby measurable cardinal.''
\end{quote}
\smallskip
\end{enumerate}
Moreover, \eqref{Bovykin equicon 2} captures precisely the theory of models $M$ of \eqref{Bovykin equicon 1} restricted to $V_\kappa^M$, where $\kappa$ is the largest cardinal of $M$.
\end{theorem}
\begin{proof}
To show that the consistency of \eqref{Bovykin equicon 1} implies that of \eqref{Bovykin equicon 2}, suppose that $\la M,\in, U\ra\models\ZFC^-_U$ with a largest cardinal $\kappa$ (although we use $\in$ for the membership relation here, we don't assume that it is the actual membership relation). We show that \eqref{Bovykin equicon 2} holds in $V_\kappa^M$. Fix $n\in\omega$. The proof of Lemma~\ref{lem:ZFCnImpliesElemSubstructures} shows that every $A\subseteq\kappa$ in $M$ is an element of $\bar M$, which $M$ thinks is a $\kappa$-model, such that $\la \bar M,\in\bar M\cap U\ra\prec_{\Sigma_n}\la M,\in, U\ra$. Thus, the ultrapower $N$ of $M$ by $U$ satisfies that $\kappa$ is $n$-baby measurable. Let's argue that $V_\kappa\prec V_{j(\kappa)}$ in $N$, where $j$ is the ultrapower map. Note that $M$ and $N$ have the same natural numbers (possibly nonstandard), so they agree about formulas. Given a (possibly nonstandard) formula $\varphi(x)$, we also have that $M$ and $N$ agree on whether $V_\kappa\models\varphi(a)$. Moreover, if $M\models``V_\kappa\models\varphi(a)"$, then $N\models ``V_{j(\kappa)}\models\varphi(j(a))"$ by elementarity. It follows, by \Los' Theorem, that there is some $\alpha<\kappa$ such that $M$ satisfies that $V_\alpha\prec V_\kappa$ and $\alpha$ is $n$-baby measurable. Thus, in particular, we actually have that $V_\alpha\prec_{\Sigma_n} V_\kappa$ because $M$ will be correct about satisfaction for standard formulas. Also, $V_\kappa$ clearly agrees with $M$ that $\alpha$ is $n$-baby measurable.

\eqref{Bovykin equicon 2} clearly implies \eqref{Bovykin equicon 3}.

To show that the consistency of \eqref{Bovykin equicon 3} implies that of \eqref{Bovykin equicon 1}, suppose that $N$ is a model  \eqref{Bovykin equicon 3}. Suppose towards a contradiction that there is no model of \eqref{Bovykin equicon 1}, meaning that for some $n<\omega$, the fragment of $\ZFC^-_U$ mentioning only instances of collection and separation for formulas of complexity at most $\Sigma_n$ is inconsistent. Fix such an $n<\omega$. The model $N$ has a model $\la M,\in,U\ra$ satisfying what it thinks is all instances of collection and separation for formulas of complexity at most $\Sigma_n$, and it must be correct about this for standard formulas. Thus, we have reached a contradiction by producing such a model.

It remains to show the ``moreover" part. So suppose that $N$ is a model of \eqref{Bovykin equicon 3}. We need to argue that the theory $\ZFC^-_U$ together with the assertions that $V_\kappa\models\varphi$ for every $\varphi\in \mathrm {Th}(N)$ is consistent. If this were not the case, then there would be some finite fragment $T$ of $\ZFC^-_U$ and some $\varphi$ such that $N\models\varphi$, but there is no model $\la M,\in, U\ra\models T$ such that $V_\kappa^M\models\varphi$. Choose $n<\omega$ bounding the complexity of $\varphi$ and all assertions in $T$. Let $\kappa$ be an $n$-baby measurable cardinal in $N$ such that $V_\kappa^N\prec_{\Sigma_n} N$. Then $N$ has a model $\la M,\in, U\ra$ witnessing that $\kappa$ is $n$-baby measurable. It follows that $\la M,\in, U\ra\models T$ and also $V_\kappa^M=V_\kappa^N\models\varphi$ by $\Sigma_n$-elementarity, contradicting our assumption that this theory is inconsistent.
\end{proof}

The above argument works as well for measurable cardinals.
Let $\ZFC_n$ denote $\ZFC$ with the replacement (equivalently, collection) and separation schemes restricted to $\Sigma_n$-formulas.
We call a cardinal $\kappa$ \emph{$n$-junior measurable} if every $A\subseteq\kappa$ is an element of a $\kappa$-model $M$ of $\ZFC_n$ with a normal ultrafilter $U\in M$.

\begin{remark}
\label{remark theory meas}
Let $T$ denote $\ZFC$ together with the existence of a measurable cardinal.
Let $S$ denote $\ZFC$ with the scheme consisting of the following sentences for all $n\in \omega$:
$$\text{``There exists an $n$-junior measurable cardinal $\kappa$ with $V_\kappa\prec_n V$''}.$$
As in Theorem \ref{Bovykin equicon}, $S$ captures precisely the consequences in $V_\kappa$ of measurable cardinals $\kappa$.
Since the argument only uses $\Pi^1_2$-indescribability, a similar claim holds for all other large cardinal notions that imply $\Pi^1_2$-indescribability, for instance for the notions of strong and supercompact cardinals.
\end{remark}

In particular, there is a hierarchy of natural large cardinal notions, the $n$-junior measurable cardinals, that reaches up all the way to measurable cardinals.
This answers a question of Daniel Isaacson asked at the first-listed author's talk in the Oxford Set Theory Seminar in May 2020.

\section{Outlook}

We provided a fine analysis of large cardinal notions in the interval between Ramsey and measurable cardinals defined by expanding the amount of collection and separation available in the relevant models.
The diagram in Figure \ref{diagram large cardinals} below provides an overview of relationships between the large cardinal notions. 
The patterns around $\alpha$-Ramsey, $(\alpha,n)$-baby measurable and $\alpha$-baby measurable cardinals  are enclosed by solid boxes. 
Repeating steps in a hierarchy that depends on $\omega\leq\alpha<\kappa$ or $1\leq n<\omega$ are enclosed by dashed boxes. 
For example, the large dashed box encloses a pattern that repeats for each $1\leq n<\omega$. 
An $[n+1]$-baby measurable cardinal is a limit of reflective $(\alpha,n)$-baby measurable cardinals by Theorem \ref{[n+1]bm limit of reflective} and such cardinals are $[n]$-baby measurable. 
Note that the notions in this box collapse for $n=0$ by Section \ref{sec:hierarchy}. 
The range of $\alpha$ for $\alpha$-Ramsey cardinals is meant to be $\omega_1\leq\alpha<\kappa$. 

The differences between the properties of $\KP_n$ and $\ZFC^-_n$ studied in Section \ref{section: amen coll} entail that closure properties are not relevant for the large cardinal notions defined via $\KP_n$ by Lemma~\ref{lem:noFaintWeakKP}, while they induce to a strict hierarchy for the ones defined via $\ZFC^-_n$ by Theorem \ref{bm stronger than weaklybm}
that is studied via the games $\g^{\theta,n}_\alpha(\kappa)$ in Section \ref{sec:hierarchy}.

We expect a similar pattern as the one around $\kappa$-baby measurable cardinals to recur at reflective $\kappa$-baby measurable cardinals. 
It is natural to ask whether reflective $(\alpha,n)$-baby measurable cardinals are precisely the cardinals $\kappa$ such that the challenger does not have a winning strategy for the game of length $\kappa\cdot\alpha$.
Some issues are left open for the strong variant of the game $\rg^{\theta,n}_\alpha(\kappa)$ in Definition \ref{def: refl game} where we ask that $M_\gamma\in N_\gamma$ for all $\gamma<\alpha$.
We have a similar characterisation as in Theorem \ref{th: char win chal} for winning strategies for the challenger by modifying the reflection in Definition \ref{def: reflective} to all sets $B\in M$ instead of just subsets of $\kappa$, but it is open whether the existence of a winning strategy for the challenger in this game is independent of $\theta$ as in Lemma \ref{prop:weakWinningStrategyIndependentOfTheta}.

Theorem \ref{Bovykin equicon} provides a bridge between large cardinals in set theory and class theory.
In particular, the theory $S$ in \eqref{Bovykin equicon 3} is interpretable in $T:=\ZFC^-_U$ and $T$ is conservative over $S$.
Since in familiar examples of conservative extensions such as G\"odel-Bernays class theory $\GBC$ and $\ZFC$, any model of $\ZFC$ can be extending to one of $\GBC$ by adding a second-order part, we ask if the same holds here.

\begin{problem}
Is every model of $S$ the restriction to $V_\kappa$ of a model of $T$, where $\kappa$ is the largest cardinal?
In other words, is the function from $\mathrm{Mod}(T)$ to $\mathrm{Mod}(S)$ induced by the interpretation surjective?
\end{problem}

The results in Section \ref{section: second-order} show how to approximate some large cardinal notions from below.
For instance, we studied the precise consequences in $V_\kappa^M$ for a measurable cardinal $\kappa$ in a model $M$ of $\ZFC^-$ in Theorem \ref{Bovykin equicon}.
We further ask whether $n$-baby measurable cardinals can be replaced by $(\kappa,n)$-baby measurable cardinals in this theorem. 
Regarding a finer version of the theorem, let $\ZFC^-_{U,n}$ denote the variant of $\ZFC^-_U$ where the collection and separation schemes are restricted to $\Sigma_n$-formulas 
for some $n\geq 1$ and $\kappa$ denotes the largest cardinal. 
Are the consequences $\ZFC_{U,n}$ in $V_\kappa$ axiomatizable by large cardinal properties?

The same problem is of interest for smaller large cardinals.

\begin{problem}
Is the theory of models  of the form $V_\kappa^M$ axiomatizable, where $M$ is a model of $\ZFC$ and $\kappa$ is a weakly compact cardinal in $M$?
\end{problem}

The approximation of measurable cardinals from below in  Remark \ref{remark theory meas}
suggests to ask whether there is a connection with Bagaria's characterization of the existence of measurable cardinals \cite[Section 5.2]{bagaria2023large}.

The following variant of the above notions for countable models may be connected with properties of sets of reals and the determinacy of infinite games.

\begin{problem}
Consider the statement that every real is contained in a countable transitive model $M$ of $\ZFC^-$ with the largest cardinal $\kappa$ and an $M$-ultrafilter $U$ on $\mathcal{P}(\kappa)^M$ such that $(M,\in,U)$ is an $\omega_1$-iterable model of $\KP_n$, where $n\geq 1$.
Is this equivalent to the determinacy of a natural class of projective sets?
\end{problem}

\newpage

\begin{figure}[H]\label{Figure.LargeCardinals}
  \begin{tikzpicture}[theory/.style={scale=.85,minimum height=6.5mm},scale=.35]
  \draw[dashed, gray]   (0.434,24.413) .. controls (0.434,22.586) and (1.915,21.105) .. (3.741,21.105) -- 
  (30.226,21.105) .. controls (32.053,21.105) and (33.534,22.586) .. (33.534,24.413) -- 
  (33.534,37.398) .. controls (33.534,39.225) and (32.053,40.705) .. (30.226,40.705) -- 
  (3.741,40.705) .. controls (1.915,40.705) and (0.434,39.225) .. (0.434,37.398) -- cycle ;
  \draw[line width=0.1mm, gray]   
  (4.5,9.7) -- (4.5,7.2) .. controls (4.5,5.7) and (6.0,5.7) .. (6.0,5.7) -- 
  (28.0,5.7) .. controls (29.5,5.7) and (29.5,7.2) .. (29.5,7.2) -- 
  (29.5,9.7); 
  \draw[dotted, line width=0.15mm]   
   (29.5,9.7)  -- (4.5,9.7); 
  \draw[line width=0.1mm, gray]  
  (29.5,14.3) --   (29.5,19.2)  .. controls (29.5,20.7) and (28.0,20.7) .. (28.0,20.7) --
  (6.0,20.7)  .. controls (4.5,20.7) and (4.5,19.2) .. (4.5,19.2) -- 
  (4.5,14.3);  
  \draw[dotted, line width=0.15mm] 
  (4.5,14.3) -- (29.5,14.3); 
  \draw[line width=0.1mm, gray]   
  (4.5,29.2) .. controls (4.5,27.8) and (6.0,27.8) .. (6.0,27.8) -- 
  (28.0,27.8) .. controls (29.5,27.8) and (29.5,29.2) .. (29.5,29.2) --  
  (29.5,35.0) .. controls (29.5,36.5) and (28.0,36.5) .. (28.0,36.5) -- 
  (6.0,36.5) .. controls (4.5,36.5) and (4.5,35.0) .. (4.5,35.0) -- cycle; 
  \draw[line width=0.1mm, gray]   
  (4.5,44.3) .. controls (4.5,42.8) and (6.0,42.8) .. (6.0,42.8) -- 
  (28.0,42.8)  .. controls (29.5,42.8) and (29.5,44.3) .. (29.5,44.3) --  
  (29.5,50.5)  .. controls (29.5,52) and (28.0,52) .. (28.0,52) -- 
  (6.0,52) .. controls (4.5,52) and (4.5,50.5) .. (4.5,50.5)  -- cycle; 
  \draw[line width=0.1mm, dashed]   
  (7.9,15.0) .. controls (7.9,14.8) and (8.1,14.8) .. (8.1,14.8) --
  (11.8,14.8) .. controls (12.0,14.8) and (12.0,15.0) .. (12.0,15.0) --  
  (12.0,15.9) .. controls (12.0,16.1) and (11.8,16.1) .. (11.8,16.1) -- 
  (8.1,16.1) .. controls (7.9,16.1) and (7.9,15.9) .. (7.9,15.9) -- cycle; 
  \draw[line width=0.1mm, dashed]   
  (5.8,30.3) .. controls (5.8,30.1) and (6.0,30.1) .. (6.0,30.1) -- 
  (14.0,30.1) .. controls (14.2,30.1) and (14.2,30.2) .. (14.2,30.2)  --  
  (14.2,31.2) .. controls (14.2,31.5) and (14.0,31.5) .. (14.0,31.5) -- 
  (6.0,31.5) .. controls (5.8,31.5) and (5.8,31.3) .. (5.8,31.3) -- cycle; 
  \draw[line width=0.1mm, dashed]   
  (4.1,39.05) .. controls (4.1,38.85) and (4.3,38.85) .. (4.3,38.85) -- 
  (15.7,38.95) .. controls (15.9,38.95) and (15.9,39.15) .. (15.9,39.15) --  
  (15.9,40.1) .. controls (15.9,40.3) and (15.7,40.3) .. (15.7,40.3) -- 
  (4.3,40.2) .. controls (4.1,40.2) and (4.1,40.0) .. (4.1,40.0) -- cycle; 
  \draw[line width=0.1mm, dashed]   
  (6.5,45.65) .. controls (6.5,45.45) and (6.7,45.45) .. (6.7,45.45) -- 
  (13.3,45.45) .. controls (13.5,45.45) and (13.5,45.65) .. (13.5,45.65) --  
  (13.5,46.25) .. controls (13.5,46.95) and (13.3,46.95) .. (13.3,46.95) -- 
  (6.7,46.95) .. controls (6.5,46.95) and (6.5,46.75) .. (6.5,46.75) -- cycle; 
  \draw[line width=0.1mm, dashed]   
  (4.75,54.5) .. controls (4.75,54.3) and (4.95,54.3) .. (4.95,54.3) -- 
  (15.05,54.3) .. controls (15.25,54.3) and (15.25,54.5) .. (15.25,54.5) --  
  (15.25,55.55) .. controls (15.25,55.75) and (15.05,55.75) .. (15.05,55.75) -- 
  (4.95,55.75) .. controls (4.75,55.75) and (4.75,55.55) .. (4.75,55.55) -- cycle; 
  \draw[line width=0.1mm, dashed]   
  (20.3,58.85) .. controls (20.3,58.65) and (20.5,58.65) .. (20.5,58.65) -- 
  (27.5,58.65) .. controls (27.7,58.65) and (27.7,58.85) .. (27.7,58.85) --  
  (27.7,59.95) .. controls (27.7,60.15) and (27.5,60.15) .. (27.5,60.15) -- 
  (20.5,60.15) .. controls (20.3,60.15) and (20.3,59.95) .. (20.3,58.95) -- cycle; 

\newcommand{\attention}{} 

\newcommand{\yy}{2.2} 
     \draw (0:10) node[theory] (1) {}
           ++(90:\yy) node[theory] (2) {}
           ++(90:\yy) node[theory] (3) {completely ineffable = faintly $\omega$-Ramsey}
           ++(90:\yy) node[theory] (4) {} 
           ++(90:\yy) node[theory] (5) {{\attention$\omega$-Ramsey}} 
           ++(90:\yy) node[theory] (6) {}
           ++(90:\yy) node[theory] (7) {}
           ++(90:\yy) node[theory] (8) {$\alpha$-Ramsey}
           ++(90:\yy) node[theory] (9) {} 
           ++(90:\yy) node[theory] (10) {$\kappa$-Ramsey}
           ++(90:\yy) node[theory] (11) {} 
           ++(90:\yy) node[theory] (12) {}
           ++(90:\yy) node[theory, blue!80!black] (13) {{\attention faintly $(\omega,n)$-baby measurable}}
           ++(90:\yy) node[theory, blue!80!black] (14) {}
           ++(90:\yy) node[theory, blue!80!black] (15) { $(\alpha,n)$-baby measurable}
           ++(90:\yy) node[theory, blue!80!black] (16) {}
           ++(90:\yy) node[theory, blue!80!black] (17) {$(\kappa,n)$-baby measurable} 
           ++(90:\yy) node[theory, blue!80!black] (18) {{\attention faintly reflective $(\omega,n)$-baby measurable}}
           ++(90:\yy) node[theory, blue!80!black] (19) {reflective $(\alpha,n)$-baby measurable}
           ++(90:\yy) node[theory, blue!80!black] (20) {}
           ++(90:\yy) node[theory, blue!80!black] (21) {}
           ++(90:\yy) node[theory, blue!80!black] (22) {$\alpha$-baby measurable} 
           ++(90:\yy) node[theory, blue!80!black] (23) {}
           ++(90:\yy) node[theory, blue!80!black] (24) {$\kappa$-baby measurable} 
           ++(90:\yy) node[theory, blue!80!black] (25) {{\attention faintly reflective $\omega$-baby measurable}} 
           ++(90:\yy) node[theory, blue!80!black] (26) {reflective $\alpha$-baby measurable}
           ++(90:\yy) node[theory, blue!80!black] (27) {}
           ++(90:\yy) node[theory, blue!80!black] (28) {}; 
     \draw (0:10) node[theory, blue!80!black] (1a) {};
     \draw (0:24) node[theory] (1a) {weakly ineffable}
           ++(90:\yy) node[theory] (2a) {$0$-iterable}
           ++(90:\yy) node[theory] (3a) {}
           ++(90:\yy) node[theory] (4a) {$1$-iterable}
           ++(90:\yy) node[theory] (5a) {}
           ++(90:\yy) node[theory] (6a) {{\attention $2$-iterable}}
           ++(90:\yy) node[theory] (7a) {Ramsey}
           ++(90:\yy) node[theory] (8a) {}
           ++(90:\yy) node[theory] (9a) {strongly Ramsey} 
           ++(90:\yy) node[theory, blue!80!black] (10a) {}
           ++(90:\yy) node[theory, blue!80!black] (11a) {$[n]$-baby measurable}        
           ++(90:\yy) node[theory, blue!80!black] (12a) {faintly $n$-baby measurable}
           ++(90:\yy) node[theory, blue!80!black] (13a) {}
           ++(90:\yy) node[theory, blue!80!black] (14a) {weakly $n$-baby measurable} 
           ++(90:\yy) node[theory, blue!80!black] (15a) {}
           ++(90:\yy) node[theory, blue!80!black] (16a) {$n$-baby measurable}
           ++(90:\yy) node[theory, blue!80!black] (17a) {}
           ++(90:\yy) node[theory, blue!80!black] (18a) {} 
           ++(90:\yy) node[theory, blue!80!black] (19a) {} 
           ++(90:\yy) node[theory, blue!80!black] (20a) {faintly baby measurable} 
           ++(90:\yy) node[theory, blue!80!black] (21a) {weakly baby measurable}
           ++(90:\yy) node[theory, blue!80!black] (22a) {}
           ++(90:\yy) node[theory, blue!80!black] (23a) {baby measurable}
           ++(90:\yy) node[theory, blue!80!black] (24a) {} 
           ++(90:\yy) node[theory, blue!80!black] (25a) {} 
           ++(90:\yy) node[theory, blue!80!black] (26a) {}
           ++(90:\yy) node[theory] (27a) {locally measurable}
           ++(90:\yy) node[theory, blue!80!black] (28a) {$n$-junior measurable}
           ++(90:\yy) node[theory] (29a) {measurable}; 
     \tikzset{>={Stealth[round, length=1.5mm,width=1.1mm]}}
     \draw[<-]
                        (3) edge (5)
                        (5) edge (8)
                        (8) edge (10); 
     \draw[<-]
                        (10) edge (13)
                        (13) edge (15)
                        (15) edge (17)
                        (17) edge (18)
                        (18) edge (19)
                        (19) edge (22)
                        (22) edge (24) 
                        (24) edge (25) 
                        (25) edge (26);
     \draw[<-]
                        (1a) edge (2a)
                        (2a) edge (4a)
                        (4a) edge (6a)
                        (6a) edge (7a)
                        (7a) edge (9a); 
     \draw[<-]
                        (9a) edge (11a)
                        (11a) edge (12a)
                        (12a) edge (14a)
                        (14a) edge (16a)
                        (16a) edge (20a)
                        (20a) edge (21a)
                        (21a) edge (23a)
                        (23a) edge (27a)
                        (27a) edge (28a)
                        (28a) edge (29a);
\newcommand{\rxx}{179}
\newcommand{\ryy}{-1}
\newcommand{\lxx}{1}
\newcommand{\lyy}{181}
     \draw[<-, dotted]
                        (3) edge [out=\lxx, in=\lyy] ([yshift=-1.5mm]4a.west)
                        (5) edge [out=\lxx, in=\lyy] (6a.west)
                        (8) edge [out=\lxx, in=\lyy] ([yshift=-1.5mm]9a.west); 
     \draw[<-, dotted]
                        (10) edge [out=\lxx, in=\lyy] (11a.west)
                        (13) edge [out=\lxx, in=\lyy]([yshift=-1mm]14a.west)
                        (15) edge [out=\lxx, in=\lyy]([yshift=-1.2mm]16a.west)
                        (22) edge [out=\lxx, in=\lyy] ([yshift=-1.2mm]23a.west); 
     \draw[<-]
                        (2a) edge [out=\rxx, in=\ryy] ([yshift=-0.5mm]3.east)
                        (4a) edge [out=\rxx, in=\ryy] ([yshift=-1.5mm]5.east)
                        (7a) edge [out=\rxx, in=\ryy] ([yshift=-1.5mm]8.east)
                        (9a) edge [out=\rxx, in=\ryy] ([yshift=-1.5mm]10.east); 
     \draw[<-]
                        (12a) edge [out=\rxx, in=\ryy] ([yshift=-0.5mm]13.east)
                        (14a) edge [out=\rxx, in=\ryy] ([yshift=-0.8mm]15.east)
                        (16a) edge [out=\rxx, in=\ryy] (17.east)
     			     (21a) edge [out=\rxx, in=\ryy] ([yshift=-1mm]22.east)
                        (23a) edge [out=\rxx, in=\ryy] (24.east); 
     \draw[<-, dotted]
         		     (19) edge [out=\lxx, in=\lyy] (20a.west)
                        (26) edge [out=\lxx, in=\lyy] (27a.west); 
  \end{tikzpicture}
  \label{diagram large cardinals} 
  \caption{Implications between large cardinal notions. Solid arrows denote direct implications, dotted arrows implications in consistency strength. }
\end{figure}
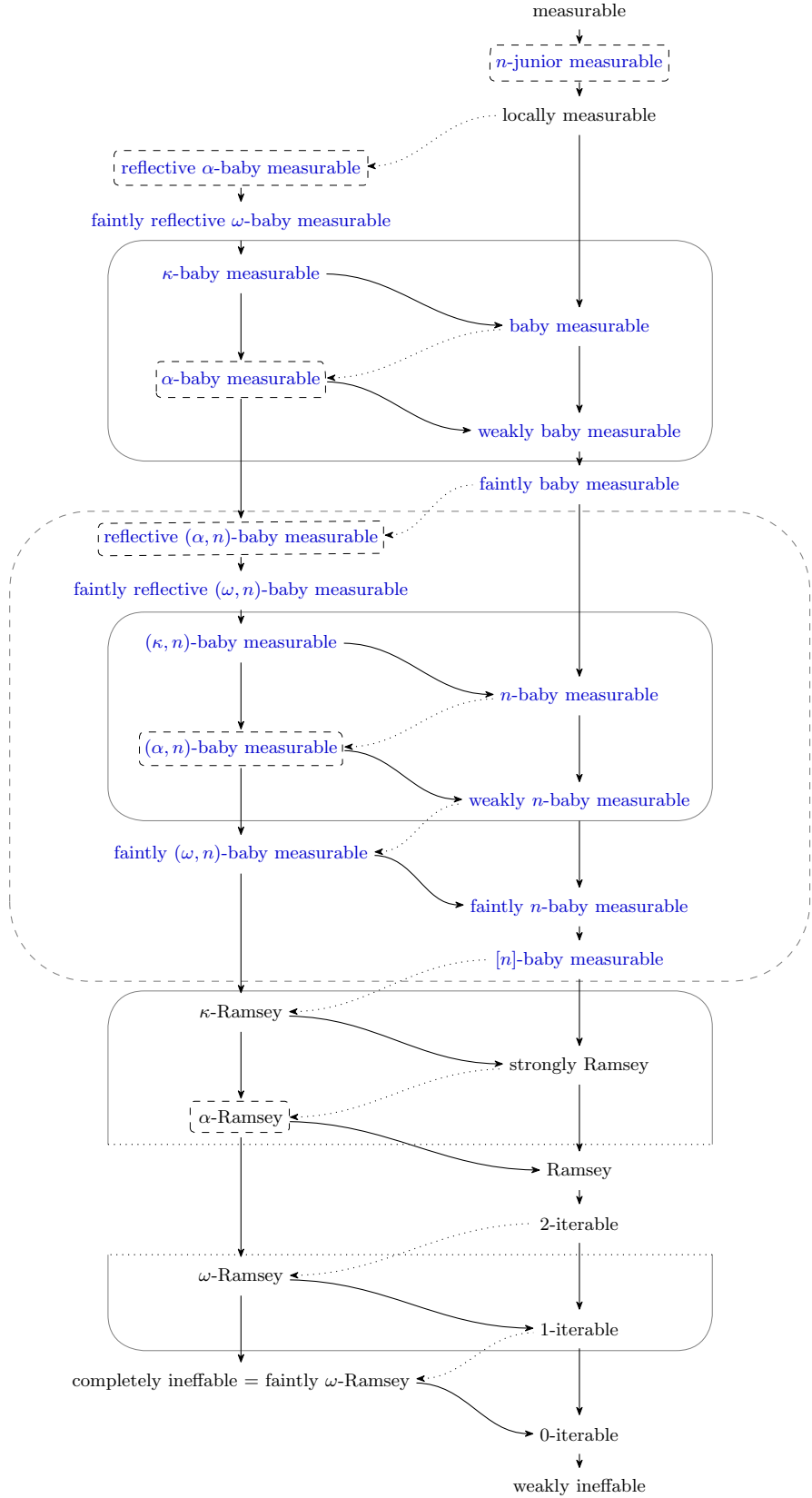


\bibliographystyle{alpha}
\bibliography{bm}

\end{document}